\documentclass[11pt,reqno,tbtags]{amsart}

\usepackage{amsthm,amssymb,amsfonts,amsmath}
\usepackage{amscd} %Commutative diagrams. 
\usepackage[numbers, sort&compress]{natbib} 
\usepackage{subfigure, graphicx}
\usepackage{vmargin}
\usepackage{setspace}
\usepackage{paralist}
\usepackage[pagebackref=true]{hyperref}
\usepackage{xcolor}
\usepackage[normalem]{ulem}

\usepackage{amsaddr}

\usepackage{stmaryrd} %for \bbr
\usepackage[refpage,noprefix,intoc]{nomencl} %for list of notation.
\newcommand{\E}[1]{{\mathbf E}\left[#1\right]}

\newcommand{\p}[1]{{\mathbf P}\left\{#1\right\}}

\newcommand{\I}[1]{{\mathbf 1}_{[#1]}}
\newcommand{\set}[1]{\left\{ #1 \right\}}
\newcommand{\brackets}[1]{\left[ #1 \right]}
\newcommand{\Cprob}[2]{\mathbf{P}\set{\left. #1 \; \right| \; #2}} 

\newcommand{\Cexp}[2]{\mathbf{E}\brackets{\left. #1 \; \right| \; #2}}

\newcommand\cA{\mathcal A}
\newcommand\cB{\mathcal B}

\newcommand\cD{\mathcal D}

\newcommand\cF{\mathcal F}

\newcommand\cR{{\mathcal R}}

\providecommand{\eps}{}
\renewcommand{\eps}{\epsilon}
\providecommand{\ora}[1]{}

\newcommand{\Bad}{\text{Bad}}

\renewcommand{\ora}[1]{\overrightarrow{#1}}
\DeclareRobustCommand{\SkipTocEntry}[5]{} %For table of contents when using AMS styles. Change 5 to 4 if not using hyperref. 
\newtheorem{thm}{Theorem}
\newtheorem{lem}[thm]{Lemma}

\newtheorem{cor}[thm]{Corollary}

\newtheorem{claim}[thm]{Claim}

\numberwithin{equation}{section}
\numberwithin{thm}{section}

\makenomenclature										
\graphicspath{ {./images/} }
\usepackage{wrapfig}
\usepackage{tikz}
\usepackage{standalone}
\newcommand{\cal}{\mathcal}

\begin{document}

\title[The probability that graphs with given degrees are connected]{What is The Probability That A Random Graph With A Given Degree Sequence is  Connected?} 
\author{Louigi Addario-Berry}
\address{Department of Mathematics and Statistics, McGill University, Montr\'eal, Canada}
\email{louigi.addario@mcgill.ca}
\author{Bruce Reed}
\address{Institute of Mathematics, Academia Sinica, Taiwan. Supported by  NSTC Grant 112-2115-M-001 -013 -MY3}
\email{bruce.al.reed@gmail.com}
\author{Dao Chen Yuan}
\address{Department of Mathematics and Statistics, McGill University, Montr\'eal, Canada}
\email{dao.yuan@mcgill.ca}
\date{April 28, 2026} %; revised ...
%\urladdrx{http://problab.ca/louigi/}

%\keywords{<keywords>}
\subjclass{05C80,60C05} 

%{60C05 (68P10,68W40)} %%{Primary: <subject>; Secondary: <subject>}
\begin{abstract} 
An $n$-tuple  ${\mathcal D}=(d(1),\dots,d(n))$ is a {\it feasible degree sequence} 
if there is a  graph on $\{1,\dots,n\}$ such that $i$ has degree $d(i)$. Any such graph will have $m=\sum_{i=1}^n d(i)/2$
edges. Letting $G({\mathcal D})$ be a graph chosen uniformly from those with the 
given degree sequence, we upper-bound the probability that $G({\cal D})$ is  disconnected based on the number  of vertices of degree 
$d$ for small $d$, and develop a powerful tool for proving such bounds.  
If there are any vertices of degree zero the  probability $G$ is disconnected is  $1$, so we assume there are no such vertices.   Our results then imply that if there are $o(\sqrt{m})$ vertices of degree $1$ and $o(m)$ vertices of degree 2 then with high probability $G$ is connected, while if there are no vertices of degree 1 or 2 then the probability $G$ is disconnected is 
$O(\frac{n^4}{m^6})$. 
\end{abstract}

\maketitle

%\tableofcontents

%%%%%%%%%%%%%%%

\section{\bf Introduction}\label{sec:intro} 

An $n$-tuple  ${\mathcal D}=(d(1),\dots,d(n))$ is a {\it feasible degree sequence} 
if there is a  graph on $\{1,\dots,n\}$ such that $i$ has degree $d(i)$. We are interested 
in the probability that the random  graph $G=G({\mathcal D})$ on  $[n]=\{1,\dots,n\}$ chosen uniformly from those with the given degree 
sequence is connected. 

The answer to this question does not depend on the order of the degrees in the sequence as 
any  ordering gives the  same probability of connectivity. Our analysis 
and discussion will assume that the degrees appear in non-decreasing order,  although our theorems are stated without this condition.

For $G$ to be connected there can be no vertex of degree $0$, and we restrict our attention to such 
degree sequences from now on. 
If every vertex has degree 1 then $G$ is not connected; indeed, in this case every component of $G$ has 2 vertices. 
If every vertex of $G$ has degree 2, then $G$ is the disjoint union of cycles. It is well known since at least 1981 \cite{WORMALD1981168} that for such degree sequences,
the probability that $G$ is a single cycle is $o(1)$. On the other hand, as shown by Wormald in the same year \cite{wormald1981asymptotic},
for every fixed $r \ge 3$, if every vertex has degree $r$ then the probability that $G({\mathcal D}_n)$ is 
connected goes to 1 as $n$ goes to infinity. 

The answer to the question that we tackle in this paper \emph{does} depend on the number of vertices of low degree. For every $1\leq i\leq n-1$, we use $n_i=n_i(\cD)$ to denote the number of vertices of degree $i$ in $\cD$. 
We note that every graph with degree sequence ${\mathcal D}$  has $m=m_{\mathcal D}= \frac{1}{2}\sum_{i=1}^n d(i)$ edges. 
If $n_1>m$ then there must be a component which is an edge, so for such degree sequences the probability $G$ is connected is $0$. On the other hand, we prove that if $n_1=o(\sqrt{m})$ and $n_2=o(m)$ then with high probability $G({\mathcal  D})$ is connected.\footnote{We use \emph{with high probability}, or whp, to mean with probability tending to one as $m$ tends to infinity (or equivalently as $n \to \infty$, since we considering simple graphs).} More strongly, we have the following theorem, which is stated in terms of the invariants
\begin{gather*}
    u_{edge}=\frac{\max(n_1-1,0)^2}{m},\,u_{\Delta}=\frac{\max(n_2-2,0)^3}{m^3},\,u_{\Delta+1}=\frac{n_1\max(n_2-1,0)^2n_3}{m^4},\\
    u_{K_4-e}=\frac{\max(n_2-1,0)^2\max(n_3-1,0)^2}{m^5},\,u_{K_4}= \frac{\max(n_3-3,0)^4}{m^6},\,u_{K_5^+}=\frac{n}{m^6}.
\end{gather*}
These invariants bound, in order, the probability of  components which are: an edge, a triangle, a triangle attached to a degree-1 vertex, a clique of size 4 with an edge removed,  a clique of size 4, and a clique of size 5 and other small graphs. 
\begin{thm}
\label{main}
For any feasible degree sequence $\cD=(d(1),\ldots,d(n))$, the probability that $G(\cD)$ is disconnected is 
\[
 O\left(u_{edge}+u_{\Delta}+u_{\Delta+1}+u_{K_4-e}+u_{K_4}+u_{K_5^+}
\right).
\]
\end{thm}

\begin{cor}\label{main-cor}
    If $n_1=o(\sqrt{m})$ and $n_2=o(m)$ then whp $G$ is connected. If $n_1<2$ and $n_2<2$ then the probability $G$ is 
    disconnected is $O(\frac{n_3^4+n}{m^{6}})$. 
\end{cor}

We  note that the upper bounds on $n_1$ and $n_2$ ensuring $G$ is connected whp, are not tight. Indeed $m$ is the required lower bound on the number of vertices of degree $1$ to ensure a random graph with a given degree sequence is disconnected with nonzero probability. 
For example, if ${\cal D}=(d(1),\dots,d(n))$ has $d(n)=n-1$ and $d(i)=1$ for $i<n$, then $G(\cD)$
is always a star and connected. In the same vein, if  $\cD$ has $d(n)=d(n-1)=n-1$ and $d(i)=2$ for $i<n-1$, then $G(\cD)$ is also always connected. However, as we discuss more fully in Section \ref{sub:prev}, for a large class of degree sequences it can be shown that
if $n_1=\Omega(\sqrt{m})$ or $n_2=\Omega(m)$  then whp $G(\cD)$ is disconnected. For example,  this is true if 
$d(n)=o(\sqrt{m})$.

A key tool in proving Theorem \ref{main}, is the following result which is of independent interest.

\begin{thm} 
\label{thekeyresult}
     For any $\gamma>0$,   for any degree sequence 
     satisfying $n_1 \le m^{1-\gamma}, n_2 \le 10^{-6}m$,  the probability that there is more than one component with more than   $4(\log m)^4$ edges  
    is $o(m^{-6 \log \log m})$.
\end{thm}
This theorem allows us to obtain tight bounds on the probability that $G({\cal D})$
is connected simply by determining   the probability that there is a component of size at most 
$4(\log m)^4$.\par 
Theorem \ref{main} is just the tip of the iceberg with respect to possible applications of 
Theorem \ref{thekeyresult} which will  allow the computation of precise bounds on the probability that $G(\cD)$ has up to any fixed number $c$ of components for, e.g., a degree sequence where every vertex has degree at most $r$.

\subsection{Generating  \texorpdfstring{$G({\mathcal D})$}{G(D)} and the Configuration Model}  

We consider a set consisting of $d(i)$ labelled half-edges corresponding to each vertex $i$.
Any matching on the half-edges corresponds to a unique, half-edge labelled multigraph (possibly with loops 
and parallel edges). Each simple graph with degree sequence ${\mathcal D}$ corresponds 
to $\prod_{i=1}^n d(i)!$ matchings. So, if we generate a random matching  ${\cR}={\cR}_{\cD}$ conditional 
on it yielding a multigraph without loops  and parallel edges, then the corresponding random 
graph is $G({\mathcal D})$. Our analysis involves constructing such a matching and hence the graph $G(\cD)$.\par
In the configuration model on multigraphs, we simply generate an unconditioned random matching and take the corresponding multigraph. If we generate the matching one edge at a time then at each step, given the current partial matching, every unmatched half-edge is equally likely to be matched to every other unmatched half-edge, which makes the analysis of this model relatively easy. For example, the expected number of edges components in the configuration model is  exactly $\frac{{\binom{n_1}{2}}}{m-1}$ and a straightforward analysis shows that if $n_1=\omega(\sqrt{m})$ then with high probability there is an edge component
and hence $G(\cD)$ is disconnected.

In proving our results we  exploit the fact that if the degree sequence is “well-behaved”, and the 
partial matching we have constructed is small, then even conditioned on generating a multigraph without loops and parallel edges, every unmatched half-edge is \emph{almost} equally likely to be matched to every other unmatched half-edge which does not create a loop or a parallel edge.  

This is not the case for all degree sequences. For example, if the degree sequence is $\{1,1,..,1,n-1\}$ and has $n$ terms, then the probability that two vertices of degree $1$ are joined is zero. The complication for this example comes from the existence of a vertex of very high degree. Similar problems may arise if the sum of the degrees of the  neighbours of  some vertex $x$, denoted $D(x)$,  is large. We use  $D^* =D^*(\cD)$ to denote $\sum_{i=n-d(n)+1}^n d(i)$ which is an upper bound on $D(x)$ for all $x$ in every  graph with degree sequence $\cD$. We note that if $\Delta =o(\sqrt{m})$, then $D^*=o(m)$.

In order to handle the degree sequences  in our analysis, we use a technique known as {\em switching}, which we now introduce.

\subsection{Switching} 

Switching was introduced by Senior in 1951 \cite{senior1951}, under the name  \emph{transfusions}, to study representative graphs of partitions (see also \cite{BBK1972}). Havel showed in 1955 \cite{havel1955remark} that all graphs on the same degree sequence can be obtained from one another via switchings.  It was first used by McKay in 1981 \cite{mckay1981subgraphs} to study random graphs of fixed degrees and was the key to Wormald’s proof in the same year \cite{wormald1981asymptotic} that almost every $k$-regular graph is connected. We shall use it to both 
\begin{enumerate}[(a)]
\item prove that, if the degree sequence is “well-behaved” and the partial matching we have constructed is small, then every unmatched half-edge is almost equally likely to be matched to every other unmatched half-edge which does not create a loop or a parallel edge, and
\item handle badly behaved degree sequences.
\end{enumerate}

A graph $J$ is obtained from a graph $H$ by a   {\it switching} on a pair of oriented edges 
$xy$ and $uv$ of $H$  if $V(J)=V(H)$ and $E(J)=E(H)-\{xy,uv\} \cup \{xv,uy\}$ (see Figure \ref{fig:switch}).
\begin{figure}
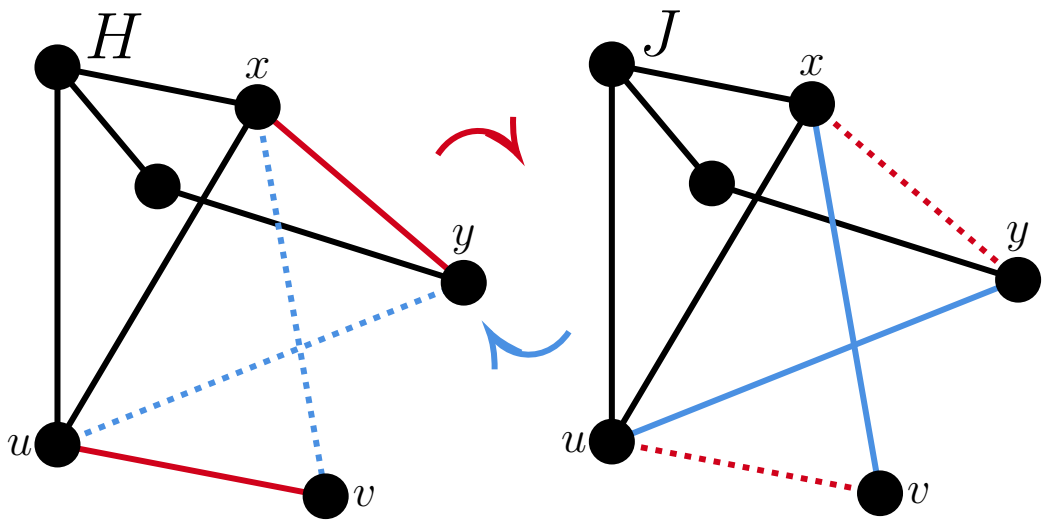

    \centering
\includestandalone{figures/switch-example}
    \caption{Example of a switching at $\{xy,uv\}$ to get edges $\{xv,uy\}$ from graphs $H$ to $J$. In this and all subsequent figures of switchings, edges used in the switching are red and edges used to undo the switching are blue. Dashed lines represent the absence of edges, unless otherwise noted in the caption.}
    \label{fig:switch}
\end{figure}
\begin{figure}
    \centering
\tikzset{every picture/.style={line width=0.75pt}} %set default line width to 0.75pt        

\begin{tikzpicture}[x=0.75pt,y=0.75pt,yscale=-1,xscale=1]
%uncomment if require: \path (0,221); %set diagram left start at 0, and has height of 221

%Straight Lines [id:da9545331945450726] 
\draw [color={rgb, 255:red, 208; green, 2; blue, 27 }  ,draw opacity=1 ][fill={rgb, 255:red, 0; green, 0; blue, 0 }  ,fill opacity=1 ][line width=2.25]    (130,32) -- (536.02,71.61) ;
\draw [shift={(540,72)}, rotate = 185.57] [color={rgb, 255:red, 208; green, 2; blue, 27 }  ,draw opacity=1 ][line width=2.25]    (17.49,-5.26) .. controls (11.12,-2.23) and (5.29,-0.48) .. (0,0) .. controls (5.29,0.48) and (11.12,2.23) .. (17.49,5.26)   ;
%Straight Lines [id:da21137804061530352] 
\draw [color={rgb, 255:red, 208; green, 2; blue, 27 }  ,draw opacity=1 ][fill={rgb, 255:red, 0; green, 0; blue, 0 }  ,fill opacity=1 ][line width=2.25]    (130,32) -- (536,32) ;
\draw [shift={(540,32)}, rotate = 180] [color={rgb, 255:red, 208; green, 2; blue, 27 }  ,draw opacity=1 ][line width=2.25]    (17.49,-5.26) .. controls (11.12,-2.23) and (5.29,-0.48) .. (0,0) .. controls (5.29,0.48) and (11.12,2.23) .. (17.49,5.26)   ;
%Straight Lines [id:da7463657320732184] 
\draw [color={rgb, 255:red, 208; green, 2; blue, 27 }  ,draw opacity=1 ][fill={rgb, 255:red, 0; green, 0; blue, 0 }  ,fill opacity=1 ][line width=2.25]    (130,32) -- (536.07,111.23) ;
\draw [shift={(540,112)}, rotate = 191.04] [color={rgb, 255:red, 208; green, 2; blue, 27 }  ,draw opacity=1 ][line width=2.25]    (17.49,-5.26) .. controls (11.12,-2.23) and (5.29,-0.48) .. (0,0) .. controls (5.29,0.48) and (11.12,2.23) .. (17.49,5.26)   ;
%Straight Lines [id:da3069183405255418] 
\draw [color={rgb, 255:red, 74; green, 144; blue, 226 }  ,draw opacity=1 ][fill={rgb, 255:red, 0; green, 0; blue, 0 }  ,fill opacity=1 ][line width=2.25]    (540,132) -- (133.93,52.77) ;
\draw [shift={(130,52)}, rotate = 11.04] [color={rgb, 255:red, 74; green, 144; blue, 226 }  ,draw opacity=1 ][line width=2.25]    (17.49,-5.26) .. controls (11.12,-2.23) and (5.29,-0.48) .. (0,0) .. controls (5.29,0.48) and (11.12,2.23) .. (17.49,5.26)   ;
%Straight Lines [id:da5518412425858183] 
\draw [color={rgb, 255:red, 74; green, 144; blue, 226 }  ,draw opacity=1 ][fill={rgb, 255:red, 0; green, 0; blue, 0 }  ,fill opacity=1 ][line width=2.25]    (540,132) -- (133.98,92.39) ;
\draw [shift={(130,92)}, rotate = 5.57] [color={rgb, 255:red, 74; green, 144; blue, 226 }  ,draw opacity=1 ][line width=2.25]    (17.49,-5.26) .. controls (11.12,-2.23) and (5.29,-0.48) .. (0,0) .. controls (5.29,0.48) and (11.12,2.23) .. (17.49,5.26)   ;
%Straight Lines [id:da7292591771306338] 
\draw [color={rgb, 255:red, 74; green, 144; blue, 226 }  ,draw opacity=1 ][fill={rgb, 255:red, 0; green, 0; blue, 0 }  ,fill opacity=1 ][line width=2.25]    (540,132) -- (134,132) ;
\draw [shift={(130,132)}, rotate = 360] [color={rgb, 255:red, 74; green, 144; blue, 226 }  ,draw opacity=1 ][line width=2.25]    (17.49,-5.26) .. controls (11.12,-2.23) and (5.29,-0.48) .. (0,0) .. controls (5.29,0.48) and (11.12,2.23) .. (17.49,5.26)   ;
%Straight Lines [id:da6638808092649823] 
\draw [color={rgb, 255:red, 74; green, 144; blue, 226 }  ,draw opacity=1 ][fill={rgb, 255:red, 0; green, 0; blue, 0 }  ,fill opacity=1 ][line width=2.25]    (540,132) -- (133.98,171.61) ;
\draw [shift={(130,172)}, rotate = 354.43] [color={rgb, 255:red, 74; green, 144; blue, 226 }  ,draw opacity=1 ][line width=2.25]    (17.49,-5.26) .. controls (11.12,-2.23) and (5.29,-0.48) .. (0,0) .. controls (5.29,0.48) and (11.12,2.23) .. (17.49,5.26)   ;
%Straight Lines [id:da30585089779492647] 
\draw [color={rgb, 255:red, 74; green, 144; blue, 226 }  ,draw opacity=1 ][fill={rgb, 255:red, 0; green, 0; blue, 0 }  ,fill opacity=1 ][line width=2.25]    (540,132) -- (133.93,211.23) ;
\draw [shift={(130,212)}, rotate = 348.96] [color={rgb, 255:red, 74; green, 144; blue, 226 }  ,draw opacity=1 ][line width=2.25]    (17.49,-5.26) .. controls (11.12,-2.23) and (5.29,-0.48) .. (0,0) .. controls (5.29,0.48) and (11.12,2.23) .. (17.49,5.26)   ;
%Shape: Circle [id:dp30832539139384263] 
\draw  [fill={rgb, 255:red, 0; green, 0; blue, 0 }  ,fill opacity=1 ][line width=0.75]  (120,42) .. controls (120,36.48) and (124.48,32) .. (130,32) .. controls (135.52,32) and (140,36.48) .. (140,42) .. controls (140,47.52) and (135.52,52) .. (130,52) .. controls (124.48,52) and (120,47.52) .. (120,42) -- cycle ;
%Shape: Circle [id:dp9010202256784454] 
\draw  [fill={rgb, 255:red, 0; green, 0; blue, 0 }  ,fill opacity=1 ][line width=0.75]  (120,82) .. controls (120,76.48) and (124.48,72) .. (130,72) .. controls (135.52,72) and (140,76.48) .. (140,82) .. controls (140,87.52) and (135.52,92) .. (130,92) .. controls (124.48,92) and (120,87.52) .. (120,82) -- cycle ;
%Shape: Circle [id:dp22634353654862172] 
\draw  [fill={rgb, 255:red, 0; green, 0; blue, 0 }  ,fill opacity=1 ][line width=0.75]  (120,122) .. controls (120,116.48) and (124.48,112) .. (130,112) .. controls (135.52,112) and (140,116.48) .. (140,122) .. controls (140,127.52) and (135.52,132) .. (130,132) .. controls (124.48,132) and (120,127.52) .. (120,122) -- cycle ;
%Shape: Circle [id:dp0058067523122686815] 
\draw  [fill={rgb, 255:red, 0; green, 0; blue, 0 }  ,fill opacity=1 ][line width=0.75]  (120,162) .. controls (120,156.48) and (124.48,152) .. (130,152) .. controls (135.52,152) and (140,156.48) .. (140,162) .. controls (140,167.52) and (135.52,172) .. (130,172) .. controls (124.48,172) and (120,167.52) .. (120,162) -- cycle ;
%Shape: Circle [id:dp3482324295793858] 
\draw  [fill={rgb, 255:red, 0; green, 0; blue, 0 }  ,fill opacity=1 ][line width=0.75]  (120,202) .. controls (120,196.48) and (124.48,192) .. (130,192) .. controls (135.52,192) and (140,196.48) .. (140,202) .. controls (140,207.52) and (135.52,212) .. (130,212) .. controls (124.48,212) and (120,207.52) .. (120,202) -- cycle ;
%Shape: Circle [id:dp9343188065712338] 
\draw  [fill={rgb, 255:red, 0; green, 0; blue, 0 }  ,fill opacity=1 ][line width=0.75]  (530,42) .. controls (530,36.48) and (534.48,32) .. (540,32) .. controls (545.52,32) and (550,36.48) .. (550,42) .. controls (550,47.52) and (545.52,52) .. (540,52) .. controls (534.48,52) and (530,47.52) .. (530,42) -- cycle ;
%Shape: Circle [id:dp16171879338827777] 
\draw  [fill={rgb, 255:red, 0; green, 0; blue, 0 }  ,fill opacity=1 ][line width=0.75]  (530,82) .. controls (530,76.48) and (534.48,72) .. (540,72) .. controls (545.52,72) and (550,76.48) .. (550,82) .. controls (550,87.52) and (545.52,92) .. (540,92) .. controls (534.48,92) and (530,87.52) .. (530,82) -- cycle ;
%Shape: Circle [id:dp7557633005352111] 
\draw  [fill={rgb, 255:red, 0; green, 0; blue, 0 }  ,fill opacity=1 ][line width=0.75]  (530,122) .. controls (530,116.48) and (534.48,112) .. (540,112) .. controls (545.52,112) and (550,116.48) .. (550,122) .. controls (550,127.52) and (545.52,132) .. (540,132) .. controls (534.48,132) and (530,127.52) .. (530,122) -- cycle ;

% Text Node
\draw (118,122) node [anchor=east] [inner sep=0.75pt]  [font=\Huge]  {$\mathcal{A}$};
% Text Node
\draw (552,82) node [anchor=west] [inner sep=0.75pt]  [font=\Huge]  {$\mathcal{B}$};

\end{tikzpicture}
\caption{If every graph in $\cA$ has at most $\Delta(\cA)$ switchings (in red) into $\cB$, while every graph in $\cB$ has at least $\delta(\cB)$ switchings (in blue) into $\cA$, then $|\cB|\leq\frac{\Delta(\cA)}{\delta(\cB)}|\cA|$.}
    \label{fig:switch-fams}
\end{figure}
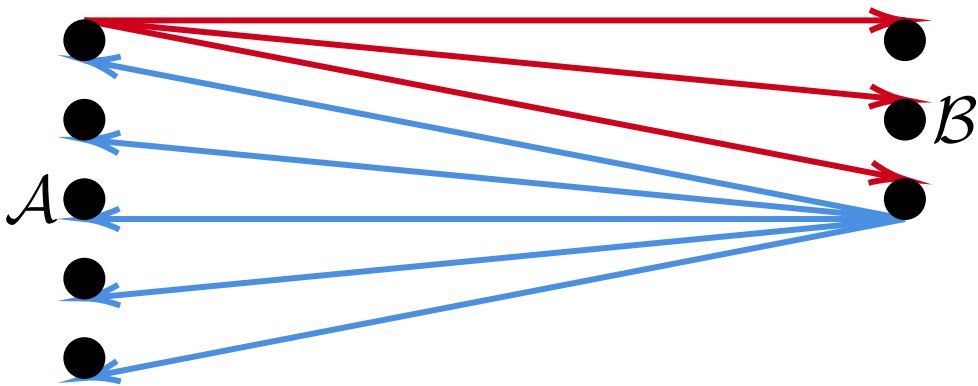
We note that a switching is only possible if $x\neq v, u \neq y$ and $xv, uy \not\in E(G)$. 
Furthermore, if $J$ is obtained from $H$ by  switching on $xy$ and $uv$ then 
it is also obtained from $H$ by switching on $yx$ and $vu$. Thus there are either no or exactly two switchings from $H$ to $J$. Finally, if $J$ is obtained from $H$ by switching on $xy$ and $uv$ 
then $H$ is obtained from $J$ by switching on $xv$ and $uy$. So  there are two switchings from 
$H$ to $J$ if and only if there are two switchings from $J$ to $H$. We also note that if we have a matching on half-edges associated with $H$ and $J$, the switching is also  
a switching  on two oriented edges of that matching.

Given two disjoint families ${\mathcal A}$ and ${\mathcal B}$ of graphs with degree sequence ${\mathcal D}$, 
we can compare their relative sizes by considering switchings between them. Specifically if  for every graph $H$  in ${\mathcal A}$ there are at most  $\Delta(\cA)$ graphs in ${\mathcal B}$ which can be obtained from $H$ by a switching and for every graph $J$  in ${\mathcal B}$ there are at least    $\delta(\cB)$ graphs in ${\mathcal A}$ which can be obtained from $H$ by a switching then, as illustrated in Figure \ref{fig:switch-fams}, $$|\cB|\le \frac{\Delta(\cA)}{\delta(\cB)}|\cA|.$$
It is this observation which allows us to exploit an analysis of switching.

To illustrate its power, especially when degrees are bounded, we prove the following lemma.

\begin{lem}\label{newlem:1/m}
Suppose  that   $N$  is a  matching on    the half-edges of size at most $\frac{m}{8}$, which is extendable to a full matching giving a simple graph,  and that $h_v$ and $h_w$ are half-edges not matched in $N$ corresponding to distinct vertices $v$ and $w$.  Suppose either that $w$ has degree 
at most $\frac{\sqrt{m}}{10}$ or that $D^* \le \frac{m}{100}$. Then the conditional probability, given $N \subseteq {\mathcal R}$, that $h_vh_w$ is an edge of $R$ and that $D(v)<\frac{m}{100}$, is at most $\frac{4}{3m}$. 
\end{lem}
\begin{proof}
    Let $\cF$ denote the event that $N$ is a submatching of ${\cal R}$,   $D(v)<\frac{m}{100}$, and  $h_vh_w\in R_\cD$, Let $\cF'$ denote the event that  $N$ is a submatching of ${\cal R}$, $h_vh_w\notin R_\cD$. We count switchings  between $\cF$ and $\cF'$. There are  at most two switchings from a graph in  $\cF'$ to a graph in $\cF$, using $\{h_vh,h_wh'\}$ (where $h,h'$ are the half-edges matched to  $h_v,h_w$ respectively). To switch from a graph in $\cF$ to one in $\cF'$, it suffices to switch $h_vh_w$ with any edge not in $N$, such that the result is a simple graph. There are at most $\frac{m}{8}$ edges in $N$, at most $\frac{m}{100}$ edges incident to a neighbour of $v$, and at most $\min(\binom{d(w)}{2}, D^*)<\frac{m}{100}$ edges  both of whose  endpoints are neighbours of $w$. This leaves more than $\frac{3m}{4}$ edges which can be  switched with   $h_vh_w$ and hence at least $\frac{3m}{2}$ switchings. So $\p{\cF}\leq\frac{4}{3m}$.
\end{proof}

\subsection{Previous Results}\label{sub:prev}
The connectivity of $G(\cD)$ was first studied by Wormald in 1981 \cite{wormald1981asymptotic}, who proved that if $\cD$ has $d(1)=k \ge3$ and $d(n)\leq C$ constant, then $G(\cD)$ is connected whp. Luczak\cite{luczak1992sparse} showed in 1992 that for any $\cD$ where $d(1)\geq3$ and $d(n)\leq n^{0.01}$, $G(\cD)$ is connected whp, and provided a characterization for connectivity whp when $d(1)=2$ and $d(n)\leq n^{0.01}$. For $d$-regular degree sequences, Cooper, Frieze and Reed \cite{COOPER_FRIEZE_REED_2002} showed in 2002 that they are connected whp for any $3\leq d\leq\eps n$ for a small constant $\eps>0$.\par
Federico and van der Hofstad \cite{federico2017critical} proved in 2017 that for sequences of degree sequences $\{\cD_n\}$ such that for some constant $C$ and every $n$,  $\sum_{i=1}^n d(i)^2<Cn $, we have the following: $\p{G({\cal D}_n)\text{ is connected}}=1-o(1)$ precisely if the number of vertices
of degree $1$ is $o(\sqrt{m})$ and the number of vertices of degree $2$ is $o(m)$. This latter result is a consequence of a more general result  on the configuration model, which translates to the uniform simple graph model when the mean of the square of the degree of a uniformly random vertex is bounded.\par
Joos, Perarnau, Rautenbach, and Reed \cite{joos2018determine}, building on earlier work of Molloy and Reed \cite{molloy1995critical} from 1995, stated in 2018 precise conditions on a sequence of 
degree sequences which ensures that a graph  chosen uniformly from those with the given degree sequence has a giant 
component\footnote{This means that there is some $\epsilon>0$ such that the probability $G$ has a component with  at least $\epsilon n$ vertices 
is $1-o(1)$.}.\par

Gao and Ohapkin \cite{gao2025subgraphprobabilityrandomgraphs} proved in 2023 that if either:
\begin{enumerate}[(a)]
    \item $D^*=o(m)$, or
    \item the total degree of vertices of degree at least $\frac{\sqrt{2m}}{\log{2m}}$ is no more than $(1-\Omega(1))m$,
\end{enumerate}
then if $n_1=o(\sqrt{m})$ and $n_2=o(m)$ then $G(\cD)$ is connected whp. They also proved a partial converse, namely that under assumption (a), for all $c>0$ there is an $\epsilon>0$, such that  if $n_1>c\sqrt{m}$ or $n_2>cm$, then $G(\cD)$ is disconnected with  probability greater than $\epsilon$  for all large $m$.\par
Our results strengthen the results  of Gao and Ohapkin in that we make no assumptions and that we give  upper  bounds 
on the probability that $G(\cD)$ is disconnected rather than simply proving it is $o(1)$. Furthermore, under the following significant weakening of their assumption (a), namely under the assumption (a')  $D^*\le \frac{m}{3}$, we can show (see \ref{sec:concl}) that the  bounds in Theorem \ref{main} are tight up to a constant factor provided $n_1>1$, $n_2>2$ or $n_3=\Omega(n^{1/4})$. We are able to prove our bounds are  tight with a much weaker assumption because of Theorem \ref{thekeyresult}, which allows us to restrict our analysis to determining the probability of the existence 
of very small components.

\section{ A Proof Outline}\label{sec:outline}

\begin{figure}
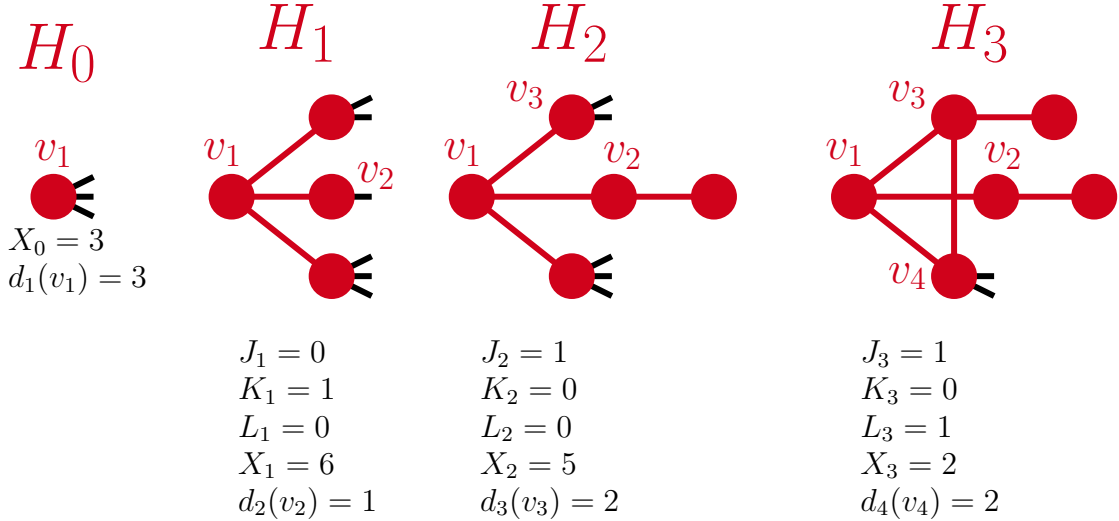

    \centering
    \includestandalone{figures/explore}
    \caption{The starting vertex and the first $3$ iterations of an exploration. The figures illustrate the start of each iteration. The $H_i$ are depicted in red for $i=0,\dots,3$ and the open half-edges are black.}
    \label{fig:explore}
\end{figure}
We  explore the  component containing each  vertex $v$  by  iteratively growing a   tree $T$ rooted at $v$ until we have constructed a spanning tree for the component. Along with $T$  we build  a subgraph $H$ of $G[V(T)]$ and the  submatching $N$  of $\cR_\cD$ corresponding to $H$.  In each iteration, for some vertex $v'$ in the current tree we expose all the matching edges leaving half-edges incident to $v'$ which have not already been exposed.
  We generate $G(\cD)$  and the 
tree at the same time, choosing  the set of edges leaving a vertex so as to generate a uniformly chosen graph
with the given degree sequence conditioned on the existence of the edges which have already been exposed.\par
At the start of iteration $i$ we have exposed 
some subtree $T_{i-1}$ of $T$ and a subgraph $H_{i-1}$ of the subgraph of $G$ induced 
by $T_i$, which corresponds to a submatching $N_{i-1}$ of $N$.  We call a half-edge {\em open} when its endpoint has been revealed to belong to $T$, but the  half-edge  it is adjacent to in $R_\cD$ has not yet been chosen.
We let $X_{i-1}$ be the number of open half-edges at the start  of iteration $i$, and $d_i(v)$ is the number of open half-edges of $v$ for every $v\in[n]$. 
Initially $T_0=H_0$  is simply some fixed vertex $v=v_1$,  and we have $d(v)=d_1(v_1)$ open  half-edges corresponding to $v$, so $X_0=d(v)$. 

In iteration $1 \le i \le m$, provided  $X_{i-1}>0$, we choose a vertex $v_i$ such that $d_i(v_i)$ is minimal among vertices $\{v\in V(H_{i-1}):d_i(v_i)>0\}$, breaking ties by prioritizing vertices with smaller labels, and reveal the edges of $\cR_{\cD}$
containing these half-edges. For each open half-edge $h$  incident to $v_i$ matched to 
a half-edge incident to a vertex $w$, we add $v_iw$ to $H$, and delete $h$ from the set of open half-edges.  If $w$  was not in $H$  we also add $v_iw$ to $T$ and add the $d(w)-1$ half-edges incident  
 to $w$ which were not matched to $h$ to the set of open half-edges. If $w$ is in $T$ we delete the half-edge corresponding to  $w$ matched to $h$   from the set of open half-edges. 
 
Thus, if $w$ is in $T$, the number   of open half-edges decreases by 2,
while if $w$ is not in $T$, the number of open half-edges changes by $d(w)-2$; so it decreases  by 1 if $d(w)=1$,
stays the same if $d(w)=2$, and increases otherwise. 

We will have completely explored the component when $X_i$ drops to $0$. We  treat 
$X_i$ as the position of a random walk which begins at $X_0=d(v)$ and where the $i$'th 
step has size  $X_i-X_{i-1}$. Our focus is on the behaviour of this walk. As is typical, we first compute $\Cexp{X_i-X_{i-1}}{H_{i-1}}$ and then show $X_i$ is concentrated around its conditional expected value.\par
To simplify the analysis of this walk, if there are no more open half-edges, we set
$d_i(v_i)=X_i-X_{i-1}=0$. This allows us to assume the walk takes $m$ steps.\par
We let $J_i$, $K_i$, and $L_i$ be, respectively, the number of half-edges  incident to $v_i$ which, during the $i$'th iteration, were exposed to be matched to a half-edge corresponding to 
a vertex of degree 1, a vertex of degree 2, or a vertex of $H_{i-1}$. Thus
\begin{equation*}
    X_i-X_{i-1} \ge d_i(v_i)-2J_i-K_i-3L_i \ge -2d_i(v_i).
\end{equation*}
The first step in our proof of Theorems \ref{thekeyresult} is to show in Section \ref{sectheyjoin}  that 
the probability there are two components of size $\Omega(m)$ is  $O(m^{-8 \log \log m})$.  
Equivalently  we show that the probability that there are $u$ and $v$  in different 
components such that for both  the exploration from $u$ 
and $v$, at some iteration $i$, we have $X_i+|E(H_i)|=\Omega(m)$ is  $O(m^{-8 \log \log m})$. 
Our approach is guided by the intuition that if there are a large number of edges coming out of two disjoint sets, there will be at least one edge joining them with probability close to 1.
This implies we need only show that if $X_i+|E(H_i)|=\Omega(m)$ for some iteration $i$, then there exists another iteration $1\leq i'\leq i$ where $X_{i'}=\Omega(m)$.

If $n_1,n_2,X_{i-1} \le \frac{m}{10^6}$  and $D^* \le \frac{m}{100}$, then applying Lemma \ref{newlem:1/m}, we obtain that, conditioning on $H_{i-1}$ throughout, each of $J_i$, $K_i$ and $L_i$ has expectation at most 
$\frac{4}{3(10^6)}$ and hence $\E{X_i-X_{i-1}}>(1-4(10^{-6}))d(v_i)$. Thus, the expected increase in the number of open half-edges during the iteration is nearly  as large as the increase in the number of edges of $H_i$, and the expected increase in $X_i$ is nearly half the expected increase in $X_i+|E(H_i)|$.   We show that in this case, the 
walk's position  is reasonably concentrated around its expectation.   This  allows us to show that if $v$ is in a component of size $\Omega(m)$ then  it is very likely that for some $i$, 
$X_i=\Omega(m)$.  

If  $n_1,n_2,X_{i-1} \le \frac{m}{10^6}$ and   $D^* \ge \frac{m}{100}$,  then we are prevented 
from  applying Lemma \ref{newlem:1/m} and  the above approach only because $D(v_i)$ may exceed $\frac{m}{100}$ and the degree $d(w)$ of the other vertex in Lemma \ref{newlem:1/m} may exceed $\frac{\sqrt{m}}{10}$. We handle the second problem by accepting as a positive outcome the existence of a vertex of degree at least $\sqrt{m}(\log{m})^{-2}$ in $T$, because we will show that all such vertices are in the same component whp when $D^*\geq\frac{m}{100}$. So, until this occurs,  both $v_i$ and all its neighbours in $T$ have degree well below $\sqrt{m}$. 
Thus, if $D(v_i)>\frac{m}{100}$, then it is not hard to see $v_i$ must have a neighbour of degree at least $\sqrt{m}>5d(v_i)$
outside $T$. But this implies that $X_i-X_{i-1}>5d_i(v_i)-2-2J_i-K_i-3L_i \ge 2d(v_i)$ which is an even 
better bound than what we needed. This  allows us to show 
that if $v$ is in a component of size $\Omega(m)$ then  it is very likely that  either 
for some $i$, $X_i=\Omega(m)$ or the component containing $v$ contains a vertex of degree 
$\sqrt{m}(\log{m})^{-2}$. 

To complete this first step in our proof of Theorem \ref{thekeyresult}, we show the following, which confirms the above  intuitions:
\begin{enumerate}
\item If $D^*\le \frac{m}{100}$ then for any vertices $u$ and $v$, the probability that $u$ and $v$ are in different components, but the explorations from both $u,v$ reach $X_i=\Omega(m)$ at some point, is $O(m^{- 8 \log \log m})$.
\item If $D^*>\frac{m}{100}$ then the probability that  there exist two vertices of degree
$\sqrt{m}(\log{m})^{-2}$  in different components  is $O(m^{- 8 \log \log m})$.
\item Furthermore, if $D^*>\frac{m}{100}$ then for any vertices $u,v$ where $d(u)\geq\sqrt{m}(\log{m})^{-2}$, the probability that $u,v$ are in different components, and  
the exploration from  $v$ reaches $X_i=\Omega(m)$  at some point, is $O(m^{- 8 \log \log m})$.
\end{enumerate}

In Section \ref{seciteration}, we  provide a deeper analysis  of the behaviour of one iteration, focusing on  bounding the probability that $J_i$, $K_i$ and $L_i$ take specific values via switching arguments. In Section \ref{seclater}, we  first use the results of Section
\ref{seciteration} to bound the probability that there is a component of $G$ whose size exceeds  
 $4(\log m)^4$ but is $o(m)$. We then combine this result with the results of 
 Section \ref{sectheyjoin} to prove Theorem \ref{thekeyresult}. 
Finally in Section \ref{secearly}, we  use the results of Section \ref{seciteration} to analyze the probability of the existence of components
of size less than $4(\log m)^4$ thereby completing the proof of Theorem \ref{main}.

To conclude this section, we explain the reason for our choice of $v_i$
as the vertex of $T$ with open half-edges minimizing the number of such edges. 
This choice of $v_i$ allows us to deterministically bound $L_i$ which helps bound the 
maximum  absolute value of a step. This is useful when proving the position of the walk is 
concentrated around its expected value. Specifically, as Figure \ref{fig:Li} illustrates, by our choice of $v_i$, each of the $L_i$ vertices of $H_{i-1}$ 
to which $v_i$ was directly matched  during the $i$'th iteration was incident  to at least $d_i(v_i) \ge L_i$ 
open half-edges at the start of the iteration, so $X_{i-1}  \ge d_i(v_i)(L_i+1)$. This implies 
that $X_i \ge L_i(L_i-1)$ and that $L_i \le \sqrt{X_{i-1}}$. 
\begin{figure}
    \centering
    \tikzset{every picture/.style={line width=0.75pt}} %set default line width to 0.75pt        

\begin{tikzpicture}[x=0.75pt,y=0.75pt,yscale=-1,xscale=1]
%uncomment if require: \path (0,143); %set diagram left start at 0, and has height of 143

%Straight Lines [id:da052009445676199695] 
\draw [line width=2.25]    (442,70.4) -- (442,100.4) ;
%Straight Lines [id:da5310437452987968] 
\draw [line width=2.25]    (561,70.4) -- (561,100.4) ;
%Straight Lines [id:da8837003835362343] 
\draw [line width=2.25]    (480,70) -- (480,100) ;
%Curve Lines [id:da8279468845000445] 
\draw [color={rgb, 255:red, 208; green, 2; blue, 27 }  ,draw opacity=1 ][line width=2.25]    (604.5,67.23) .. controls (604.5,44.4) and (561,47.57) .. (561,70.4) ;
%Curve Lines [id:da7373880678576613] 
\draw [color={rgb, 255:red, 208; green, 2; blue, 27 }  ,draw opacity=1 ][line width=2.25]    (604.5,67.23) .. controls (604.4,22.73) and (481.65,25.9) .. (482,70.4) ;
%Curve Lines [id:da5283144253306312] 
\draw [color={rgb, 255:red, 208; green, 2; blue, 27 }  ,draw opacity=1 ][line width=2.25]    (604.5,67.23) .. controls (604.15,0.84) and (441.9,3.46) .. (442,70.4) ;
%Straight Lines [id:da5063637138470952] 
\draw [line width=2.25]    (482,70.4) -- (500,100) ;
%Straight Lines [id:da2500054984220863] 
\draw [line width=2.25]    (482,70.4) -- (472,100.4) ;
%Straight Lines [id:da2739093863624953] 
\draw [line width=2.25]    (442,70.4) -- (460,100) ;
%Straight Lines [id:da701837281908749] 
\draw [line width=2.25]    (442,70.4) -- (432,100.4) ;
%Curve Lines [id:da4586578646589333] 
\draw [color={rgb, 255:red, 0; green, 0; blue, 0 }  ,draw opacity=1 ][line width=2.25]    (234,70) .. controls (234,42.6) and (194,42.6) .. (194,70) ;
%Curve Lines [id:da7030705149542309] 
\draw [color={rgb, 255:red, 0; green, 0; blue, 0 }  ,draw opacity=1 ][line width=2.25]    (234,70) .. controls (233.9,16.6) and (113.65,16.6) .. (114,70) ;
%Curve Lines [id:da48821010887624317] 
\draw [color={rgb, 255:red, 0; green, 0; blue, 0 }  ,draw opacity=1 ][line width=2.25]    (234,70) .. controls (233.65,-9.67) and (73.9,-10.33) .. (74,70) ;
%Shape: Circle [id:dp5769828759264135] 
\draw  [color={rgb, 255:red, 208; green, 2; blue, 27 }  ,draw opacity=1 ][fill={rgb, 255:red, 208; green, 2; blue, 27 }  ,fill opacity=1 ][line width=2.25]  (74,60) .. controls (79.52,60) and (84,64.48) .. (84,70) .. controls (84,75.52) and (79.52,80) .. (74,80) .. controls (68.48,80) and (64,75.52) .. (64,70) .. controls (64,64.48) and (68.48,60) .. (74,60) -- cycle ;
%Shape: Circle [id:dp2690440203729305] 
\draw  [color={rgb, 255:red, 208; green, 2; blue, 27 }  ,draw opacity=1 ][fill={rgb, 255:red, 208; green, 2; blue, 27 }  ,fill opacity=1 ][line width=2.25]  (115,60) .. controls (120.52,60) and (125,64.48) .. (125,70) .. controls (125,75.52) and (120.52,80) .. (115,80) .. controls (109.48,80) and (105,75.52) .. (105,70) .. controls (105,64.48) and (109.48,60) .. (115,60) -- cycle ;
%Shape: Circle [id:dp05432550854042295] 
\draw  [color={rgb, 255:red, 208; green, 2; blue, 27 }  ,draw opacity=1 ][fill={rgb, 255:red, 208; green, 2; blue, 27 }  ,fill opacity=1 ][line width=2.25]  (194,60) .. controls (199.52,60) and (204,64.48) .. (204,70) .. controls (204,75.52) and (199.52,80) .. (194,80) .. controls (188.48,80) and (184,75.52) .. (184,70) .. controls (184,64.48) and (188.48,60) .. (194,60) -- cycle ;
%Shape: Circle [id:dp553715031230097] 
\draw  [color={rgb, 255:red, 208; green, 2; blue, 27 }  ,draw opacity=1 ][fill={rgb, 255:red, 208; green, 2; blue, 27 }  ,fill opacity=1 ][line width=2.25]  (234,60) .. controls (239.52,60) and (244,64.48) .. (244,70) .. controls (244,75.52) and (239.52,80) .. (234,80) .. controls (228.48,80) and (224,75.52) .. (224,70) .. controls (224,64.48) and (228.48,60) .. (234,60) -- cycle ;
%Shape: Circle [id:dp6330773952820273] 
\draw  [color={rgb, 255:red, 208; green, 2; blue, 27 }  ,draw opacity=1 ][fill={rgb, 255:red, 208; green, 2; blue, 27 }  ,fill opacity=1 ][line width=2.25]  (442,60.4) .. controls (447.52,60.4) and (452,64.88) .. (452,70.4) .. controls (452,75.92) and (447.52,80.4) .. (442,80.4) .. controls (436.48,80.4) and (432,75.92) .. (432,70.4) .. controls (432,64.88) and (436.48,60.4) .. (442,60.4) -- cycle ;
%Shape: Circle [id:dp23384238428119064] 
\draw  [color={rgb, 255:red, 208; green, 2; blue, 27 }  ,draw opacity=1 ][fill={rgb, 255:red, 208; green, 2; blue, 27 }  ,fill opacity=1 ][line width=2.25]  (482,60.4) .. controls (487.52,60.4) and (492,64.88) .. (492,70.4) .. controls (492,75.92) and (487.52,80.4) .. (482,80.4) .. controls (476.48,80.4) and (472,75.92) .. (472,70.4) .. controls (472,64.88) and (476.48,60.4) .. (482,60.4) -- cycle ;
%Shape: Circle [id:dp9379326068210156] 
\draw  [color={rgb, 255:red, 208; green, 2; blue, 27 }  ,draw opacity=1 ][fill={rgb, 255:red, 208; green, 2; blue, 27 }  ,fill opacity=1 ][line width=2.25]  (602,60.4) .. controls (607.52,60.4) and (612,64.88) .. (612,70.4) .. controls (612,75.92) and (607.52,80.4) .. (602,80.4) .. controls (596.48,80.4) and (592,75.92) .. (592,70.4) .. controls (592,64.88) and (596.48,60.4) .. (602,60.4) -- cycle ;
%Straight Lines [id:da27223696796041097] 
\draw [line width=2.25]    (561,70.4) -- (580,100) ;
%Straight Lines [id:da7442228011267132] 
\draw [line width=2.25]    (561,70.4) -- (551,100.4) ;
%Shape: Circle [id:dp1406923002595577] 
\draw  [color={rgb, 255:red, 208; green, 2; blue, 27 }  ,draw opacity=1 ][fill={rgb, 255:red, 208; green, 2; blue, 27 }  ,fill opacity=1 ][line width=2.25]  (561,60.4) .. controls (566.52,60.4) and (571,64.88) .. (571,70.4) .. controls (571,75.92) and (566.52,80.4) .. (561,80.4) .. controls (555.48,80.4) and (551,75.92) .. (551,70.4) .. controls (551,64.88) and (555.48,60.4) .. (561,60.4) -- cycle ;

% Text Node
\draw (602,83.8) node [anchor=north] [inner sep=0.75pt]  [font=\LARGE,color={rgb, 255:red, 208; green, 2; blue, 27 }  ,opacity=1 ]  {$v_{i}$};
% Text Node
\draw (345.5,37.5) node  [font=\Huge,color={rgb, 255:red, 0; green, 0; blue, 0 }  ,opacity=1 ] [align=left] {$\displaystyle \Longrightarrow $};
% Text Node
\draw (234,83.4) node [anchor=north] [inner sep=0.75pt]  [font=\LARGE,color={rgb, 255:red, 208; green, 2; blue, 27 }  ,opacity=1 ]  {$v_{i}$};
% Text Node
\draw (155,78.26) node [anchor=south] [inner sep=0.75pt]  [font=\LARGE,color={rgb, 255:red, 208; green, 2; blue, 27 }  ,opacity=1 ] [align=left] {$\displaystyle \dotsc $};
% Text Node
\draw (246,70) node [anchor=west] [inner sep=0.75pt]  [font=\Huge,color={rgb, 255:red, 208; green, 2; blue, 27 }  ,opacity=1 ]  {$H_{i-1}$};
% Text Node
\draw (429,72.5) node [anchor=east] [inner sep=0.75pt]  [font=\Huge,color={rgb, 255:red, 208; green, 2; blue, 27 }  ,opacity=1 ]  {$H_{i}$};
% Text Node
\draw (522,78.66) node [anchor=south] [inner sep=0.75pt]  [font=\LARGE,color={rgb, 255:red, 208; green, 2; blue, 27 }  ,opacity=1 ] [align=left] {$\displaystyle \dotsc $};
% Text Node
\draw (462.04,97.7) node [anchor=east] [inner sep=0.75pt]   [align=left] {$\displaystyle \dotsc $};
% Text Node
\draw (502.04,97.7) node [anchor=east] [inner sep=0.75pt]   [align=left] {$\displaystyle \dotsc $};
% Text Node
\draw (581.04,97.7) node [anchor=east] [inner sep=0.75pt]   [align=left] {$\displaystyle \dotsc $};

\end{tikzpicture}
    \caption{An illustration of $X_i\geq L_i(L_i-1)$. The fact that $v_i$ had the least amount of open half-edges forces each other vertex hit by a back edge to have at least $L_i-1$ other open half-edges.}
    \label{fig:Li}
\end{figure}
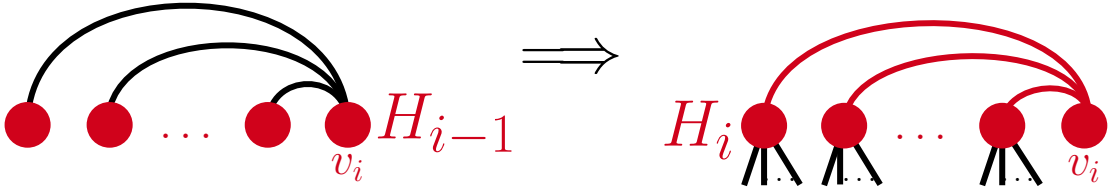

\section{The Vertices in Large Components Are All in The Same Component}
\label{sectheyjoin} 

In this section we discuss one angle to analyze the exploration, which allows us to take our first 
step towards proving Theorem \ref{thekeyresult}, by proving the following:
\begin{lem}
    \label{theyjoinmain}
    If $m$ is large enough and $n_1,n_2 \le \frac{m}{10^6}$, then for sufficiently small $\epsilon$,  the probability there 
    are two components each  of which has at least $\epsilon m$ edges  is at most  $m^{-8 \log \log m}$.  
\end{lem}

We start by proving a key lemma, Lemma \ref{lem:mlogm-same-comp}, which shows that any two vertices $u,v$ of degree at least $(\log{m})^4$ and where $d(u)d(v)\geq m(\log{m})^4$ are in the same component with high probability. Because the switching used in the proof, as depicted in Figure \ref{fig:3.3}, requires $u$ to have many non-leaf neighbours, we first prove Lemma \ref{lem:non-leaf-neighbour}, which states that no single vertex with $\Omega((\log{m})^3)$ degree will have almost all neighbours be leaves.

\begin{lem}\label{lem:non-leaf-neighbour}
    Let $\cD$ be such that $m$ is large enough and $n_1\leq  n(1-\log^{-1}{m})$. Fix $u\in[n]$ with $d(u)\geq100(\log{m})^3$. Then the probability that $u$ has no more than $\frac{d(u)}{2(\log{m})^2}$ neighbours of degree at least $2$ is at most $m^{-49\log\log{m}}$.
\end{lem}
\begin{proof}
    For every $0\leq i\leq d(u)$ let $\cF_i$ denote the event that $u$ has $i$ neighbours of degree at least $2$. We consider switchings from $\cF_i$ into $\cF_{i+1}$ which swap $\{ua,bv\}$ with $\{uv,ba\}$, where $d(a)=1$, $v$ is a non-neighbour of $u$ of degree at least $2$, and $b$ is a neighbour of $v$, as in Figure \ref{fig:3.2}. Such a switching is always valid.\par
\begin{figure}
    \centering
    \tikzset{every picture/.style={line width=0.75pt}} %set default line width to 0.75pt        

\begin{tikzpicture}[x=0.75pt,y=0.75pt,yscale=-1,xscale=1]
%uncomment if require: \path (0,214); %set diagram left start at 0, and has height of 214

%Straight Lines [id:da7679044652423737] 
\draw [color={rgb, 255:red, 208; green, 2; blue, 27 }  ,draw opacity=1 ][line width=2.25]    (80,50) -- (80,102) -- (80,150) ;
%Straight Lines [id:da18605012011692024] 
\draw [color={rgb, 255:red, 74; green, 144; blue, 226 }  ,draw opacity=1 ][line width=2.25]  [dash pattern={on 2.53pt off 3.02pt}]  (80,50) -- (240,150) ;
%Straight Lines [id:da4371013298446975] 
\draw [color={rgb, 255:red, 208; green, 2; blue, 27 }  ,draw opacity=1 ][line width=2.25]    (240,50) -- (240,107) -- (240,150) ;
%Curve Lines [id:da8126349070766706] 
\draw [color={rgb, 255:red, 208; green, 2; blue, 27 }  ,draw opacity=1 ][line width=2.25]    (260,67.62) .. controls (281.18,38.67) and (343.48,37.19) .. (367.04,64.98) ;
\draw [shift={(369.45,68.12)}, rotate = 235.39] [color={rgb, 255:red, 208; green, 2; blue, 27 }  ,draw opacity=1 ][line width=2.25]    (17.49,-7.84) .. controls (11.12,-3.68) and (5.29,-1.07) .. (0,0) .. controls (5.29,1.07) and (11.12,3.68) .. (17.49,7.84)   ;
%Curve Lines [id:da8624717965909442] 
\draw [color={rgb, 255:red, 74; green, 144; blue, 226 }  ,draw opacity=1 ][line width=2.25]    (263.35,141.02) .. controls (287.39,167.12) and (348.67,166.44) .. (370,138.12) ;
\draw [shift={(260.55,137.62)}, rotate = 53.7] [color={rgb, 255:red, 74; green, 144; blue, 226 }  ,draw opacity=1 ][line width=2.25]    (17.49,-7.84) .. controls (11.12,-3.68) and (5.29,-1.07) .. (0,0) .. controls (5.29,1.07) and (11.12,3.68) .. (17.49,7.84)   ;
%Shape: Circle [id:dp11884265334764177] 
\draw  [color={rgb, 255:red, 0; green, 0; blue, 0 }  ,draw opacity=1 ][fill={rgb, 255:red, 0; green, 0; blue, 0 }  ,fill opacity=1 ][line width=2.25]  (80,40) .. controls (85.52,40) and (90,44.48) .. (90,50) .. controls (90,55.52) and (85.52,60) .. (80,60) .. controls (74.48,60) and (70,55.52) .. (70,50) .. controls (70,44.48) and (74.48,40) .. (80,40) -- cycle ;
%Shape: Circle [id:dp9339668109838892] 
\draw  [color={rgb, 255:red, 0; green, 0; blue, 0 }  ,draw opacity=1 ][fill={rgb, 255:red, 0; green, 0; blue, 0 }  ,fill opacity=1 ][line width=2.25]  (80,140) .. controls (85.52,140) and (90,144.48) .. (90,150) .. controls (90,155.52) and (85.52,160) .. (80,160) .. controls (74.48,160) and (70,155.52) .. (70,150) .. controls (70,144.48) and (74.48,140) .. (80,140) -- cycle ;
%Shape: Circle [id:dp1497156351101191] 
\draw  [color={rgb, 255:red, 0; green, 0; blue, 0 }  ,draw opacity=1 ][fill={rgb, 255:red, 0; green, 0; blue, 0 }  ,fill opacity=1 ][line width=2.25]  (240,40) .. controls (245.52,40) and (250,44.48) .. (250,50) .. controls (250,55.52) and (245.52,60) .. (240,60) .. controls (234.48,60) and (230,55.52) .. (230,50) .. controls (230,44.48) and (234.48,40) .. (240,40) -- cycle ;
%Shape: Circle [id:dp973341145275968] 
\draw  [color={rgb, 255:red, 0; green, 0; blue, 0 }  ,draw opacity=1 ][fill={rgb, 255:red, 0; green, 0; blue, 0 }  ,fill opacity=1 ][line width=2.25]  (240,140) .. controls (245.52,140) and (250,144.48) .. (250,150) .. controls (250,155.52) and (245.52,160) .. (240,160) .. controls (234.48,160) and (230,155.52) .. (230,150) .. controls (230,144.48) and (234.48,140) .. (240,140) -- cycle ;
%Straight Lines [id:da9130051346517337] 
\draw [color={rgb, 255:red, 0; green, 0; blue, 0 }  ,draw opacity=1 ][line width=2.25]  [dash pattern={on 2.53pt off 3.02pt}]  (80,150) -- (130,160) ;
%Straight Lines [id:da2592766171212829] 
\draw [color={rgb, 255:red, 0; green, 0; blue, 0 }  ,draw opacity=1 ][line width=2.25]    (240,150) -- (200,160) ;
%Straight Lines [id:da9295789519145127] 
\draw [color={rgb, 255:red, 74; green, 144; blue, 226 }  ,draw opacity=1 ][line width=2.25]    (400,49.74) -- (560,149.74) ;
%Shape: Circle [id:dp19678369097951665] 
\draw  [color={rgb, 255:red, 0; green, 0; blue, 0 }  ,draw opacity=1 ][fill={rgb, 255:red, 0; green, 0; blue, 0 }  ,fill opacity=1 ][line width=2.25]  (400,39.74) .. controls (405.52,39.74) and (410,44.21) .. (410,49.74) .. controls (410,55.26) and (405.52,59.74) .. (400,59.74) .. controls (394.48,59.74) and (390,55.26) .. (390,49.74) .. controls (390,44.21) and (394.48,39.74) .. (400,39.74) -- cycle ;
%Shape: Circle [id:dp9631633016084361] 
\draw  [color={rgb, 255:red, 0; green, 0; blue, 0 }  ,draw opacity=1 ][fill={rgb, 255:red, 0; green, 0; blue, 0 }  ,fill opacity=1 ][line width=2.25]  (560,139.74) .. controls (565.52,139.74) and (570,144.21) .. (570,149.74) .. controls (570,155.26) and (565.52,159.74) .. (560,159.74) .. controls (554.48,159.74) and (550,155.26) .. (550,149.74) .. controls (550,144.21) and (554.48,139.74) .. (560,139.74) -- cycle ;
%Straight Lines [id:da29406028394256056] 
\draw [color={rgb, 255:red, 0; green, 0; blue, 0 }  ,draw opacity=1 ][line width=2.25]  [dash pattern={on 2.53pt off 3.02pt}]  (400,149.74) -- (450,159.74) ;
%Straight Lines [id:da3262279046209834] 
\draw [color={rgb, 255:red, 0; green, 0; blue, 0 }  ,draw opacity=1 ][line width=2.25]    (560,149.74) -- (520,159.74) ;
%Straight Lines [id:da2978590365145354] 
\draw [color={rgb, 255:red, 74; green, 144; blue, 226 }  ,draw opacity=1 ][line width=2.25]    (560,49.74) -- (400,149.74) ;
%Shape: Circle [id:dp8515058264886011] 
\draw  [color={rgb, 255:red, 0; green, 0; blue, 0 }  ,draw opacity=1 ][fill={rgb, 255:red, 0; green, 0; blue, 0 }  ,fill opacity=1 ][line width=2.25]  (400,139.74) .. controls (405.52,139.74) and (410,144.21) .. (410,149.74) .. controls (410,155.26) and (405.52,159.74) .. (400,159.74) .. controls (394.48,159.74) and (390,155.26) .. (390,149.74) .. controls (390,144.21) and (394.48,139.74) .. (400,139.74) -- cycle ;
%Shape: Circle [id:dp15276628569732853] 
\draw  [color={rgb, 255:red, 0; green, 0; blue, 0 }  ,draw opacity=1 ][fill={rgb, 255:red, 0; green, 0; blue, 0 }  ,fill opacity=1 ][line width=2.25]  (560,39.74) .. controls (565.52,39.74) and (570,44.21) .. (570,49.74) .. controls (570,55.26) and (565.52,59.74) .. (560,59.74) .. controls (554.48,59.74) and (550,55.26) .. (550,49.74) .. controls (550,44.21) and (554.48,39.74) .. (560,39.74) -- cycle ;

% Text Node
\draw (80,38.26) node [anchor=south] [inner sep=0.75pt]  [font=\LARGE] [align=left] {$\displaystyle u$};
% Text Node
\draw (80.5,161.74) node [anchor=north] [inner sep=0.75pt]  [font=\LARGE] [align=left] {$\displaystyle a$};
% Text Node
\draw (239,161.74) node [anchor=north] [inner sep=0.75pt]  [font=\LARGE] [align=left] {$\displaystyle v$};
% Text Node
\draw (238,37.26) node [anchor=south] [inner sep=0.75pt]  [font=\LARGE] [align=left] {$\displaystyle b$};
% Text Node
\draw (400,38) node [anchor=south] [inner sep=0.75pt]  [font=\LARGE] [align=left] {$\displaystyle u$};
% Text Node
\draw (400.5,161.48) node [anchor=north] [inner sep=0.75pt]  [font=\LARGE] [align=left] {$\displaystyle a$};
% Text Node
\draw (559,161.48) node [anchor=north] [inner sep=0.75pt]  [font=\LARGE] [align=left] {$\displaystyle v$};
% Text Node
\draw (558,37) node [anchor=south] [inner sep=0.75pt]  [font=\LARGE] [align=left] {$\displaystyle b$};

\end{tikzpicture}
    \caption{Switching used in the proof of Lemma \ref{lem:non-leaf-neighbour}: we switch away degree $1$ neighbours of $u$ (vertex $a$) one-by-one for non-neighbours of $u$ of degree $2$ or more (vertex $v$).}
    \label{fig:3.2}
\end{figure}
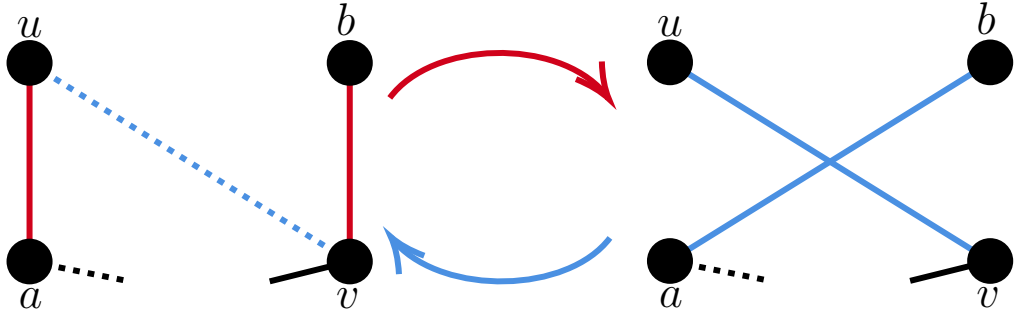
    In $\cF_i$, there are $d(u)-i$ choices of $ua$ and $n-i-n_1$ choices of $v$ (each giving at least one choice of $bv$), so there are at least $(d(u)-i)(n-i-n_1)$ switchings from $\cF_i$ into $\cF_{i+1}$. In $\cF_{i+1}$, there are at most $i+1$ choices of $uv$ and at most $n_1$ choices of $ba$, so at most $n_1(i+1)$ switchings from $\cF_{i+1}$ back into $\cF_i$. Thus for any $0\leq i<\frac{d(u)}{(\log{m})^2}-1$, we have
\begin{equation*}
    \p{\cF_i}\leq\p{\cF_{i+1}}\frac{n_1(i+1)}{(d(u)-i)(n-i-n_1)}\leq\p{\cF_{i+1}}\frac{n\frac{d(u)}{(\log{m})^2}}{\frac{d(u)}{2}\cdot\frac{n}{2\log{m}}}\leq\frac{4}{\log{m}}.
\end{equation*}
So for any fixed $0\leq K\leq\frac{d(u)}{2(\log{m})^2}$, $\p{\cF_K}\leq(\frac{4}{\log{m}})^{50\log{m}}$ and we union bound over the choices of $K$.
\end{proof}

\begin{lem}\label{lem:mlogm-same-comp}
    Let $\cD$ be such that $m$ is large enough and $n_1 \le n(1-\log{m})$. Fix $u,v\in[n]$ with $d(u),d(v)\geq(\log{m})^4$ and such that $d(u)d(v)\geq m(\log{m})^4$. Then the probability that $u$ and $v$ are not in the same component is at most $m^{-9\log\log{m}}$.
\end{lem}
\begin{proof}
Wlog $u$ has $d(u)\geq\sqrt{m}(\log{m})^2$. Let $\cA$ denote the event that $u$ has at least $\frac{d(u)}{(\log{m})^2}$ neighbours of degree at least $2$. Let $\cB=\cB_0$ denote the event that $u$ and $v$ are not in the same component. For $i\geq0$, let $\cB_{i+1}$ denote the event where $u$ and $v$ have $i$ common neighbours, and in the graph obtained by deleting these common neighbours and at most $i$ other neighbours of $u$, there is no path of at most four edges from $u$ to $v$. Let $\cF_i=\cA\cap\cB_i$.

Fix $i\geq0$. For every $G$ in $\cF_i$ we denote by $S=S(G)$ the set of at most $i$ neighbours of $u$ such that there is no path of length at most $4$ between $u$ and $v$ in $G-S\cup(N(u)\cap N(v))$. We claim we can  switch  from a graph $G$ in ${\mathcal F}_i$ to a graph $G'$ by swapping $\{xa,yv\}$ with  $\{xv,ya\}$ whenever  $x\in N(u)\setminus (N(v) \cup S)$, $a\neq u$ and $y \in N(v)\setminus N(u)$, as in Figure \ref{fig:3.3}. 
Note that since $x\notin S$ and there is no path of length at most $4$ between $u$ and $v$ in $G-S\cup(N(u)\cap N(v))$, $ya \not \in E(G)$ and our claim is true. In $G'$ there are  $i+1$ common neighbours of  $u$ and $v$. Any induced path of length at most $4$ between $u$ and $v$ either goes through $N(u)\cap N(v)$, or through both $N(u)\setminus N(v)$ and $N(v)\setminus N(u)$. Since no new vertices were added to $N(u)\setminus N(v)$ or $N(v)\setminus N(u)$ in $G'$ from $G$, every path of at most 4 edges  from $u$ to $v$ in $G-S\cup(N(u)\cap N(v))$ must use $ya$, and since none of  $yv$, $yu$ or $av$  is an edge of $G'$, the only possible such path is $uaybv$ for some common neighbour $b$ of $y$ and $v$. 
Adding $a$ to $S$ and noting that $N(u)$ does not change, we see that indeed $G' \in  {\mathcal F}_{i+1}$. 

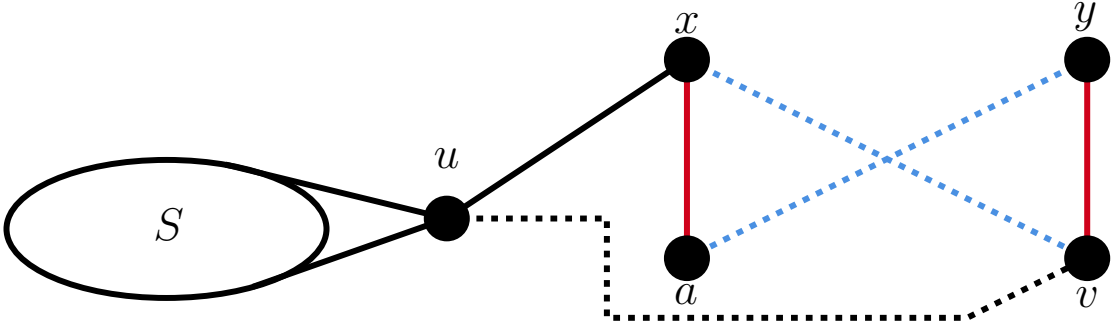
\begin{figure}
    \centering
    \tikzset{every picture/.style={line width=0.75pt}} %set default line width to 0.75pt        

\begin{tikzpicture}[x=0.75pt,y=0.75pt,yscale=-1,xscale=1]
%uncomment if require: \path (0,285); %set diagram left start at 0, and has height of 285

%Straight Lines [id:da7902544107254279] 
\draw [color={rgb, 255:red, 208; green, 2; blue, 27 }  ,draw opacity=1 ][line width=2.25]    (420,40) -- (420,92) -- (420,140) ;
%Straight Lines [id:da5973227050682391] 
\draw [color={rgb, 255:red, 74; green, 144; blue, 226 }  ,draw opacity=1 ][line width=2.25]  [dash pattern={on 2.53pt off 3.02pt}]  (420,40) -- (620,140) ;
%Straight Lines [id:da3736928674677139] 
\draw [color={rgb, 255:red, 208; green, 2; blue, 27 }  ,draw opacity=1 ][line width=2.25]    (620,40) -- (620,97) -- (620,140) ;
%Straight Lines [id:da4843478709402741] 
\draw [color={rgb, 255:red, 74; green, 144; blue, 226 }  ,draw opacity=1 ][line width=2.25]  [dash pattern={on 2.53pt off 3.02pt}]  (620,40) -- (420,140) ;
%Straight Lines [id:da5116064430857867] 
\draw [color={rgb, 255:red, 0; green, 0; blue, 0 }  ,draw opacity=1 ][line width=2.25]    (300,120) -- (420,40) ;
%Shape: Ellipse [id:dp8422052247887316] 
\draw  [line width=2.25]  (80,125.25) .. controls (80,106.06) and (115.79,90.5) .. (159.95,90.5) .. controls (204.11,90.5) and (239.9,106.06) .. (239.9,125.25) .. controls (239.9,144.44) and (204.11,160) .. (159.95,160) .. controls (115.79,160) and (80,144.44) .. (80,125.25) -- cycle ;
%Straight Lines [id:da07993130498116319] 
\draw [color={rgb, 255:red, 0; green, 0; blue, 0 }  ,draw opacity=1 ][line width=2.25]    (189.44,92.82) -- (300,120) ;
%Straight Lines [id:da5136443527495451] 
\draw [color={rgb, 255:red, 0; green, 0; blue, 0 }  ,draw opacity=1 ][line width=2.25]    (203.04,154.42) -- (300,120) ;
%Shape: Circle [id:dp1778737973884188] 
\draw  [color={rgb, 255:red, 0; green, 0; blue, 0 }  ,draw opacity=1 ][fill={rgb, 255:red, 0; green, 0; blue, 0 }  ,fill opacity=1 ][line width=2.25]  (300,110) .. controls (305.52,110) and (310,114.48) .. (310,120) .. controls (310,125.52) and (305.52,130) .. (300,130) .. controls (294.48,130) and (290,125.52) .. (290,120) .. controls (290,114.48) and (294.48,110) .. (300,110) -- cycle ;
%Shape: Circle [id:dp9751039364949254] 
\draw  [color={rgb, 255:red, 0; green, 0; blue, 0 }  ,draw opacity=1 ][fill={rgb, 255:red, 0; green, 0; blue, 0 }  ,fill opacity=1 ][line width=2.25]  (420,30) .. controls (425.52,30) and (430,34.48) .. (430,40) .. controls (430,45.52) and (425.52,50) .. (420,50) .. controls (414.48,50) and (410,45.52) .. (410,40) .. controls (410,34.48) and (414.48,30) .. (420,30) -- cycle ;
%Shape: Circle [id:dp5473192861357994] 
\draw  [color={rgb, 255:red, 0; green, 0; blue, 0 }  ,draw opacity=1 ][fill={rgb, 255:red, 0; green, 0; blue, 0 }  ,fill opacity=1 ][line width=2.25]  (620,30) .. controls (625.52,30) and (630,34.48) .. (630,40) .. controls (630,45.52) and (625.52,50) .. (620,50) .. controls (614.48,50) and (610,45.52) .. (610,40) .. controls (610,34.48) and (614.48,30) .. (620,30) -- cycle ;
%Shape: Circle [id:dp853875323478874] 
\draw  [color={rgb, 255:red, 0; green, 0; blue, 0 }  ,draw opacity=1 ][fill={rgb, 255:red, 0; green, 0; blue, 0 }  ,fill opacity=1 ][line width=2.25]  (420,130) .. controls (425.52,130) and (430,134.48) .. (430,140) .. controls (430,145.52) and (425.52,150) .. (420,150) .. controls (414.48,150) and (410,145.52) .. (410,140) .. controls (410,134.48) and (414.48,130) .. (420,130) -- cycle ;
%Shape: Circle [id:dp8207598738412637] 
\draw  [color={rgb, 255:red, 0; green, 0; blue, 0 }  ,draw opacity=1 ][fill={rgb, 255:red, 0; green, 0; blue, 0 }  ,fill opacity=1 ][line width=2.25]  (620,130) .. controls (625.52,130) and (630,134.48) .. (630,140) .. controls (630,145.52) and (625.52,150) .. (620,150) .. controls (614.48,150) and (610,145.52) .. (610,140) .. controls (610,134.48) and (614.48,130) .. (620,130) -- cycle ;
%Straight Lines [id:da27329750064369285] 
\draw [color={rgb, 255:red, 0; green, 0; blue, 0 }  ,draw opacity=1 ][line width=2.25]  [dash pattern={on 2.53pt off 3.02pt}]  (300,120) -- (380,120) -- (380,170) -- (560,170) -- (620,140) ;

% Text Node
\draw (420,28) node [anchor=south] [inner sep=0.75pt]  [font=\LARGE] [align=left] {$\displaystyle x$};
% Text Node
\draw (419.5,151.74) node [anchor=north] [inner sep=0.75pt]  [font=\LARGE] [align=left] {$\displaystyle a$};
% Text Node
\draw (620,153) node [anchor=north] [inner sep=0.75pt]  [font=\LARGE] [align=left] {$\displaystyle v$};
% Text Node
\draw (619,27.26) node [anchor=south] [inner sep=0.75pt]  [font=\LARGE] [align=left] {$\displaystyle y$};
% Text Node
\draw (300,89.5) node  [font=\LARGE] [align=left] {$\displaystyle u$};
% Text Node
\draw (161,123) node  [font=\LARGE] [align=left] {$\displaystyle S$};

\end{tikzpicture}
    \caption{The switching used in the proof of Lemma \ref{lem:mlogm-same-comp}. The vertex $a$ is added into $S$ after the switch. The dashed black lines indicate that no paths of length at most $4$ exist between $u$ and $v$ outside of $S$, which implies the absence of the edges $xv$ and $ay$.}
    \label{fig:3.3}
\end{figure}

For  any $G$ in ${\cal F}_i$ for $i\leq10\log{m}$ there are at least $\frac{d(u)}{2(\log{m})^2}-2i>\frac{d(u)}{3(\log{m})^2}$ choices of $xa$ and at least $d(v)-2i>\frac{d(v)}{2}$ choices of $yv$. On the other hand, for a $G'$ in ${\mathcal F}_{i+1}$, there are at most $2(i+1)m$ choices of $\{xv,ya\}$. Thus for $0\leq i<10\log{m}$,
\begin{equation*}
    \p{\cF_i}\leq\p{\cF_{i+1}}\frac{20m\log{m}}{\frac{d(u)d(v)}{6(\log{m})^2}}\leq\p{\cF_{i+1}}\frac{120}{\log{m}}
\end{equation*}
so $\p{\cF_0}\leq(\frac{120}{\log{m}})^{10\log{m}}<m^{-9.5\log\log{m}}$. By Lemma \ref{lem:non-leaf-neighbour}, it follows that $\p{\cA^\complement}<m^{-49\log\log{m}}$ so $\p{\cB_0}<m^{-9.5\log\log{m}}+m^{-49\log\log{m}}$.
\end{proof}
In particular, when $d(n)$ is high enough, all vertices of degree at least $\sqrt{m}(\log{m})^{-2}$ are in the same component directly by Lemma \ref{lem:mlogm-same-comp}.
\begin{cor}
\label{sec6cor1}
     Let $\cD$ be such that $m$ is large enough, $n_1 \le n(1-\log{m})$, and   $d(n) \ge \sqrt{m}(\log{m})^6$. Then with probability  $1-o(m^{-8\log\log{m}})$, all the vertices  of degree at least $\sqrt{m}(\log{m})^{-2}$ are in the same component.
\end{cor}
\begin{proof}
     Applying  Lemma \ref{lem:mlogm-same-comp} to   $d(n)$ and any vertex $v$ with $d(v)\geq\sqrt{m}(\log{m})^{-2}$, the probability $v$ is not in the same component as $d(n)$ is at most $m^{-9\log\log{m}}$. Summing up over all the at most $m$  choices for $v$ yields the desired result. 
\end{proof}
The next lemma states that when $D^*\geq\eps m$ but $d(n)\leq\sqrt{m}(\log{m})^6$, the vertices of degree at least $\sqrt{m}(\log{m})^{-2}$ are also in the same component. Lemma \ref{bigguyssticktogether} and Corollary \ref{sec6cor1} together show that when $D^*\geq\eps m$ or $d(n)\geq\sqrt{m}(\log{m})^6$, every exploration which reaches a vertex of degree at least $\sqrt{m}(\log{m})^{-2}$ is exploring the same component.
\begin{figure}
    \centering
    \tikzset{every picture/.style={line width=0.75pt}} %set default line width to 0.75pt        

\begin{tikzpicture}[x=0.75pt,y=0.75pt,yscale=-1,xscale=1]
%uncomment if require: \path (0,225); %set diagram left start at 0, and has height of 225

%Shape: Ellipse [id:dp46642751053148046] 
\draw  [line width=2.25]  (198,115) .. controls (198,62.53) and (220.39,20) .. (248,20) .. controls (275.61,20) and (298,62.53) .. (298,115) .. controls (298,167.47) and (275.61,210) .. (248,210) .. controls (220.39,210) and (198,167.47) .. (198,115) -- cycle ;
%Straight Lines [id:da623556535074324] 
\draw [color={rgb, 255:red, 208; green, 2; blue, 27 }  ,draw opacity=1 ][line width=2.25]    (269,60) -- (459,60) ;
%Straight Lines [id:da3400640579920118] 
\draw [color={rgb, 255:red, 208; green, 2; blue, 27 }  ,draw opacity=1 ][line width=2.25]    (269,170) -- (459,170) ;
%Straight Lines [id:da3114727471982405] 
\draw [color={rgb, 255:red, 74; green, 144; blue, 226 }  ,draw opacity=1 ][line width=2.25]  [dash pattern={on 2.53pt off 3.02pt}]  (459,170) -- (459,60) ;
%Straight Lines [id:da9248670273760919] 
\draw [color={rgb, 255:red, 74; green, 144; blue, 226 }  ,draw opacity=1 ][line width=2.25]  [dash pattern={on 2.53pt off 3.02pt}]  (269,170) -- (269,60) ;
%Straight Lines [id:da30139049363577164] 
\draw [line width=2.25]  [dash pattern={on 2.53pt off 3.02pt}]  (269,60) -- (328,90) -- (328,140) -- (269,170) ;
%Shape: Circle [id:dp9699031125370364] 
\draw  [color={rgb, 255:red, 0; green, 0; blue, 0 }  ,draw opacity=1 ][fill={rgb, 255:red, 0; green, 0; blue, 0 }  ,fill opacity=1 ][line width=2.25]  (269,50) .. controls (274.52,50) and (279,54.48) .. (279,60) .. controls (279,65.52) and (274.52,70) .. (269,70) .. controls (263.48,70) and (259,65.52) .. (259,60) .. controls (259,54.48) and (263.48,50) .. (269,50) -- cycle ;
%Shape: Circle [id:dp3719090852955125] 
\draw  [color={rgb, 255:red, 0; green, 0; blue, 0 }  ,draw opacity=1 ][fill={rgb, 255:red, 0; green, 0; blue, 0 }  ,fill opacity=1 ][line width=2.25]  (269,160) .. controls (274.52,160) and (279,164.48) .. (279,170) .. controls (279,175.52) and (274.52,180) .. (269,180) .. controls (263.48,180) and (259,175.52) .. (259,170) .. controls (259,164.48) and (263.48,160) .. (269,160) -- cycle ;
%Shape: Circle [id:dp00749877793575926] 
\draw  [color={rgb, 255:red, 0; green, 0; blue, 0 }  ,draw opacity=1 ][fill={rgb, 255:red, 0; green, 0; blue, 0 }  ,fill opacity=1 ][line width=2.25]  (459,50) .. controls (464.52,50) and (469,54.48) .. (469,60) .. controls (469,65.52) and (464.52,70) .. (459,70) .. controls (453.48,70) and (449,65.52) .. (449,60) .. controls (449,54.48) and (453.48,50) .. (459,50) -- cycle ;
%Shape: Circle [id:dp5939183261785399] 
\draw  [color={rgb, 255:red, 0; green, 0; blue, 0 }  ,draw opacity=1 ][fill={rgb, 255:red, 0; green, 0; blue, 0 }  ,fill opacity=1 ][line width=2.25]  (459,160) .. controls (464.52,160) and (469,164.48) .. (469,170) .. controls (469,175.52) and (464.52,180) .. (459,180) .. controls (453.48,180) and (449,175.52) .. (449,170) .. controls (449,164.48) and (453.48,160) .. (459,160) -- cycle ;

% Text Node
\draw (258,59.5) node [anchor=east] [inner sep=0.75pt]  [font=\LARGE] [align=left] {$\displaystyle n$};
% Text Node
\draw (257,170) node [anchor=east] [inner sep=0.75pt]  [font=\LARGE] [align=left] {$\displaystyle u$};
% Text Node
\draw (197,112.5) node [anchor=east] [inner sep=0.75pt]  [font=\LARGE] [align=left] {$\displaystyle H$};
% Text Node
\draw (471,60) node [anchor=west] [inner sep=0.75pt]  [font=\LARGE] [align=left] {$\displaystyle v$};
% Text Node
\draw (471,170) node [anchor=west] [inner sep=0.75pt]  [font=\LARGE] [align=left] {$\displaystyle x$};

\end{tikzpicture}
    \caption{First switching in the proof of Lemma \ref{bigguyssticktogether}. The black lines indicate no paths of length at most $3$ exist between $n,u$, implying the absence of the blue edges.}
    \label{fig:3.5-1}
\end{figure}
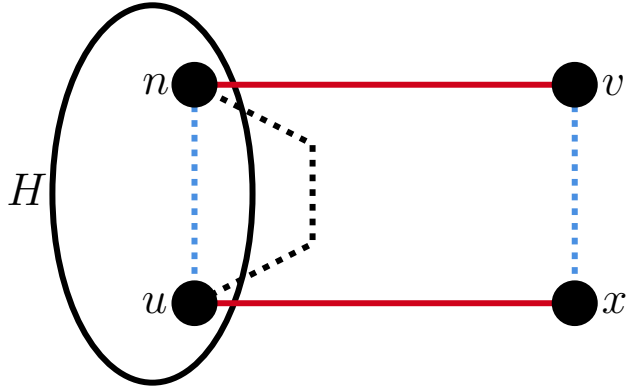
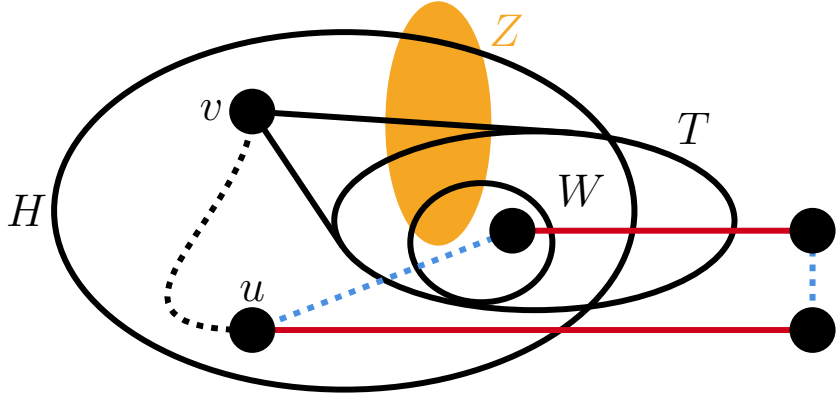
\begin{figure}
    \centering
    \tikzset{every picture/.style={line width=0.75pt}} %set default line width to 0.75pt        

\begin{tikzpicture}[x=0.75pt,y=0.75pt,yscale=-1,xscale=1]
%uncomment if require: \path (0,246); %set diagram left start at 0, and has height of 246

%Shape: Ellipse [id:dp48065709568994597] 
\draw  [color={rgb, 255:red, 245; green, 166; blue, 35 }  ,draw opacity=1 ][fill={rgb, 255:red, 245; green, 166; blue, 35 }  ,fill opacity=1 ][line width=2.25]  (308,66) .. controls (308,32.86) and (319.19,6) .. (333,6) .. controls (346.81,6) and (358,32.86) .. (358,66) .. controls (358,99.14) and (346.81,126) .. (333,126) .. controls (319.19,126) and (308,99.14) .. (308,66) -- cycle ;
%Shape: Circle [id:dp9057546915818485] 
\draw  [color={rgb, 255:red, 0; green, 0; blue, 0 }  ,draw opacity=1 ][fill={rgb, 255:red, 0; green, 0; blue, 0 }  ,fill opacity=1 ][line width=2.25]  (240,50) .. controls (245.52,50) and (250,54.48) .. (250,60) .. controls (250,65.52) and (245.52,70) .. (240,70) .. controls (234.48,70) and (230,65.52) .. (230,60) .. controls (230,54.48) and (234.48,50) .. (240,50) -- cycle ;
%Shape: Ellipse [id:dp8809288024275339] 
\draw  [line width=2.25]  (141,110) .. controls (141,60.29) and (205.92,20) .. (286,20) .. controls (366.08,20) and (431,60.29) .. (431,110) .. controls (431,159.71) and (366.08,200) .. (286,200) .. controls (205.92,200) and (141,159.71) .. (141,110) -- cycle ;
%Straight Lines [id:da2261636438074991] 
\draw [color={rgb, 255:red, 208; green, 2; blue, 27 }  ,draw opacity=1 ][line width=2.25]    (240,170) -- (520,170) ;
%Straight Lines [id:da17888062605491684] 
\draw [color={rgb, 255:red, 74; green, 144; blue, 226 }  ,draw opacity=1 ][line width=2.25]  [dash pattern={on 2.53pt off 3.02pt}]  (520,170) -- (520,120) ;
%Straight Lines [id:da9715646035518999] 
\draw [line width=2.25]    (240,60) -- (285.2,128.13) ;
%Straight Lines [id:da6136888223734489] 
\draw [line width=2.25]    (240,60) -- (400.4,70.72) ;
%Shape: Ellipse [id:dp4181280577570591] 
\draw  [line width=2.25]  (281,115) .. controls (281,90.15) and (325.77,70) .. (381,70) .. controls (436.23,70) and (481,90.15) .. (481,115) .. controls (481,139.85) and (436.23,160) .. (381,160) .. controls (325.77,160) and (281,139.85) .. (281,115) -- cycle ;
%Shape: Ellipse [id:dp2549783642576722] 
\draw  [line width=2.25]  (319,126) .. controls (319,109.43) and (334.89,96) .. (354.5,96) .. controls (374.11,96) and (390,109.43) .. (390,126) .. controls (390,142.57) and (374.11,156) .. (354.5,156) .. controls (334.89,156) and (319,142.57) .. (319,126) -- cycle ;
%Straight Lines [id:da6527446170351245] 
\draw [color={rgb, 255:red, 208; green, 2; blue, 27 }  ,draw opacity=1 ][line width=2.25]    (370,120) -- (520,120) ;
%Shape: Circle [id:dp018824971347771835] 
\draw  [color={rgb, 255:red, 0; green, 0; blue, 0 }  ,draw opacity=1 ][fill={rgb, 255:red, 0; green, 0; blue, 0 }  ,fill opacity=1 ][line width=2.25]  (520,110) .. controls (525.52,110) and (530,114.48) .. (530,120) .. controls (530,125.52) and (525.52,130) .. (520,130) .. controls (514.48,130) and (510,125.52) .. (510,120) .. controls (510,114.48) and (514.48,110) .. (520,110) -- cycle ;
%Straight Lines [id:da43891207276686217] 
\draw [color={rgb, 255:red, 74; green, 144; blue, 226 }  ,draw opacity=1 ][line width=2.25]  [dash pattern={on 2.53pt off 3.02pt}]  (240,170) -- (370,120) ;
%Shape: Circle [id:dp34077723116913516] 
\draw  [color={rgb, 255:red, 0; green, 0; blue, 0 }  ,draw opacity=1 ][fill={rgb, 255:red, 0; green, 0; blue, 0 }  ,fill opacity=1 ][line width=2.25]  (240,160) .. controls (245.52,160) and (250,164.48) .. (250,170) .. controls (250,175.52) and (245.52,180) .. (240,180) .. controls (234.48,180) and (230,175.52) .. (230,170) .. controls (230,164.48) and (234.48,160) .. (240,160) -- cycle ;
%Shape: Circle [id:dp21365404494414986] 
\draw  [color={rgb, 255:red, 0; green, 0; blue, 0 }  ,draw opacity=1 ][fill={rgb, 255:red, 0; green, 0; blue, 0 }  ,fill opacity=1 ][line width=2.25]  (370,110) .. controls (375.52,110) and (380,114.48) .. (380,120) .. controls (380,125.52) and (375.52,130) .. (370,130) .. controls (364.48,130) and (360,125.52) .. (360,120) .. controls (360,114.48) and (364.48,110) .. (370,110) -- cycle ;
%Shape: Circle [id:dp9812384210225791] 
\draw  [color={rgb, 255:red, 0; green, 0; blue, 0 }  ,draw opacity=1 ][fill={rgb, 255:red, 0; green, 0; blue, 0 }  ,fill opacity=1 ][line width=2.25]  (520,160) .. controls (525.52,160) and (530,164.48) .. (530,170) .. controls (530,175.52) and (525.52,180) .. (520,180) .. controls (514.48,180) and (510,175.52) .. (510,170) .. controls (510,164.48) and (514.48,160) .. (520,160) -- cycle ;
%Curve Lines [id:da39180397874705286] 
\draw [line width=2.25]  [dash pattern={on 2.53pt off 3.02pt}]  (240,60) .. controls (240.48,110.8) and (144.48,170.4) .. (240,170) ;

% Text Node
\draw (227,59.5) node [anchor=east] [inner sep=0.75pt]  [font=\LARGE] [align=left] {$\displaystyle v$};
% Text Node
\draw (240,157) node [anchor=south] [inner sep=0.75pt]  [font=\LARGE] [align=left] {$\displaystyle u$};
% Text Node
\draw (139,110) node [anchor=east] [inner sep=0.75pt]  [font=\LARGE] [align=left] {$\displaystyle H$};
% Text Node
\draw (356.96,28.26) node [anchor=south west] [inner sep=0.75pt]  [font=\LARGE,color={rgb, 255:red, 245; green, 166; blue, 35 }  ,opacity=1 ] [align=left] {$\displaystyle Z$};
% Text Node
\draw (222,237) node  [font=\large,color={rgb, 255:red, 245; green, 166; blue, 35 }  ,opacity=1 ] [align=left] {$ $};
% Text Node
\draw (451,78) node [anchor=south west] [inner sep=0.75pt]  [font=\LARGE] [align=left] {$\displaystyle T$};
% Text Node
\draw (404.5,99.5) node  [font=\LARGE] [align=left] {$\displaystyle W$};

\end{tikzpicture}
    \caption{Second switching in the proof of Lemma \ref{bigguyssticktogether}. Every switching makes a member of $W$ into a neighbour of $u$. The dotted black line indicates there are no paths between $u,v$ outside of $Z$, which implies the absence of the blue edges.}
    \label{fig:3.5-2}
\end{figure}
\begin{lem}
\label{bigguyssticktogether}

     For all sufficiently small $\epsilon>0$, the following holds. Let $\cD$ be such that $m$ is large enough, $d(n) \le \sqrt{m}(\log{m})^6$, and $D^*\ge\eps m$. Then with probability $1-o(m^{-8\log\log{m}})$, all the vertices  of degree at least $\sqrt{m}(\log{m})^{-2}$ are in the same component.
\end{lem}
\begin{proof}
    We let $H$ be the set of vertices of degree at least $\frac{\sqrt{m}}{log^7 m}$. Our hypotheses imply that the sum of the degrees of the vertices 
    in $H$ is at least $\frac{\epsilon m}{2}$ and that $|H| \ge \frac{\sqrt{m}}{log^{20} m}$.

    We claim that the probability that there is no vertex of $H$ within distance three of at least $m^{1/8}$ vertices of $H$  is at most $m^{-10 \log \log m}$. 

    To prove our claim, for $i \le 2m^{1/8}$,  we let ${\mathcal F}_i$ be the set of graphs for which $n$ is adjacent to  $i$ vertices of  $H$, no vertex except $n$
    is incident to more than $m^{1/8}$ vertices of $H-n$  and there are at most $2(i+1) m^{1/8}$ vertices of $H$ at distance within 3 of any vertex of $H$ in  $G-n$  and at most  $2(i+1)^2 m^{1/8}$ vertices  of $H$ at distance within 3 of $n$. We  show $\p{\cup_{i \le m^{1/8}} {\mathcal F_i}} \le m^{-12 \log \log m}$.

    To do so, for $i \le 2m^{1/8}$, we consider a swap from a  graph in ${\mathcal F}_i$  using  edges $nv, ux$ 
    where $v$ is a neighbour of $n$ not in $H$, $u$ is a vertex of $H$ at distance at least four from $n$ 
    and $x$ is a neighbour of $u$ not in $H$ as depicted in Figure \ref{fig:3.5-1}. We obtain a graph $G'$ such that $n$ is adjacent to $i+1$ vertices of $H$
    and no vertex except  $n$ is adjacent to more than $m^{1/8}$ vertices of $H-n$. For any $w$ in $H-n$, any vertices of $H$ which are now at distance 3 from $w$ in $G-n$ but were not previously must be joined to $w$ by a path with three edges 
    whose internal vertices are $v$ and $x$. Hence they must be one of the at most $2m^{1/8} $ vertices  in $H-n$ which are neighbours of 
    $v$ or $x$.  Any vertices of $H$ which are now at distance 3 from $n$ in $G$ but were not previously must either 
    have been at distance at most 3 from $u$ in $G-n$ or  be joined to $w$ by a path with three edges 
    whose internal vertices are $v$ and $x$. But the latter is impossible as neither $x$ nor $v$  is  a neighbour of $n$  in $G'$. 
    So $G' \in {\mathcal F}_{i+1}$. 
    
    Now, there are at least $\frac{\sqrt{m}}{\log^8 m}-i>m^{3/7}$ choices for $v$. There are at least $|H|-2(i+1)^2m^{1/8}>m^{3/7}$
    choices for $u$. For each such choice there are at least $\frac{\sqrt{m}}{\log^8 m}-m^{1/8}-1>m^{3/7}$ choices for $x$.
    Hence there are at least $m^{9/7}$ choices for our swap. on the other hand, for any $G'$ in ${\mathcal F}_{i+1}$  there are at most $2(i+1)m<m^{8/7}$ swaps from $G'$ to graphs in ${\mathcal F}_i$. 

    So for each $i\leq 2m^{1/8}$, we have $\frac{|{\mathcal F}_i|}{|{\mathcal F}_{i+1}|} \le m^{-1/7}$ and our claim is proved. 

    We next claim that for any $v$ and $u \in H$, the probability that $u$ and $v$ are in different components and there are at least $m^{1/8}
    $ vertices of $H$ within distance $3$ of $v$ is at most $m^{-12 \log \log m}$. The lemma follows from this, so it remains to prove the claim.
    
    To this end we let ${\mathcal F}_i$ be the event that there are 
    there are at least $m^{1/8}$ vertices of $H$ within distance $3$ of $v$, there are  $i$  edges between these vertices and $u$ 
    and there is a set  $Z$ of at most $2i$ vertices such that $u$ and $v$ are in different components of $G-Z$.

    For $i \le\log^2 m$, and  any graph in ${\mathcal F}_i$ we let $T$ be a tree  
    of height three obtained starting from the one-vertex tree containing only $v$ by repeatedly adding a vertex of $H$ at distance  at most three from $v$ and a subset of a shortest path from this vertex to $v$ until the tree contains a set $W$ of at least $m^{1/8}$ vertices of $H$ at distance at most 3 from $v$, $i$ of which are adjacent to $u$ (these latter are added first).  We add at most three vertices to $T$ in each iteration so  $|T| \le 3m^{1/8}+1$. We can obtain a graph in ${\mathcal F}_{i+1}$ by swapping any pair of edges consisting of  (i) an edge from one of the at least $m^{1/9}$ vertices of $W$ nonadjacent to $u$ to one of its at least $\frac{\sqrt{m}}{\log^7 m}-3m^{1/8}-1-2i>m^{13/27}$  neighbours outside $Z \cup T$ and (ii)  any of the at least $\frac{\sqrt{m}}{\log^7 m}-3m^{1/8}-1-2i>m^{13/27}$  edges from $u$ to  a neighbour outside $Z \cup T$, as depicted in Figure \ref{fig:3.5-2}.
    So, there are at least $m^{29/27}$ such swaps from $G$. On the other hand there are at most $(i+1)m$ swaps to ${\mathcal F}_i$
    from  a $G'$ in ${\mathcal F}_{i+1}$. Our claim then follows in the standard way. 
\end{proof}
We now have a strategy to prove Lemma \ref{theyjoinmain}, by exploring out of every $v\in[n]$ as described in Section \ref{sec:outline}. We have two cases based on $D^*$ and $d(n)$:
\begin{itemize}
    \item If  $D^*\geq \eps m$ or   $d(n) \ge \sqrt{m}(\log{m})^6$, explore until $X_i\geq\eps m$ or $X_i=0$, or until a vertex of degree at least $\sqrt{m}(\log{m})^{-2}$ is reached. This means $V(H_i)$ has maximum degree less than $\sqrt{m}(\log{m})^{-2}$ until the iteration where a vertex of degree at least $\sqrt{m}(\log{m})^{-2}$ is reached.
    \item If $D^*<\eps m$ and $d(n)<\sqrt{m}(\log{m})^6$, explore until  $X_i\geq\eps m$ or $X_i=0$. 
   
\end{itemize}
To apply the strategy, we prove another key lemma which shows that any two vertices whose explorations reach $\Omega(m)$ are with high probability in the same component, and so are any vertex whose exploration reaches  $\Omega(m)$ and any vertex of degree exceeding $\sqrt{m}(\log{m})^{-2}$ (we will refer to such vertices as \emph{high degree vertices} in the remainder of this section). Together, this tells us that we may stop an exploration if it reaches $\eps m$ or hits a high degree vertex.
\begin{figure}
    \centering
    \tikzset{every picture/.style={line width=0.75pt}} %set default line width to 0.75pt        

\begin{tikzpicture}[x=0.75pt,y=0.75pt,yscale=-1,xscale=1]
%uncomment if require: \path (0,240); %set diagram left start at 0, and has height of 240

%Shape: Ellipse [id:dp018152647777371023] 
\draw  [line width=3]  (120.1,54.75) .. controls (120.1,35.56) and (155.89,20) .. (200.05,20) .. controls (244.21,20) and (280,35.56) .. (280,54.75) .. controls (280,73.94) and (244.21,89.5) .. (200.05,89.5) .. controls (155.89,89.5) and (120.1,73.94) .. (120.1,54.75) -- cycle ;
%Shape: Ellipse [id:dp23144488450200273] 
\draw  [line width=3]  (360.1,184.75) .. controls (360.1,165.56) and (395.89,150) .. (440.05,150) .. controls (484.21,150) and (520,165.56) .. (520,184.75) .. controls (520,203.94) and (484.21,219.5) .. (440.05,219.5) .. controls (395.89,219.5) and (360.1,203.94) .. (360.1,184.75) -- cycle ;
%Straight Lines [id:da43699233212686106] 
\draw [color={rgb, 255:red, 0; green, 0; blue, 0 }  ,draw opacity=1 ][line width=3]  [dash pattern={on 3.38pt off 3.27pt}]  (180,100) -- (360,180) ;
%Straight Lines [id:da5783058288191817] 
\draw [color={rgb, 255:red, 208; green, 2; blue, 27 }  ,draw opacity=1 ][line width=3]    (230,60) -- (230,137) -- (230,180) ;
%Straight Lines [id:da059979122908381255] 
\draw [color={rgb, 255:red, 74; green, 144; blue, 226 }  ,draw opacity=1 ][line width=3]  [dash pattern={on 3.38pt off 3.27pt}]  (230,60) -- (400,180) ;
%Straight Lines [id:da6088534596138745] 
\draw [color={rgb, 255:red, 208; green, 2; blue, 27 }  ,draw opacity=1 ][line width=3]    (400,60) -- (400,140) -- (400,183) ;
%Straight Lines [id:da857488192157762] 
\draw [color={rgb, 255:red, 74; green, 144; blue, 226 }  ,draw opacity=1 ][line width=3]  [dash pattern={on 3.38pt off 3.27pt}]  (400,60) -- (230,180) ;
%Shape: Ellipse [id:dp7881250013700054] 
\draw  [color={rgb, 255:red, 245; green, 166; blue, 35 }  ,draw opacity=1 ][fill={rgb, 255:red, 245; green, 166; blue, 35 }  ,fill opacity=1 ][line width=3]  (130,80) .. controls (130,57.91) and (141.19,40) .. (155,40) .. controls (168.81,40) and (180,57.91) .. (180,80) .. controls (180,102.09) and (168.81,120) .. (155,120) .. controls (141.19,120) and (130,102.09) .. (130,80) -- cycle ;
%Curve Lines [id:da8262928277519199] 
\draw [color={rgb, 255:red, 0; green, 0; blue, 0 }  ,draw opacity=1 ][line width=3]  [dash pattern={on 3.38pt off 3.27pt}]  (420,150) .. controls (529.87,16.66) and (549.25,17.33) .. (200.1,20) ;
%Shape: Circle [id:dp9019684144139832] 
\draw  [color={rgb, 255:red, 0; green, 0; blue, 0 }  ,draw opacity=1 ][fill={rgb, 255:red, 0; green, 0; blue, 0 }  ,fill opacity=1 ][line width=2.25]  (230,50) .. controls (235.52,50) and (240,54.48) .. (240,60) .. controls (240,65.52) and (235.52,70) .. (230,70) .. controls (224.48,70) and (220,65.52) .. (220,60) .. controls (220,54.48) and (224.48,50) .. (230,50) -- cycle ;
%Shape: Circle [id:dp7879928344252409] 
\draw  [color={rgb, 255:red, 0; green, 0; blue, 0 }  ,draw opacity=1 ][fill={rgb, 255:red, 0; green, 0; blue, 0 }  ,fill opacity=1 ][line width=2.25]  (400,50) .. controls (405.52,50) and (410,54.48) .. (410,60) .. controls (410,65.52) and (405.52,70) .. (400,70) .. controls (394.48,70) and (390,65.52) .. (390,60) .. controls (390,54.48) and (394.48,50) .. (400,50) -- cycle ;
%Shape: Circle [id:dp428378078557647] 
\draw  [color={rgb, 255:red, 0; green, 0; blue, 0 }  ,draw opacity=1 ][fill={rgb, 255:red, 0; green, 0; blue, 0 }  ,fill opacity=1 ][line width=2.25]  (230,170) .. controls (235.52,170) and (240,174.48) .. (240,180) .. controls (240,185.52) and (235.52,190) .. (230,190) .. controls (224.48,190) and (220,185.52) .. (220,180) .. controls (220,174.48) and (224.48,170) .. (230,170) -- cycle ;
%Shape: Circle [id:dp9858049189508404] 
\draw  [color={rgb, 255:red, 0; green, 0; blue, 0 }  ,draw opacity=1 ][fill={rgb, 255:red, 0; green, 0; blue, 0 }  ,fill opacity=1 ][line width=2.25]  (400,170) .. controls (405.52,170) and (410,174.48) .. (410,180) .. controls (410,185.52) and (405.52,190) .. (400,190) .. controls (394.48,190) and (390,185.52) .. (390,180) .. controls (390,174.48) and (394.48,170) .. (400,170) -- cycle ;
%Shape: Ellipse [id:dp7825312116906887] 
\draw  [draw opacity=0][fill={rgb, 255:red, 184; green, 233; blue, 134 }  ,fill opacity=1 ][line width=3]  (490,115) .. controls (490,101.19) and (501.19,90) .. (515,90) .. controls (528.81,90) and (540,101.19) .. (540,115) .. controls (540,128.81) and (528.81,140) .. (515,140) .. controls (501.19,140) and (490,128.81) .. (490,115) -- cycle ;
%Straight Lines [id:da4348774490049879] 
\draw [color={rgb, 255:red, 0; green, 0; blue, 0 }  ,draw opacity=1 ][line width=3]    (525.44,139.5) -- (483.84,155.5) ;
%Straight Lines [id:da905543624950799] 
\draw [color={rgb, 255:red, 0; green, 0; blue, 0 }  ,draw opacity=1 ][line width=3]    (490.24,112.3) -- (483.84,155.5) ;

% Text Node
\draw (218,60) node [anchor=east] [inner sep=0.75pt]  [font=\LARGE] [align=left] {$\displaystyle u'$};
% Text Node
\draw (218,180) node [anchor=east] [inner sep=0.75pt]  [font=\LARGE] [align=left] {$\displaystyle x$};
% Text Node
\draw (412,180) node [anchor=west] [inner sep=0.75pt]  [font=\LARGE] [align=left] {$\displaystyle v'$};
% Text Node
\draw (412.79,59.5) node [anchor=west] [inner sep=0.75pt]  [font=\LARGE] [align=left] {$\displaystyle y$};
% Text Node
\draw (128,80) node [anchor=east] [inner sep=0.75pt]  [font=\LARGE,color={rgb, 255:red, 245; green, 166; blue, 35 }  ,opacity=1 ] [align=left] {$\displaystyle Z$};
% Text Node
\draw (117.31,50.5) node [anchor=east] [inner sep=0.75pt]  [font=\LARGE] [align=left] {$\displaystyle S_{u}$};
% Text Node
\draw (522,184.75) node [anchor=west] [inner sep=0.75pt]  [font=\LARGE] [align=left] {$\displaystyle S_{v}$};
% Text Node
\draw (542,115) node [anchor=west] [inner sep=0.75pt]  [font=\LARGE,color={rgb, 255:red, 184; green, 233; blue, 134 }  ,opacity=1 ] [align=left] {$\displaystyle Y$};

\end{tikzpicture}
    \caption{Switching used in the proof of Lemma \ref{theyjoin}. The black dotted lines indicate that there are no edge between $Z$ and $S_v$ outside of $Y$, and no paths between $S_u$ and $S_v$ outside of $Z$. The edge $yv'\notin Y$, so $y\notin Z$. The vertices $u'$ and $x$ are not in $Z$, so the fact that there are no paths between $S_u$ and $S_v$ outside of $Z$ ensures the absence of $u'v'$ and $xy$. }
    \label{fig:3.6}
\end{figure}
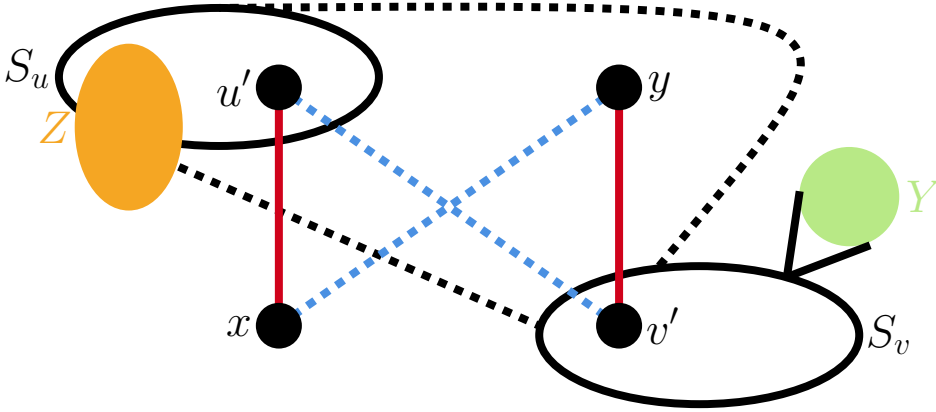
\begin{lem}
\label{theyjoin}
    Let $0<\eps<10^{-6}$ and let $\cD$ be such that $m$ is large enough and $n_1 \le (1-\frac{1}{\log m})n$. Fix $u,v\in[n]$ and let $\cF$ denote the event that $u,v$ are not in the same component, the exploration starting from $u$ has boundary size reaching at least $\eps m$ without reaching $0$ and either  (i) the same holds for $v$ or (ii) $v$ is a vertex of degree exceeding $\sqrt{m}(\log m)^{-2}$. Then $\p{\cF}<m^{-7\log\log{m}}$.
\end{lem}
\begin{proof}
We let $\cA$ denote the event that no vertex of degree at least  $\sqrt{ m}(\log m)^{-2}$ lies in a different component from any vertex of degree at least  $\sqrt{ m}(\log m)^{6}$. We let $\cB$ be the event that (ii) holds and $v$ has fewer than  $\sqrt{m}(\log{m})^{-5}$ non-leaf neighbours. By Corollary \ref{sec6cor1} and Lemma \ref{lem:non-leaf-neighbour}, $\p{\cA^\complement \cup \cB}=o(m^{-8\log\log{m}})$, so  it is enough to show 
$\p{(\cF \cap A)\setminus \cB}<m^{-8\log\log{m}}$.
We prove this conditioned on the choice of $H_u$,   which is the exploration  from $u$   up to the time at which the boundary first reached size $\epsilon m$, and $H_v$, which is defined symmetrically if (i) holds and is simply $H_1(v)$ if (ii) holds.  We use  ${\cal F}'=\cF'(H_u,H_v)$ to denote the event that $\cF$ holds with a specific  choice of   $H_u,H_v$. We use $\cF_0=(\cF' \cap\cA)\setminus\cB$. 

For every $G\in\cF'$,  if (i) holds let $S_u=S_u(G),S_v=S_v(G)$ denote the set of unexplored vertices in $H_u,H_v$ respectively when their boundary size first reaches $\eps m$.  If (ii) holds, the definition of $S_u$ is unchanged but define $S_v=\{w\in N(v):d(w)\geq 2\}$ instead. If (ii) holds, we only consider $H_v$ satisfying the event $\{|S_v|\geq\sqrt{m}(\log{m})^{-5}\}$ as for other choices, $\cF_0$ is empty. 
Relabelling $u$ and $v$ if necessary, we can assume that 
for every graph in ${\mathcal F}_0$, no vertex in the component containing $u$ has degree exceeding $\sqrt{m}(\log{m})^6$.  

For $i\geq0$ let $\cF_{i+1}$ denote the set of graphs with the given choice of $H_u$ and $H_v$ for  which  the number of edges 
between $S_u$ and $S_v$  is between  $i$ and $2i$,  and there is a set $Z\subseteq [n]\setminus S_v$ of  at most $2i$  vertices each of degree  at most $\sqrt{m}(\log{m})^6$, and a set $Y$ of at most $2i$ edges incident to $S_v$, such that the following hold:
\begin{enumerate}
    \item There is no path from $S_u$ to the union of $S_v$ and the set of vertices of degree exceeding $\sqrt{m}(\log{m})^6$ in $G-Z$, and
    \item All edges between $Z$ and $S_v$ are in $Y$.
\end{enumerate}
 \par 
For $0\leq i<m^{1/3}$, if $G\in\cF_i$ then there are at least $\epsilon m -2i(\sqrt{m}(\log{m})^6+1)$ unexplored  edges $u'x$ with   $u'\in S_u\setminus Z$ and $x\notin Z$, and  at least $\sqrt{m}(\log m)^{-5}-2i$ unexplored edges  with $v' \in S_v$ and $yv'\notin Y$.  
Note that $u'v'$ is not an edge, as there  is no path from $S_u$ to $S_v$ in $G-Z$; since $v'y \not \in Y$, $y$ is not in $Z$ so $yx$ is not an edge, as  $u'xyv'$  cannot be a path from $S_u$ to $S_v$ in $G-Z$. 
So the switching is valid. By the definition of $Z$, both $u'$ and $x$ have degree at most $\sqrt{m}(\log{m})^6$, and neither vertex is in $S_v$. Furthermore, after switching $\{u'x,yv'\}$ to get $\{u'v',yx\}$, the only edge between $u'$ and $S_v$ is $u'v'$, and the only possible edge between $x$ and $S_v$ is $yx$, since there were no such edges before the switching. Thus the resulting graph is in $\cF_{i+1}$ by putting $u',x'$ in $Z$ and (at most) $u'v',xy$ in $Y$.  So for every $G \in {\mathcal F}_i$ there are at least $ m^{13/9}$ switchings to graphs in $\cF_{i+1}$. 

There are at most $2(i+1)$ edges between $S_u,S_v$, so there are at most $4(i+1)m$ switchings from $\cF_{i+1}$ to $\cF_i$. So
\begin{equation*}
    \Cprob{\cF_0}{H_u,H_v}<\left(\frac{4m(m^{1/3})}{m^{13/9}}\right)^{m^{1/3}}<m^{-m^{1/4}}.
\end{equation*}
Thus $\p{\cF \cap \cA-\cB}<m^{-m^{1/4}}$.
\end{proof}

We now move onto studying the process $X_i$ for any starting vertex $v\in[n]$. The goal is to show that if  $v$ is in a linear sized component, then 
for some $i$, $X_i$ will be $\Omega(m)$ or $H_i$  will contain a high degree vertex, via a martingale argument. Since $X_i-X_{i-1}$ does not have a good upper bound, we study a truncated version of $X_i-X_{i-1}$, that we call $X_i^*$. For $1 \le i\leq m$, if $X_j=0$ for some $j\leq i-1$, set $X^*_i=0$, and otherwise define:
\begin{equation*}
    X^*_i=\begin{cases}
        \min(3d_i(v_i),X_i-X_{i-1}),&\text{if }d_i(v_i)<\frac{\sqrt{m}}{(\log m)^2};\\
        0.9d_i(v_i),&\text{if }d_i(v_i)\geq \frac{\sqrt{m}}{(\log m)^2}.
    \end{cases}
\end{equation*}
 Define $Y_0=0$, and for $1\leq i\leq m$ let
\begin{equation*}
    Y_i=Y_{i-1}+X^*_i-\Cexp{X^*_i}{H_{i-1}}.
\end{equation*}

\begin{lem}\label{lem:martingale}
If $\cD$ is such that $m$ is large enough, then for any vertex $v\in[n]$, for all $0 \le i \le m$, $\p{Y_i \le -\frac{m}{\log m}}<e^{-m^{1/10}}$. 
\end{lem}
\begin{proof}
 By definition $\{Y_i\}_{i=0}^m$ is a martingale with respect to $\{H_i\}_{i=0}^m$, and $|Y_i-Y_{i-1}|=|X_i^*-\Cexp{X_i^*}{H_{i-1}}|$ is $0$ if $d_i(v_i) > \frac{\sqrt{m}}{(\log m)}$ and is at most $3d_i(v_i) \le \frac{3\sqrt{m}}{(\log m)^2}$ otherwise. Furthermore,
 \begin{equation*}
\sum_{j\leq i:d_j(v_j)<\frac{\sqrt{m}}{(\log{m})^2}} (3d_j(v_j))^2 \le\frac{3\sqrt{m}}{(\log{m})^2}\sum_{j\leq i}d_j(v_j)\leq\frac{3m^{3/2}}{(\log{m})^2}     
\end{equation*}
so by Azuma's inequality \cite{azuma},
\begin{equation*}
    \p{Y_i<-\frac{m}{\log m}}<\exp\left\{\frac{-m^2(\log m)^{-2}}{3m^{3/2}(\log m)^{-2}}\right\}<e^{-m^{1/10}}.\qedhere
\end{equation*}
\end{proof}

\begin{figure}
    \centering
    \tikzset{every picture/.style={line width=0.75pt}} %set default line width to 0.75pt        

\begin{tikzpicture}[x=0.75pt,y=0.75pt,yscale=-1,xscale=1]
%uncomment if require: \path (0,220); %set diagram left start at 0, and has height of 220

%Straight Lines [id:da8733037974958292] 
\draw [color={rgb, 255:red, 74; green, 144; blue, 226 }  ,draw opacity=1 ][line width=2.25]  [dash pattern={on 2.53pt off 3.02pt}]  (130,60) -- (510,180) ;
%Straight Lines [id:da3359585018778789] 
\draw [color={rgb, 255:red, 208; green, 2; blue, 27 }  ,draw opacity=1 ][line width=2.25]    (510,60) -- (510,120) -- (510,180) ;
%Straight Lines [id:da917935894533302] 
\draw [color={rgb, 255:red, 74; green, 144; blue, 226 }  ,draw opacity=1 ][line width=2.25]  [dash pattern={on 2.53pt off 3.02pt}]  (510,60) -- (130,180) ;
%Shape: Ellipse [id:dp60900445683508] 
\draw  [line width=2.25]  (69,115) .. controls (69,62.53) and (89.15,20) .. (114,20) .. controls (138.85,20) and (159,62.53) .. (159,115) .. controls (159,167.47) and (138.85,210) .. (114,210) .. controls (89.15,210) and (69,167.47) .. (69,115) -- cycle ;
%Curve Lines [id:da16656843007572764] 
\draw [color={rgb, 255:red, 208; green, 2; blue, 27 }  ,draw opacity=1 ][line width=2.25]    (130,60) .. controls (211.2,59.85) and (201.2,179.51) .. (130,180) ;
%Shape: Circle [id:dp774363905106755] 
\draw  [color={rgb, 255:red, 0; green, 0; blue, 0 }  ,draw opacity=1 ][fill={rgb, 255:red, 0; green, 0; blue, 0 }  ,fill opacity=1 ][line width=2.25]  (130,50) .. controls (135.52,50) and (140,54.48) .. (140,60) .. controls (140,65.52) and (135.52,70) .. (130,70) .. controls (124.48,70) and (120,65.52) .. (120,60) .. controls (120,54.48) and (124.48,50) .. (130,50) -- cycle ;
%Shape: Circle [id:dp912557750520373] 
\draw  [color={rgb, 255:red, 0; green, 0; blue, 0 }  ,draw opacity=1 ][fill={rgb, 255:red, 0; green, 0; blue, 0 }  ,fill opacity=1 ][line width=2.25]  (130,170) .. controls (135.52,170) and (140,174.48) .. (140,180) .. controls (140,185.52) and (135.52,190) .. (130,190) .. controls (124.48,190) and (120,185.52) .. (120,180) .. controls (120,174.48) and (124.48,170) .. (130,170) -- cycle ;
%Shape: Circle [id:dp9112463790462468] 
\draw  [color={rgb, 255:red, 0; green, 0; blue, 0 }  ,draw opacity=1 ][fill={rgb, 255:red, 0; green, 0; blue, 0 }  ,fill opacity=1 ][line width=2.25]  (510,50) .. controls (515.52,50) and (520,54.48) .. (520,60) .. controls (520,65.52) and (515.52,70) .. (510,70) .. controls (504.48,70) and (500,65.52) .. (500,60) .. controls (500,54.48) and (504.48,50) .. (510,50) -- cycle ;
%Shape: Circle [id:dp6059528227857944] 
\draw  [color={rgb, 255:red, 0; green, 0; blue, 0 }  ,draw opacity=1 ][fill={rgb, 255:red, 0; green, 0; blue, 0 }  ,fill opacity=1 ][line width=2.25]  (510,170) .. controls (515.52,170) and (520,174.48) .. (520,180) .. controls (520,185.52) and (515.52,190) .. (510,190) .. controls (504.48,190) and (500,185.52) .. (500,180) .. controls (500,174.48) and (504.48,170) .. (510,170) -- cycle ;
%Straight Lines [id:da41325578236158844] 
\draw [color={rgb, 255:red, 208; green, 2; blue, 27 }  ,draw opacity=1 ][line width=2.25]    (140,60) -- (220,30) ;
%Straight Lines [id:da8641515810131055] 
\draw [color={rgb, 255:red, 208; green, 2; blue, 27 }  ,draw opacity=1 ][line width=2.25]    (140,60) -- (220,60) ;
%Shape: Circle [id:dp17348299555350188] 
\draw  [color={rgb, 255:red, 0; green, 0; blue, 0 }  ,draw opacity=1 ][fill={rgb, 255:red, 0; green, 0; blue, 0 }  ,fill opacity=1 ][line width=2.25]  (220,20) .. controls (225.52,20) and (230,24.48) .. (230,30) .. controls (230,35.52) and (225.52,40) .. (220,40) .. controls (214.48,40) and (210,35.52) .. (210,30) .. controls (210,24.48) and (214.48,20) .. (220,20) -- cycle ;
%Shape: Circle [id:dp5459249305473383] 
\draw  [color={rgb, 255:red, 0; green, 0; blue, 0 }  ,draw opacity=1 ][fill={rgb, 255:red, 0; green, 0; blue, 0 }  ,fill opacity=1 ][line width=2.25]  (220,50) .. controls (225.52,50) and (230,54.48) .. (230,60) .. controls (230,65.52) and (225.52,70) .. (220,70) .. controls (214.48,70) and (210,65.52) .. (210,60) .. controls (210,54.48) and (214.48,50) .. (220,50) -- cycle ;
%Straight Lines [id:da8759767230439282] 
\draw [color={rgb, 255:red, 0; green, 0; blue, 0 }  ,draw opacity=1 ][line width=2.25]    (220,30) -- (260,10) ;
%Straight Lines [id:da8781007885345837] 
\draw [color={rgb, 255:red, 0; green, 0; blue, 0 }  ,draw opacity=1 ][line width=2.25]  [dash pattern={on 2.53pt off 3.02pt}]  (220,30) -- (260,30) ;
%Straight Lines [id:da6963623809645355] 
\draw [color={rgb, 255:red, 0; green, 0; blue, 0 }  ,draw opacity=1 ][line width=2.25]  [dash pattern={on 2.53pt off 3.02pt}]  (220,60) -- (260,60) ;
%Straight Lines [id:da6705623288576883] 
\draw [color={rgb, 255:red, 0; green, 0; blue, 0 }  ,draw opacity=1 ][line width=2.25]    (470,180) -- (510,180) ;
%Straight Lines [id:da38794033799069494] 
\draw [color={rgb, 255:red, 0; green, 0; blue, 0 }  ,draw opacity=1 ][line width=2.25]    (470,200) -- (510,180) ;

% Text Node
\draw (118,60) node [anchor=east] [inner sep=0.75pt]  [font=\LARGE] [align=left] {$\displaystyle v$};
% Text Node
\draw (118,180) node [anchor=east] [inner sep=0.75pt]  [font=\LARGE] [align=left] {$\displaystyle u$};
% Text Node
\draw (522,180) node [anchor=west] [inner sep=0.75pt]  [font=\LARGE] [align=left] {$\displaystyle b$};
% Text Node
\draw (522,60) node [anchor=west] [inner sep=0.75pt]  [font=\LARGE] [align=left] {$\displaystyle a$};
% Text Node
\draw (39.5,112.5) node  [font=\Huge] [align=left] {$\displaystyle H$};
% Text Node
\draw (190,110.5) node [anchor=west] [inner sep=0.75pt]  [font=\LARGE,color={rgb, 255:red, 208; green, 2; blue, 27 }  ,opacity=1 ] [align=left] {$\displaystyle h$};
% Text Node
\draw (198.5,70.5) node  [font=\LARGE,color={rgb, 255:red, 208; green, 2; blue, 27 }  ,opacity=1 ] [align=left] {$ $};

\end{tikzpicture}
    \caption{Switching used in the proof of Lemma \ref{lem:prob-bad-edge}. Here, the three types of ``bad'' edges are all illustrated: an edge not in $E(H)$ from $v$ back into $V(H)$, or onto a degree $1$ or $2$ vertex. Vertex $b$ has degree at least $3$.}
    \label{fig:3.8}
\end{figure}
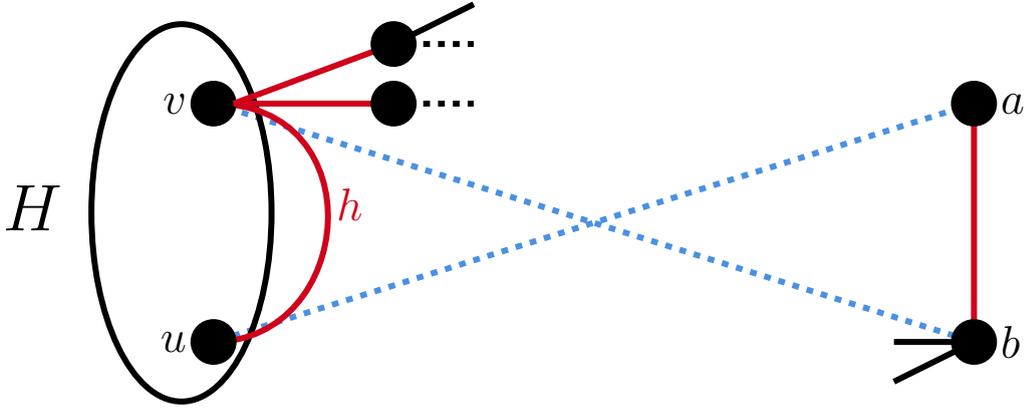
Now, it remains to show that (1) $X_i^*$ is a good proxy for $X_i-X_{i-1}$, and (2) lower bound  $\Cexp{X_i^*}{H_{i-1}}$. Our proofs of both (1) and (2) rely on the fact that the probability that a single open half-edge gets matched either onto a vertex of degree no more than $2$, or onto $V(H_{i-1})$, is small. This fact is shown by the following lemma.
\begin{lem}\label{lem:prob-bad-edge}
Fix $0<\eps<10^{-6}$. Let $\cD$ be such that $m$ is large enough and $n_1,n_2 \le 10^{-6}m$. Let $H$ be a possible subgraph of $G(\cD)$ and let $h$ be a half-edge out of $v\in V(H)$ and not matched in $E(H)$. Suppose $H$ is such that the following hold:
\begin{enumerate}
    \item $\sum_{w\in V(H)}d(w)\le \epsilon m$, and
    \item Either $D^*\leq\eps m$, or $d(w)<\sqrt{m}(\log{m})^{-2}$ for all $w \in V(H)$.
\end{enumerate}
Let $\cF$ denote the event that $h$ is matched onto $V(H)$ or onto a vertex of degree at most $2$. Let $\cA$ denote the event that no neighbour of $v$ outside of $V(H)$ has degree $10d(v)$ or more. Then $\Cprob{\cF\cap\cA}{H\subseteq G(\cD)}<0.001$. If $D^*\leq\eps m$, then $\Cprob{\cF}{H\subseteq G(\cD)}<0.001$.
\end{lem}
\begin{proof}
    We consider the switching from $\cF\cap\cA$ (or from $\cF$ if $D^*\leq\eps m$) into $\cF^\complement$ swapping $\{vu,ab\}$ with $\{vb,au\}$, where $h$ is matched onto $u$, which is either in $V(H)$ or has degree at most $2$, and $ab$ is such that $b$ has degree at least $3$ and is not in $V(H)\cup N(v)$, and $a$ is neither     $u$ nor a neighbour of $u$ (see  Figure \ref{fig:3.8}). Note that $vu,ab$ are not in $E(H)$ so $H$ being a subgraph is preserved.\par
    Let $G\in\cF\cap\cA$ (or $G\in\cF$ if $D^*\leq\eps m$) such that $H\subseteq G$. For the choice of $ab$, {we exclude the edges in $E(G)\setminus E(H)$ which are incident to $V(H)\cup N(v)$ or to a vertex of degree at most $2$, or have both endpoints in $N(u)\cup\{u\}$. Orient any non-excluded edge $ab$ such that $a\notin N(u)\cup\{u\}$; then $b$ is not in $V(H)\cup N(v)$ and has degree at least $3$, so such $ab$ is a valid choice for switching with $vu$. There are:
\begin{itemize}
    \item at most $\sum_{w\in V(H)}d(w)\leq\eps m$ edges incident to $V(H)$;
    \item at most $\sum_{w\in N(v)}d(w)$ edges incident to $N(v)$;
    \item at most $\max\left(\binom{d(u)+1}{2},\sum_{w\in N(u)}d(w)\right)$ edges inside $N(u)\cup\{u\}$; and
    \item at most $n_1+2n_2<10^{-5}m$ edges incident to a vertex of degree at most $2$.
\end{itemize}
Now, if $D^*\leq\eps m$, then $\sum_{w\in N(v)}d(w),\sum_{w\in N(u)}d(w)\leq D^*\leq\eps m$.} Otherwise  $d(w)<\sqrt{m}(\log{m})^{-2}$ for all $w\in V(H)$, so since $u\in V(H)$ or has degree at most $2$, $d(v),d(u)<\sqrt{m}(\log{m})^{-2}$. Furthermore, in this case $G\in\cA$ so every $w\in N(v)$ has degree at most $10d(v)$, so $\sum_{w\in N(v)}d(w)\leq 10(d(v))^2<10m(\log{m})^{-4}<\eps m$ and $\binom{d(u)+1}{2}<\eps m$. In both cases, we forbid at most $(3\eps+10^{-5})m$ edges from $E(G)\setminus E(H)$. So there at least $(1-4\eps-10^{-5})m$ choices of $\{vu,ab\}$ to switch from $\cF$ to $\cF^\complement$.\par
From $\cF^\complement$ to $\cF$, since $au$ is an edge incident to $V(H)$ or to a vertex of degree at most $2$, there are at most $n_1+2n_2+\sum_{w\in V(H)}d(w)<(10^{-5}+\eps)m$ choices of $\{vb,au\}$. Thus
\begin{equation*}
    \Cprob{\cF\cap\cA}{H\subseteq G(\cD)}\leq\Cprob{\cF^\complement}{H\subseteq G(\cD)}\frac{(10^{-5}+\eps)m}{(1-4\eps-10^{-5})m}<0.001
\end{equation*}
and similarly $\Cprob{\cF}{H\subseteq G(\cD)}<0.001$ if $D^*\leq\eps m$.
\end{proof}
\begin{cor}\label{lem:exp-x-star}
Fix $0<\eps<10^{-6}$. Let $\cD$ be such that $m$ is large enough and $n_1,n_2  \le 10^{-6}m$. Let $v\in[n]$ be fixed. If, for $i\geq1$, after exposing the first $i-1$ iterations, we have: 
\begin{enumerate}
    \item $X_{i-1} \neq 0$, 
    \item $2|E(H_{i-1})|+|X_{i-1}| \le \epsilon m$, 
    \item Either $D^*\leq\eps m$ or $d(w)<\sqrt{m}(\log{m})^{-2}$ for all $w \in V(H_{i-1})$
\end{enumerate}
then $\Cexp{X_i^*}{H_{i-1}}\geq 0.9d_i(v_i)$.
\end{cor}
\begin{proof}
We condition on the first $i-1$ iterations such that (1), (2), (3) hold, and the choice of $v_i$ (which are completely determined by $H_{i-1}$) and work within this space.\par
If $d_i(v_i)\geq\sqrt{m}{(\log m)^{-2}}$ then 
$X_i^*=0.9d_i(v_i)$ by the definition of $X_i^*$  so we assume this is not the case. We  need to show $\E{\min(3d_i(v_i),X_i-X_{i-1})}\geq0.9d_i(v_i)$. We let $\Bad$ be the event that $v_i$ is not adjacent to a vertex outside 
    $H_{i-1}$ of degree $5d_i(v_i)$. For each unexplored half-edge $h\in S$ where $S$ consists of the  $d_i(v_i)$ unexplored half-edges incident to $v_i$,  we let $E_h$  be the event that $h$ is adjacent to a 
    vertex of degree at most 2 or a vertex of $H_{i-1}$. We note that $d_i(v_i)-\sum_{h \in S}3(\I{E_h} )\le\min( 3d_i(v_i),X_i-X_{i-1})$ deterministically. When $\Bad$ fails, $\min( 3d_i(v_i),X_i-X_{i-1})\geq 3d_i(v_i)$. Thus, the expectation is at least
\begin{align*}
    &3d_i(v_i)(1-\p{\Bad})+
    d_i(v_i)\p{\Bad} -\sum_{h \in S} 3\p{E_h \cap \Bad}\\
    \ge\,&d_i(v_i)-\sum_{h \in S} 3\p{E_h \cap \Bad}.
\end{align*}

    Since $|S|=d_i(v_i)$, it suffices to show for each $h \in S$, $\p{E_h \cap \Bad}<0.1$; and we can see that this holds by applying the first bound of Lemma \ref{lem:prob-bad-edge}.
\end{proof}

\begin{cor}\label{cor:no-decrease-when-log}
Fix $0<\eps<10^{-6}$. Let $\cD$ be such that $m$ is large enough and $n_1,n_2  \le 10^{-6}m$. Let $v\in[n]$. Then for any $1 \le i \le m$, if $i$ is such that $d_i(v_i)>\sqrt{m}(\log m)^{-2}$ and 
    \begin{enumerate}
    \item $X_{i-1} \neq 0$,
    \item $2|E(H_{i-1})|+|X_{i-1}| \le \epsilon m$, and
    \item  $D^*\leq\eps m$,
\end{enumerate}
then $\Cprob{X_i-X_{i-1}< 0.9d_i(v_i)}{H_{i-1}}<m^{-m^{1/3}}$.
\end{cor}
\begin{proof}
Let $H'_0=H_i$. For $1\leq j\leq d_i(v_i)$, let $h_j$ be the lowest index half-edge on $v_i$ out of $H'_{j-1}$, and form $H'_j$ from $H'_{j-1}$ by adding the edge containing $h_j$. Let $\cF_j$ denote the event that $h_j$ is matched onto $V(H_{j-1}')$ or onto a vertex of degree at most $2$. By Lemma~\ref{lem:prob-bad-edge}, if (1) (2) (3) hold then $\Cprob{\cF_j}{H_{j-1}'}<0.001$. Let $\cB_i$ denote the event that at least $0.01d_i(v_i)$ events among $\cF_1,\dots,\cF_{d_i(v_i)}$ occur. Then
\begin{align*}
    \Cprob{X_i-X_{i-1}< 0.9d_i(v_i)}{H_{i-1}}\leq\Cprob{\cB_i}{H_{i-1}}&<\binom{d_i(v_i)}{0.01d_i(v_i)}0.001^{0.01d_i(v_i)}\\
    &<\left(\frac{0.001e}{0.01}\right)^{0.01d_i(v_i)}<m^{-m^{1/3}},
\end{align*}
the final bound holding when $d_i(v_i)>\sqrt{m}(\log{m})^{-2}>100m^{1/3}\log_{10/e}{m}$.
\end{proof}
We are now ready to combine our lemmas to prove if $v$ is in a linear sized component then $X_i$ will reach $\Omega(m)$ or the exploration will reach a high degree vertex.
\begin{cor}
\label{69}
    Fix $0<\eps<10^{-6}$ and let $\cD$ be such that $m$ is large enough and $n_1,n_2 \le 10^{-6} m$.
    Let $v\in[n]$ be fixed. Then the probability that there exists an $i$  with  $1\leq i \leq m$  such that  $2|E(H_{i-1})|+|X_{i-1}| \leq \eps m$, and such that neither
\begin{enumerate}[(a)]
    \item $X_i \ge 0.9\sum_{j=1}^i d_j(v_j)-\frac{m}{\log m}$, nor
    \item  $D^*\ge \epsilon m$ and 
    $H_i$ contains a vertex of degree at least $\sqrt{m}(\log m)^{-2}$
\end{enumerate}
holds, is at most $e^{-m^{1/11}}$. 
\end{cor}
\begin{proof}

Unless there is some $j  \le i$ such that $X^*_j>X_j-X_{j-1}$ then by Corollary \ref{lem:exp-x-star} we have
\begin{equation*}
    X_i \ge\sum_{j=1}^i X^*_j =Y_i+\sum_{j=1}^i\E{X^*_j | H_{j-1}} \ge Y_i+0.9 \sum_{j=1}^i d_j(v_j).
\end{equation*}
So, by our observation,  the probability we are bounding is at most
\begin{equation*}
    \p{Y_i<\frac{- m}{\log m}}+ \p{\exists 1 \le j \le i:X^*_j>X_j-X_{j-1}}.
\end{equation*}
Lemma \ref{lem:martingale}  implies $\p{Y_i<\frac{- m}{\log m}}<e^{-m^{1/10}}$, and Corollary \ref{cor:no-decrease-when-log}  and the definition of $X^*_j$ together imply  $\p{\exists 1 \le j \le i:X_j^*>X_j-X_{j-1}} \le m( m^{-m^{1/3}})$.
\end{proof}

\begin{cor}
\label{goodtogo}
    The probability that there is a vertex $v$ in a component with at least $\epsilon m$ edges for which both
\begin{enumerate}[(a)]
    \item $X_i$ 
    is never more than $\frac{\epsilon m}{8}$, and 
\item $D^*<\epsilon m$ or  the component contains no vertex of degree greater than $\sqrt{m}(\log m)^2$
\end{enumerate}
hold, is at most $me^{-m^{1/11}}$. 
\end{cor}

\begin{proof}
   Note that  $|E(H_{i-1})|$ is at most $\sum_{j=1}^{i-1} d_j(v_j)$ for all $i\geq1$.
    We consider the first $i$ for which $\sum_{j=1}^i d_j(v_j) \ge \frac{\epsilon m}{4}$.
     Observe that if   $2|E(H_{i-1})|+X_{i-1} \ge  \epsilon m$, (a) fails. 
    Applying Corollary \ref{69} we are done. 
\end{proof}

We are now in a position to prove the main result of this section, Lemma \ref{theyjoinmain}. 

\begin{proof}[Proof of Lemma \ref{theyjoinmain}]
First suppose $D^*<\eps m$. By Corollary \ref{goodtogo}, with probability at least $1-me^{-m^{1/11}}$, every vertex in a component with at least $\eps m$ edges has its $X_i$ reach $\frac{\eps m}{8}$, so by Lemma \ref{theyjoin} these vertices are also all in the same component with probability at least $1-m^{-7\log\log{m}+2}$.\par
Now suppose $D^*\geq\eps m$ but $d(n)\geq\sqrt{m}(\log{m})^6$. By Corollary \ref{sec6cor1}, all vertices of degree at least $\sqrt{m}(\log{m})^{-2}$ are in the same component with probability at least $1-m^{-8\log\log{m}}$. If $D^*\geq\eps m$ but $d(n)<\sqrt{m}(\log{m})^6$, then by Lemma \ref{bigguyssticktogether}, again all vertices of degree at least $\sqrt{m}(\log{m})^{-2}$ are in the same component with probability at least $1-m^{-8\log\log{m}}$. By Corollary \ref{goodtogo}, with probability at least $1-me^{-m^{1/11}}$, every vertex in a component with $\epsilon m$ edges is either in a component with a vertex of degree greater than $\sqrt{m}(\log{m})^{-2}$, or $X_i$ reaches $\frac{\eps m}{8}$. Finally, by Lemma \ref{theyjoin}, the probability that there exists a vertex whose exploration reaches $\frac{\eps m}{8}$ in a different component from any vertex of degree $\sqrt{m}(\log{m})^{-2}$  or from another vertex whose exploration also reaches $\frac{\eps m}{8}$  is at most $m^{-7\log\log{m}+2}$.
\end{proof}

\section{Analyzing  Iteration by Iteration}
\label{seciteration}
In this section we carry out our second angle of analysis  of the exploration $(X_i)_{i\geq0}$ out of an arbitrary fixed vertex $v\in[n]$, as defined in Section \ref{sec:outline}. We consider 
each iteration in turn and
focusing on the  the random variable $X_i-X_{i-1}$ especially on its behaviour when it is negative. 
The results we obtain will be used
both to complete the proof of Theorem \ref{thekeyresult}, which is done in the next   section, 
and to bound the probability that there is a component of size at most $4(\log m)^4$ which we do in the  section after that.

We recall that $X_i-X_{i-1}$ is at least $d_i(v_i)-2J_i-K_i-3L_i$. Thus our focus is on the behaviour of $J_i$, $K_i$, and $L_i$. 
We handle $L_i$ separately from $J_i$ and $K_i$. 
Given an iteration $i\geq1$ and $j,k\geq0$ integers, define $\Bad_{j,k}^{(i)}$ to be the event that $X_i \le X_{i-1}$, $J_i\geq j$,  and $K_i\geq k$. 

We note that if $X_i \le X_{i-1}$ then  $J_i+K_i+L_i\geq\frac{d_i(v_i)}{3}$, which implies either $L_i \ge \frac{d_i(v_i)}{6}$ or $\Bad_{j,k}^{(i)}$ occurs for some nonnegative $j,k$ which sum to at least $\frac{d_i(v_i)}{6}$. We will now bound the probability of $\Bad^{(i)}_{j,k}$ conditioned on $H_{i-1}$.

\begin{lem}\label{lem:boundary-grows2}
    Let $0<\eps<10^{-6}$ and assume $\cD$ is such that $m$ is large enough and that $n_1,n_2<\epsilon m$. Suppose that for some  $i\geq1$, $|E(H_{i-1})| \le \frac{m}{10^{5}}$ and   $X_{i-1}<10^{-4}m$.  
    Then for any nonnegative $j,k$ with $j+k \le \frac{\sqrt{m}}{1000}$, the probability, conditioning on $H_{i-1}$, of the event $\Bad_{j,k}^{(i)}$ is at most
\begin{equation}\label{eq:jk}
\frac{jm}{jm-2ed_i(v_i)n_1}\left(\frac{2ed_i(v_i)n_1}{jm}\right)^{j}\frac{km}{km-10ed_i(v_i)n_2}\left(\frac{10ed_i(v_i)n_2}{km}\right)^{k}
\end{equation}
if $j,k>0$, $2ed_i(v_i)n_1<jm$, and $10ed_i(v_i)n_2<km$; at most
\begin{equation}\label{eq:k}
  \frac{km}{km-10ed_i(v_i)n_2}\left(\frac{10ed_i(v_i)n_2}{km}\right)^{k}
\end{equation}
if $j=0$ and $10ed_i(v_i)n_2<km$; and at most
\begin{equation}\label{eq:j}
  \frac{jm}{jm-2ed_i(v_i)n_1}\left(\frac{2ed_i(v_i)n_1}{jm}\right)^{j}
\end{equation}
if $k=0$ and $2ed_i(v_i)n_1<jm$.
Furthermore, for any $\ell\geq1$, the probability, conditioning on $H_{i-1}$, of the event $\Bad_{j,k}^{(i)}\cap\{L_i\geq\ell\}$ is at most 
\begin{equation}\label{eq:l}
p^{(i)}_{j,k}\left(\frac{2d_i(v_i)X_{i-1}}{\ell m}\right),
\end{equation}
where $p^{(i)}_{j,k}$ (depends on $j$, $k$, and $H_{i-1}$ via $d_i(v_i)$) is equal to the minimum of $1$ and \eqref{eq:jk}, \eqref{eq:k}, or \eqref{eq:j} if their respective assumptions are satisfied.
\end{lem}

\begin{proof}
Fix $i\geq1$ where $|E(H_{i-1})|\leq\frac{m}{10}$ and $X_{i-1}<10^{-4}m$ and denote $\Bad^{(i)}_{j,k}$. For every $j',k'\geq0$ we define ${\cal F}_{j',k'}$ to be the event that $X_i \le X_{i-1}<10^{-4}m$, $E(H_{i-1})\leq\frac{m}{10}$, $J_i=j'$, $K_i=k'$.\par
For every $(j',k')$ with $j'\ge j, k'\ge k$, for each $(p,q)$ with $0 \le p \le j',0\le  q \le k'$, and $p+q\le \frac{\sqrt{m}}{1000}$, we let $\cA_{p,q}={\cal A}_{p,q}^{(j',k')}$ be the event that $V(H_i)\setminus V(H_{i-1})$ (i.e.~the neighbours of $v_i$ added to $H$ during iteration $i$) has total degree at most $20(p+q)\sqrt{m}+2d_i(v_i)$, that $J_i=j'-p$, and that $K_i=k'-q$. We note  that $X_i\leq X_{i-1}$ implies $V(H_i)\setminus V(H_{i-1})$ has total degree at most $2d_i(v_i)$, so ${\cal F}_{j',k'} \subseteq\cA_{0,0}^{(j',k')}$.

For every graph in $\cA_{p,q}$, we say an edge is a \emph{valid switch edge} if the following hold:
it has  one endpoint of degree  at most $10\sqrt{m}$, and  both of its endpoints are not neighbours of $v_i$, have degree at least three, and are not in $V(H_{i-1})$. For a graph for which $\cA_{p,q}$ occurs, there are at most:
\begin{itemize}
    \item $\frac{m}{50}$ edges both of whose  endpoints have degree exceeding  $10\sqrt{m}$;
    \item $20(p+q)\sqrt{m}+2d_i(v_i) \le \frac{m}{25}$  edges out of neighbours of $v_i$ outside of $V(H_{i-1})$;
    \item $n_1+2n_2 \le 10^{-5}m$ edges incident to vertices of degree less than 3; and
     \item $\frac{m}{10^5} +10^{-4}m <\frac{m}{500}$ edges  incident to $V(H_{i-1})$.
\end{itemize}
Thus there are at least $\frac{7m}{8}$ valid switch edges in any such $G$. Also note that
\begin{equation*}
    2m=n_1+2n_2+3(n-n_1-n_2)=3n-2n_1-n_2>3n-3\eps m\implies n<\frac{(2+3\eps)m}{3}.
\end{equation*}
\par
For every $(j',k'),(p,q)$ as specified, with $0 \le p \le j'-1 $, we consider switchings from $\cA_{p,q}$ to $\cA_{p+1,q}$ using switchings swapping $\{v_iu,ab\}$ with $\{v_ib,au\}$, where $d(u)=1$ and $ab$ is a valid switch edge, oriented so that   For any $G\in\cA_{p,q}$, there are at least $(j'-p)\frac{7m}{8}$ such switchings. Now, there are no more than $d_i(v_i)n_1$ choices of $\{v_ib,au\}$ in $\cA_{p+1,q}$. Hence
\begin{equation*}
    \frac{|\cA_{p,q}|}{|\cA_{p+1,q}|}\leq\frac{2d_i(v_i)n_1}{(j'-p)m}.
\end{equation*}
So $\p{\cA_{0,q}}\leq\frac{1}{j'!}(\frac{2d_i(v_i)n_1}{m})^{j'}\leq(\frac{2ed_i(v_i)}{j'm})^{j'}$. Very similarly, for every $(j',k'),(p,q)$ as specified, with $0\leq q\leq k'-1$, we consider switchings from $\cA_{p,q}$ to $\cA_{p,q+1}$ using switchings swapping $\{v_iu,ab\}$ with $\{v_ib,au\}$, where $d(u)=2$ and $ab$ is a valid switch edge such that neither of $a,b$ is the other neighbour of $u$, oriented so that $d(b)\leq 10\sqrt{m}$. Since less than $n<\frac{(2+3\eps)m}{3}$ edges are incident to the other neighbour of $u$, we get at least $(k'-q)(\frac{7m}{8}-n)>(k'-q)\frac{m}{5}$ switchings from $\cA_{p,q}$ and at most $d_i(v_i)2n_2$ switchings from $\cA_{p,q+1}$. Thus 
\begin{equation*}
    \frac{|\cA_{p,q}|}{|\cA_{p,q+1}|}\leq\frac{10d_i(v_i)n_2}{(k'-q)m}.
\end{equation*}
So $\p{\cA_{p,0}}\leq\frac{1}{k'!}(\frac{10d_i(v_i)n_1}{m})^{k'}\leq(\frac{10ed_i(v_i)}{k'm})^{k'}$. Thus, all conditioning on $H_i$, 
\begin{equation*}
\p{\cF_{j',k'}}\leq\p{\cA_{0,0}^{(j',k')}}\leq\begin{cases}\left(\frac{2ed_i(v_i)n_1}{j'm}\right)^{j'}\left(\frac{10ed_i(v_i)n_2}{k'm}\right)^{k'},&j',k'>0;\\
\left(\frac{10ed_i(v_i)n_2}{k'm}\right)^{k'},&j'=0;\\
\left(\frac{2ed_i(v_i)n_1}{j'm}\right)^{j'},&k'=0.
\end{cases}
\end{equation*}
Union bounding over all $j'\geq j$ then over all $k'\geq k$, we get (when the fractions in parentheses are strictly less than one, as in the lemma hypotheses)
\begin{equation*}
\p{\Bad^{(i)}_{j,k}}\leq
\begin{cases}\frac{jm}{jm-2ed_i(v_i)n_1}\left(\frac{2ed_i(v_i)n_1}{jm}\right)^{j}\frac{km}{km-10ed_i(v_i)n_2}\left(\frac{10ed_i(v_i)n_2}{km}\right)^{k},&j,k>0;\\
\frac{km}{km-10ed_i(v_i)n_2}\left(\frac{10ed_i(v_i)n_2}{km}\right)^{k},&j=0;\\
\frac{jm}{jm-2ed_i(v_i)n_1}\left(\frac{2ed_i(v_i)n_1}{jm}\right)^{j},&k=0.
\end{cases}
\end{equation*}
Finally, for any $\ell\geq1$, we consider swaps from graphs in $\cA_{j',k'}\cap\{L_i=\ell\}$ into graphs in $\{L_i=\ell-1\}$, conditioning on $H_i$. We can use any pair consisting of
 one of the $\ell$ direct back edges and one of the at least $\frac{m}{2}$ valid switch edges to switch from $\Bad^{(i)}_{j,k}\cap\{L_i=\ell\}$, and have at most $d_i(v_i)X_i$ switches from $\{L_i=\ell-1\}$. Thus if $L_i\geq1$ then we get an additional factor of $\frac{2d_i(v_i)X_i}{\ell m}$, as claimed.
\end{proof}

\begin{cor}
\label{iterationanalysiscor}
     For all $0<\gamma<1$ and $0<\epsilon<10^{-8}$,  if $m$ is large enough and $n_1 \le m^{1-\gamma}$ while $n_2\leq\eps m$ and $2|E(H_i)|+X_i \le \epsilon m$, then
\begin{equation*}
    \Cprob{X_i<X_{i-1}-12\sqrt{X_{i-1}}}{H_{i-1}} \le  \left(\frac{d_i(v_i) n_1}{10\sqrt{X_{i-1}}m}\right)^{10\sqrt{X_{i-1}}}.
\end{equation*}
\end{cor}

\begin{proof} 
We have $X_i-X_{i-1} \ge d(v_i)-2J_i-K_i-3L_i \ge -J_i-2L_i$ and  $L_i\le\sqrt{X_{i-1}}$. 
 So if $ X_i<X_{i-1}-12\sqrt{X_i}$ then $J_i\ge 10\sqrt{|X_i|}$. Lemma \ref{lem:boundary-grows2} implies that the probability this occurs is at most $(\frac{d_i(v_i) n_1}{10\sqrt{X_i}m})^{10 \sqrt{X_i}}$ .
\end{proof}

Techniques similar to those used in the proof of Lemma \ref{lem:boundary-grows2}  can be pushed further when we are considering edges between low degree vertices. Specifically, defining 
\begin{equation}\label{eq:delta-star}
    \delta^*=\min\left(10^6,\max\left(k\geq1:\sum_{j=1}^kn_j\leq\frac{n}{10^{24}}\right)\right),
\end{equation}
they allow us to prove the following lemma.

\begin{lem}
\label{lem:boundary-grows3}
    Fix $0<\eps<\frac{1}{100}$.   Assume $H_{i-1}$ is such that  $X_{i-1}+2|E(H_{i-1})|\le 10^6$. Then  for any $(\delta^*+1)$-tuple $(s_1,\dots,s_{\delta^*+1})$, the probability conditioned on $H_{i-1}$ that $v_i$ has $s_j$ neighbours outside $H_{i-1}$ of degree $j$ for each $1\leq j\leq\delta^*+1$  is $O(\prod_{j=1}^{\delta^*+1} (\frac{n_j}{m})^{s_j})$. If $d_i(v_i) \le \delta^*+1$  then the probability that this occurs and $L_i=\ell$ is
    $O( \prod_{j=1}^{\delta^*+1} (\frac{n_j}{m})^{s_j} (\frac{1}{m})^\ell)$.
\end{lem} 
\begin{proof}

    We note that our bound on $X_i$ implies that the probability we are bounding is $0$  unless the sum of the $s_j$ is at most $10^6$. So, if $n_{\delta^*+1} >\frac{m}{10^{18}}$ then  $(\frac{n_{\delta^*+1}}{m})^{s_{\delta^*+1}}=\Omega(1)$.
    So, in this case,  the lemma is equivalent to the statement obtained from it  by replacing the upper bound on $j$ by
    $\delta^{*}$. We let $\delta'$ be $\delta^*$ in this case, and otherwise we set $\delta'=\delta^*+1$. 
    So we need only prove the statement obtained from the lemma   by replacing the upper bound on $j$ by
    $\delta'$.\par
    Let $B$ be the set of vertices of degree at least $\delta'+1$. By our choice of $\delta^*$, there are at most $ \frac{n}{10^{12}}<\frac{m}{10^{12}}$
    edges incident to vertices of degree at most $\delta^*$, in particular when $\delta'=\delta^*$. By our choice of $\delta'$, if $\delta'=\delta^*+1$ then there are at most $\frac{m}{10^{12}}$
    edges incident to vertices of degree  $\delta'$. So the sum of the degrees of the vertices  
    in $B$  is at least $\frac{199m}{100}$. For any set $S$ of at most $\delta'$ vertices and large $m$,
    the sum of the degrees of the vertices in $S$ is at most $\binom{\delta'}{2}+m<\frac{26m}{25}$. Hence the sum of the degrees of the vertices 
    of $B\setminus S$ is at least $\frac{4m}{5}$. 
    Every vertex in $B\setminus S$ has  degree exceeding $\delta'$ and hence at least one neighbour outside of $S$, so 
    the number of edges which have one endpoint in $B\setminus S$ and the other in $V\setminus S$ is at least $|B\setminus S|\geq\frac{4m}{5(\delta'+1)}>\frac{m}{10^7}$. By our upper bound $\frac{m}{10^{12}}$ on the number of edges incident to $V\setminus B$
    we see there are $\Omega(m)$ edges within $B\setminus S$. 
        
    Fix $H_{i-1}$ as in the statement of the lemma. For any $(\delta'+1)$-tuple $(\ell_1,\dots,\ell_{\delta'},\ell)$ where $0\leq \ell_j\leq s_j$ for each $1\leq j\leq\delta'$ and $0\leq\ell\leq d_i(v_i)$, let $\cF_{(\ell_1,\dots,\ell_{\delta'},\ell)}$ denote the event that $H_{i-1}$ is explored and $v_i$ has $\ell_j$ neighbours outside $H_{i-1}$ of degree $j$ for each $1\leq j\leq\delta' $, and $L_i=\ell$.

    We consider first switchings which  reduce the value of some $\ell_j\geq1$ for some $j \le \delta'$. 
    We consider an edge $v_iw$ where $w$ is outside $H_{i-1}$ and has degree $j$. We set  $S=N(w)-v$ and note $|S| \le \delta'-1$. As noted above, there are  $\Omega(m)$ edges within  $B\setminus S$. At most $10^{12}$ have both their endpoints in $N(v_i)$ 
    and at most $2(10)^{12}$ have an endpoint in $V(H_{i-1})$. So there are $\Omega(m)$ valid switches using  $v_iw$ and an $xy$ 
    where $x \in B\setminus (S\cup V(H_{i-1}))$ and $y \in B\setminus(S\cup N(v_i)\cup V(H_{i-1}))$. Since $d_i(v_i),d(w) \le 10^6$, there are  $d_i(v_i)jn_j=O(n_j)$ swaps in the opposite direction. Letting $L=(\ell_1,\dots,\ell_{\delta^*},\ell)$ and $L'$ denote $L$ with $\ell_j$ replaced by $\ell_j-1$, then
\begin{equation*}
    \p{\cF_L}\leq\p{\cF_{L'}}\frac{O(n_j)}{\Omega(m)}
\end{equation*}
    and the first statement follows.\par
    We turn to the second statement, so we can assume $d_i(v_i) \le \delta^*+1$. 
    We consider switchings which  reduce  $L_i=\ell\geq1$. We let $v_iw$ be an edge of $G$ with $w \in V(H_{i-1})$.
    We let $S$ be the neighbours of  $v_i$ outside $V(H_{i-1})$. So, $|S| \le \delta'$. 
    There are  $\Omega(m)$ edges within  $B\setminus S$. At most $10^{12}$ have both their endpoints in $N(w)$ 
    and at most $2(10)^{12}$ have an endpoint in $V(H_{i-1})$. So there $\Omega(m)$ valid switches using  $v_iw$ and an $xy$ 
    where $x \in B\setminus(S\cup V(H_i))$ and $y \in B\setminus(S\cup N(w)\cup V(H_{i-1}))$. Since $X_{i-1}\le 2(10)^{12}$, there are $X_{i-1}^2=O(1)$ swaps in the opposite direction. Letting $L=(\ell_1,\dots,\ell_{\delta^*},\ell)$ and let $L'$ denote $L$ with $\ell$ replaced by $\ell-1$, then
\begin{equation*}
    \p{\cF_L}\leq\p{\cF_{L'}}\frac{O(1)}{\Omega(m)}
\end{equation*}
    and the second statement follows.
\end{proof}

\section{Analyzing The  Later Stages of The Exploration}
\label{seclater}
In this section, we  consider the behaviour of the exploration after it has survived for some $C=C(m)$ iterations. By the end of the section we  will have developed enough machinery to prove Theorem \ref{thekeyresult}.

We first  prove:

\begin{lem}\label{lem:sum-J}
    For any $m$ large enough, $0<\gamma<1$, $0<\epsilon<10^{-10}$, $n_1 \le m^{1-\gamma}$, $n_2<\eps m$, $0<C \le m^{\gamma /20}$, $1\leq i$, $0\leq r\le C(\log m)^2$, the following holds: if $2|E(H_{i-1})|+X_{i-1} \le \epsilon m$ then  
\begin{equation*}
    \Cprob{\sum_{t=i}^{i+r-1} J_t\I{X_t\leq X_{t-1}} \ge \frac{C}{3} \text{ and }\forall i\leq t\leq i+r-1: X_{t-1}\le 2C}{H_{i-1}} <m^{-\gamma C/4}.
\end{equation*}
\end{lem}
\begin{proof}
If the event whose conditional probability we aim to bound occurs, we can choose one of a set of at most $r^{C/3}$  sequences of nonnegative integers $J'_t$ which sum to $\frac{C}{3}$  such that for each $i \le t \le i+r-1$, $J'_t \le J_t$, and $J'_t>0$ only when $X_t \le X_{t-1}$. 
We can choose such a set by repeatedly increasing  some $J'_t$ which is less than $J_t$. For any $t$  with $i \le t \le i+r-1$ we have by the hypotheses  and that fact that $d_s(v_s)<X_s$ that $2|E(H_{t-1})|+X_{t-1} \le 2|E(H_{i-1})| +2\sum_{s=i}^{t-1} d_s(v_s) +X_{t-1}  \le \frac{m}{10^{5}}$. We may thus apply the bound \eqref{eq:j} from Lemma \ref{lem:boundary-grows2} to the positive $J'_t\leq\frac{C}{3}\leq\sqrt{m}/1000$. Doing so, and then applying our bounds on $r$, $n_1$ and $C$, and the fact $d_t(v_t)\leq X_{t-1}<2C$, we obtain that 
\begin{align*}
&\Cprob{\sum_{t=i}^{i+r-1} J_t\I{X_t\leq X_{t-1}} \ge \frac{C}{3} \text{ and }\forall i\le t \le i+r-1: X_{t-1} \le 2C}{H_{i-1}}\\
\leq\,&\sum_{\{J'_t\}_{t=i}^{i+r-1}}\prod_{t=i}^{i+r-1}\left(\frac{2e(2C)n_1}{J'_tm}\right)^{J'_t}\leq r^{C/3}\left(\frac{2e(2C)n_1}{m}\right)^{C/3}\\
\leq\,&\left(\frac{4e(C\log{m})^2n_1}{m}\right)^{C/3}\leq(m^\frac{-8\gamma}{9})^{C/3}.\qedhere
\end{align*}
\end{proof}

\begin{lem}\label{lem:sum-K}
    For any $m$ large enough, $0<\gamma<1$, $0<\epsilon<10^{-6}$, $n_1<\eps m$, $n_2 \le 10^{-6}m$, $0<C \le m^{\gamma /20}$,  $i \ge 1$, and $0\leq r\le C(\log m)^2$, the following holds: if $2|E(H_{i-1})|+X_{i-1} \le \epsilon m$ then, conditioned on $H_{i-1}$,
\begin{multline*}
    \p{\left|\left\{i \le t \le  i+ r-1:K_t=d_t(v_t)>0\right\}\right|>\frac{ r}{2} \text{ and }\forall i \le t\le i+ r-1:X_{t-1}  \le 2C}\\
    \le 14^{- r}.
\end{multline*}
\end{lem}
\begin{proof} We have $C>0$ by hypothesis and we can assume $n_2>0$ as otherwise every $K_t=0$. 
For any $t$  with $i \le t \le i+ r-1$ we have by the hypotheses that $2|E(H_{t-1})|+X_{t-1} \le 2|E(H_{i-1})|+X_{t-1} +\sum_{s=i-1}^{t-1} d_s(v_s) \le \frac{m}{10^{5}}
$. The bound \eqref{eq:k} from Lemma~\ref{lem:boundary-grows2}  and our hypotheses imply that the probability that $0<K_t =d_t(v_t)\leq 2C<\sqrt{m}/1000$ conditioned on $H_{t-1}$ is at most $\frac{1}{800}$. 
So on the event that $2|E(H_{i-1})|+X_{i-1} \le \epsilon m$ and that for all $i \le t \le i+ r$, $X_{t-1} \le 2C$, the size of the set $\left\{i \le t \le  i+ r-1:K_t=d_t(v_t)>0\right\}$ is stochastically dominated by a Binomial random subset of a set of size $ r$, and thus the expected 
number of subsets of $\{i,\dots,i+ r-1\}$ of size at least $\frac{ r}{2}$ such that for each $t$ in the set  $K_t=d_t(v_t)>0$ is at most $2^ r800^{- r/2}<14^{- r}$. 
\end{proof}

\begin{lem}\label{lem:sum-L}
    For any $m$ large enough, $0<\gamma<1$, $0<\epsilon<10^{-6}$, $n_1,n_2<\eps m$, $C \le m^{\gamma /20}$, $1\leq i$, $ r\le C(\log m)^2$ the following holds, if $2|E(H_{i-1})|+X_{i-1} \le \epsilon m$ then
\begin{align*}
    \Cprob{\sum_{t=i}^{i+r-1} L_t\I{X_t\leq X_{t-1}} \ge \frac{C}{15} \text{ and }\forall i \le t\le i+ r-1:X_{t-1}\leq 2C}{H_{i-1}} 
    &\leq m^{-\sqrt{C}/25}.
\end{align*}
\end{lem}
\begin{proof}
    Note $L_t<\sqrt{X_{t-1}}\leq\sqrt{2C}$ deterministically. Thus
\begin{equation*}
    \sum_{t=i}^{i+r-1} L_t\I{X_t\leq X_{t-1}}\leq\sqrt{2C}\sum_{t=i}^{i+ r-1}\I{X_t\leq X_{t-1}\text{,  }L_t>0}.
\end{equation*}
and we need only bound the probability that $\sum_{t=i}^{i+r-1}\I{X_{t} \leq X_{t-1}\text{, }L_t>0} \le  \frac{\sqrt{C}}{15\sqrt{2}}$.\par
Since $|E(H_{t-1})| \le |E(H_{i-1})| +\sum_{s=i-1}^{t-1} d_s(v_s)$ the hypotheses imply that for 
all $t$ with $i\leq t \le i+r-1$, $|E(H_{t-1})|+X_{t-1} \le \frac{m}{10^5 }$.
So, by the bound \eqref{eq:l} from Lemma \ref{lem:boundary-grows2} with $p^{(i)}_{j,k}\leq 1$, since $d_t(v_t),X_{t-1}\leq 2C$,
\begin{align*}
    &\Cprob{\sum_{t=i}^{i+ r-1}L_t\I{X_t\leq X_{t-1}} \ge \frac{C}{15} \text{ and }\forall i \le t\le i+ r-1:X_{t-1}\leq 2C}{H_{i-1}}\\<\,& r^{\frac{\sqrt{C}}{15\sqrt{2}}}\left(\frac{8C^2}{m}\right)^{\frac{\sqrt{C}}{15\sqrt{2}}}\leq\left(\frac{8C^3(\log{m})^2}{m}\right)^{\frac{\sqrt{C}}{15\sqrt{2}}}<m^{-\sqrt{C}/25} \qedhere
\end{align*}
\end{proof}
\begin{cor}
\label{firstgoodcor}
      For any $m$ large enough, $0<\epsilon<10^{-8}$, $0<\gamma<1$, if $n_1 \le m^{1-\gamma}, n_2 \le 10^{-6}m$, and $3\leq C \le m^{\gamma /20}$, then the probability that for a specific $i$, $|E(H_{i-1})|+X_{i-1} \le \epsilon m$ and $0<X_t<2C$ for all $i \le t<i+C(\log m)^2$ is at most  $m^{-\gamma\sqrt{C}/26}$. 
\end{cor}
\begin{proof}

Recall that $X_{s}-X_{s-1} \ge d_s(v_s)-2J_s-K_s-3L_s$.
So, if $d_s(v_s) \neq K_s$ then either $J_s>0$, $L_s>0$, or $X_s>X_{s-1}$.
Hence letting   $   r =\lfloor  C(\log{m})^2 \rfloor$ and  $Y$ be the number of $s$ with $i \le s \le  i+ r-1$ for which $K_s \neq d_s(v_s)$, we have that  the number of open edges increases by  at  least $Y-\sum_{s=i}^{i+ r-1} J_s\I{X_s\leq X_{s-1}}-\sum_{s=i}^{i+ r-1}L_s\I{X_s\leq X_{s-1}}$ of the iterations between $i$ and  $i+ r-1$. Furthermore, the total decrease over those of these $r$ iterations in which the number of open edges decreases is at most $\sum_{s=i}^{i+ r-1} J_s\I{X_s\leq X_{s-1}}+\sum_{s=i}^{i+ r-1}2L_s\I{X_s\leq X_{s-1}}$. So, $X_{i+ r} \ge X_i+Y-2\sum_{s=i}^{i+ r-1} J_s\I{X_s\leq X_{s-1}}-3\sum_{s=i}^{i+ r-1}L_s\I{X_s\leq X_{s-1}}$.
 
Let $\cA$ denote the event that  $0<d_t(v_t)< 2C$ for all $i \le t \leq i+ r$. 
So we need to show $\p{\cA}  \le m^{-\gamma\sqrt{C}/26}$. 
We note that if $\cA$ holds then  for every $t$ with $i \le t \le i+ r-1$,
$|E(H_t)|+X_t \le \frac{m}{10^{6}}$. 
Applying  Lemma \ref{lem:sum-K}, the probability $\cA$  holds and $Y<3C<\frac{ r}{2}$   is 
at most $14^{- r} <m^{-C}$. Applying  Lemma \ref{lem:sum-J}, the probability $\cA$  holds and
$\sum_{s=i}^{i+ r-1} J_s\I{X_s\leq X_{s-1}}>\frac{C}{3}$ is at most
$m^{-\gamma C/4}$,    
Finally, by Lemma \ref{lem:sum-L}, the probability $\cA$  holds and  $\sum_{s=i}^{i+ r-1} L_s\I{X_s\leq X_{s-1}}>\frac{C}{15}$  is  $m^{-\sqrt{C}/25}$. But if $\cA$ holds while $Y \ge 3C$, $\sum_{s=i}^{i+ r-1} J_s \le \frac{C}{3}$, and $\sum_{s=i}^{i+ r-1}L_s\I{X_s\leq X_{ r-1}} \le \frac{C}{15}$ 
then $X_ r \ge 3C-\frac{2C}{3}-\frac{C}{5} >2C$, contradicting the fact $\cA$ holds; so the probability of this event is $0$.  
\end{proof}

\begin{cor}
\label{secondgoodcor}
     For any $m$ large enough, if $0<\gamma<1$, $0<\eps<10^{-8}$, $n_1 \le m^{1-\gamma}$, $n_2 \le 10^{-6}m$, and $3\leq C \le m^{\gamma /20}$, then the probability given that for a specific $i \ge 1$, $X_{i-1} \ge C$ 
    and there is no $t>i$ such that either $X_t\ge 2C$ or  $2|E(H_{t-1})|+X_{t-1} \ge \epsilon m$ 
    is $2m^{-\gamma\sqrt{C}/26}$. 
\end{cor}
\begin{proof}
Fix iteration $i\geq 1$ such that $X_{i-1} \ge C$ and condition all probabilities on $H_{i-1}$. For $ r = \lceil C(\log m)^2 \rceil$,  we let   
$\cA$ be the event that  for all $t$ with $i \le t \le i + r-1$,  $X_t<2C$ and $2|E(H_{t-1})|+X_{t-1} \le \epsilon m$. So it suffices to show $\Cprob{\cA} {H_{i-1}}  \le m^{-\gamma\sqrt{C}/20}$. 
We note that since ${2|E(H_{i-1})|+X_{i-1}}<\epsilon m$, the assumptions of Lemmas \ref{lem:sum-J}, \ref{lem:sum-K}, \ref{lem:sum-L} and Corollary \ref{firstgoodcor} are satisfied.

We first  compute the probability of the intersection of  $\cA$  and the event $\{X_{i+r-1} \neq 0\}$.
This implies $X_t>0$ for all $i \le t \le i+ r-1$.  
Applying Corollary \ref{firstgoodcor}, we have that the conditional probability of $\cA \cap\{X_ r>0\}$ is at most $m^{-\gamma \sqrt{C}/26}$.

We next   compute the probability of the intersection of  $\cA$  and  $\{X_ r=0\}$. The latter event is contained in the union of $\{\sum_{s=i}^{i+ r-1} J_s\I{X_s\leq X_{s-1}}>\frac{C}{2}\}$
and  $\{\sum_{s=i}^{i+ r-1} 2L_s\I{X_t\leq X_{t-1}} >\frac{C}{2}\}$.
Applying  Lemma \ref{lem:sum-J}, the (conditional)  probability of  $\cA \cap \{\sum_{s=i}^{i+ r-1} J_s\I{X_s\leq X_{s-1}}>\frac{C}{2}\}<m^{\gamma C/4}$. Applying  Lemma \ref{lem:sum-L}, the (conditional) probability 
of the event $\cA \cap \{\sum_{s=i}^{i+ r-1} 2L_s\I{X_t\leq X_{t-1}} >\frac{C}{2}\}$ is $m^{-\sqrt{C}/22}$. So the (conditional) probability of $\cA \cap \{X_ r=0\}$ is 
also at most $m^{-\gamma \sqrt{C}/26}$. 

\end{proof}
\begin{cor}
\label{thirdgoodcor}
     For any $\gamma>0$ and sufficiently small $\epsilon>0$, if $n_1 \le m^{1-\gamma}, n_2 \le 10^{-6}m$ then the probability that $\sum_{t=1}^i d_t(v_t) \ge 4(\log m)^4$ for some $i$,  and  $v$ does not lie in a component with at least $\frac{\epsilon m}{2}$ edges 
    is $o(m^{-(\log m)^{1/3}})$.
\end{cor}
\begin{proof}

We consider the largest $t \le \lfloor \frac{ \gamma  \log m}{20} \rfloor$ such that $X_i \ge 2^t$ for some $i$,
and the largest $i$ for which this holds for our choice of $t$. If $t \le \lfloor \log \log m \rfloor$
then there is no $i$ such that $X_i> 2 \log m$ and hence no $i$ such that $d_i(v_i)>2 \log m$.
In this case, in order for $\sum_{t=1}^i d_t(v_t) \ge 4(\log m)^4$ to occur for some $i$, in particular $X_t$ must remain above 0 for the first  $\lceil 2(\log m)^3\rceil $ iterations without reaching $2 \log m$. Applying Corollary \ref{firstgoodcor} with $i=1$ (note that the role of $i$ in that corollary is not the same as the role of $i$ in this one) and  $C=2 \log m$,
we see that the probability this occurs is  $o(m^{-(\log m)^{1/3}})$. Applying Corollary \ref{secondgoodcor} for $\lfloor   \log \log m \rfloor <t<\lfloor \frac{\gamma \log m}{20} \rfloor$ times, we see that the probability we have $\lfloor   \log \log m \rfloor <t<\lfloor \frac{\gamma \log m}{20} \rfloor$ is $o(m^{-(\log m)^{1/3}})$. 

If $t =\lfloor \frac{\gamma \log m}{20} \rfloor$ then applying  Corollary \ref{secondgoodcor}, to the first $i$ for which $X_i \ge 2^t$,  we see that the probability  we have no  $i'>i$ for which either $X_{i'} \ge 2^{t+1}$ or $2|E(H_{i'-1})|+X_{i'-1} \ge \epsilon m$  is $o(m^{-(\log m)^{1/3}})$.
We further apply Corollary \ref{secondgoodcor}, to every  $i$ for which $2^t \le X_i<2^{t+1}$ and $X_{i-1} \ge2^{t+1}$, combined with the union bound, this yields that 
the probability that there is such an   $i$ for which  there is no  $i'>i$ with $X_{i'} \ge 2^{t+1}$ or $2|E(H_{i'-1})|+X_{i'-1} \ge \epsilon m$  is $o(m^{-(\log m)^{1/3}})$. In particular, this implies we need only bound the probability of the event that there is an $i$ such that $X_i \le 2^{t}$ and $X_{i-1} \ge 2^{t+1}$. Since this means $X_i <X_{i-1}-12\sqrt{X_{i-1}}$, applying Corollary \ref{iterationanalysiscor}, we are done.  
\end{proof}

We can now give the proof of Theorem \ref{thekeyresult}.
\begin{proof}[Proof of Theorem \ref{thekeyresult}]
 We simply combine  $n$ applications of Corollary \ref{thirdgoodcor} with 
 one application of Lemma \ref{theyjoinmain}.
\end{proof}

\section{Analyzing The Start of The Exploration}\label{theearlystages}
\label{secearly} 

We now prove Theorem \ref{main}. 

We recall that the theorem states that for any feasible degree sequence the probability that $G(\cD)$ is disconnected is 
\[
 O\left(u_{edge}+u_{\Delta}+u_{\Delta+1}+u_{K_4-e}+u_{K_4}+u_{K_5^+}
\right).
\]
where 
\begin{gather*}
    u_{edge}=\frac{\max(n_1-1,0)^2}{m},\,u_{\Delta}=\frac{\max(n_2-2,0)^3}{m^3},\,u_{\Delta+1}=\frac{n_1\max(n_2-1,0)^2n_3}{m^4},\\
    u_{K_4-e}=\frac{\max(n_2-1,0)^2\max(n_3-1,0)^2}{m^5},\,u_{K_4}= \frac{\max(n_3-3,0)^4}{m^6},\,u_{K_5^+}=\frac{n}{m^6}.
\end{gather*}

\begin{proof}[Proof of Theorem \ref{main}]

The theorem  is trivially true unless $n_1 < 10^{-6}\sqrt{m}$ and $n_2\leq 10^{-6}m$ so we assume this is the case. We prove the theorem by exploring out of every vertex $v\in[n]$.  By Theorem \ref{thekeyresult},  the probability that the component containing $v$ has at least  $2 (\log m)^4$ edges but $G(\cD)$ is not connected, is $o(m^{\log\log{m}})$. Therefore it suffices to bound, for every $v\in[n]$,  the probability that the exploration from $v$ dies out before reaching a step $i$ with  $2|E(H_i)|+X_i \ge 4(\log m)^4$. We refer to the period before $2|E(H_i)|+X_i \ge 4(\log m)^4$ as the early stage and note that it consists of at most $2(\log  m)^4$ iterations. For any vertex $v$ which dies out in the early stage, we define $i^*$ as the minimum  of the $i$ for which  $X_i=0$, so $i^* \le 2(\log m)^4$.  Our  precise results on the early stage of the exploration   depend on the shape of the degree distribution and on the degree of $v$.\par
We first compute the probability that a vertex $v$ of degree $1$ is the endpoint of a path component of length $i^*+1$ by considering the exploration from it. For this to occur, we must have $d_i(v_i)=X_{i-1}=1=K_i$ for all $1\le i < i^*$ (if $i^*=1$ then there are no such $i$) and $d_{i^*}(v_{i^*})=X_{i^*-1}=J_{i^*}=1$. Applying Lemma \ref{lem:boundary-grows2}, we obtain that the probability this occurs is at most 
\begin{equation*}
    \frac{10\max(n_1-1,0)}{m}\left(\frac{20n_2}{m}\right)^{i^*}.
\end{equation*}
Hence, summing over all vertices of degree $1$, the probability a path component is created  in the early stages is at most
\begin{equation*}
    n_1\sum_{0 \le i^*\le 2(\log m)^2}\frac{10\max(n_1-1,0)10^{-4i^*}}{m}\leq \frac{20\max(n_1-1,0)^2}{(1-10^{-4})m}<\frac{21max(n_1-1,0)^2}{m}.
\end{equation*}\par
We next  compute the probability that a vertex $v$ of degree $2$ lies  on a cycle of length $ i^*+1$ by considering the exploration from it. If this occurs then $i^* \ge 2$, $X_0=d_1(v_1)=K_1=2$, $d_i(v_i)=K_i=1$ and $X_i=2$ for all $2\leq i<i^*$ (if $i^*=2$ then there are no such $i$), and $d_{i^*}(v_{i^*})=L_{i^*}=1$ while $X_{i^*-1}=2$.
Applying Lemma \ref{lem:boundary-grows2}, we obtain that the probability this occurs is no more than
\begin{equation*}
    \left(\frac{30\max(n_2-2,0)}{m}\right)^2\left(\frac{30\max(n_2-2,0)}{m}\right)^{i^*-2}\frac{4}{m}\leq\frac{3600\max(n_2-2,0)^2(10^{-4i^*+8})}{m^3}.
\end{equation*}
So the probability there is a vertex in  such a cycle component is at most
\begin{equation*}
    n_2\sum_{j\geq0}\frac{3600\max(n_2-2)^2(10^{-4j})}{m^3}\leq\frac{10800\max(n_2-2,0)^3}{(1-10^{-4})m^3}<\frac{11000\max(n_2-2,0)^3}{m^3}.
\end{equation*}\par

Finally, we consider explorations from $v$  of degree at least $3$ which die out in the early stages.  
We note that for any iteration $i$, if $L_i\geq 2$ then $X_i\geq L_i (d_i(v_i)-1)>0$, and if $L_i=1$ then $X_i\geq d_i(v_i)-1>0$ unless $d_i(v_i)=1$. So either $L_{i^*}=0$ (and $J_{i^*}=X_{i^*-1}=d_{i^*}(v_{i^*})$), or $L_{i^*}=d_{i^*}(v_{i^*})=1$ and $X_{i^*-1}=2$ (and $J_{i^*}=K_{i^*}=0$).\par

It is straightforward to show that if, for some  $i$, $X_i>4000$  but the exploration from $v$ dies out in the early stages, then either
\begin{enumerate}[(i)]
    \item there is a $j>i$ for which $X_{j-1}>500$  and $X_j<X_{j-1}-12\sqrt{X_{j-1}}$, or
    \item there are at least $13$  $j>i$  for which $X_j<X_{j-1}$, or
    \item there are at least 7  $j>i$ for which $X_j<X_{j-1}-1$. 
\end{enumerate}
Furthermore if $X_i<X_{i-1}$ at most $12$ times, no $X_i$ exceeds $4000$, and (i) does not hold,  then there are at most 
10000 $i$ for which $X_i>X_{i-1}$. 
So to deal with explorations from vertices of degree at least  three  it is enough to show that 
\begin{enumerate}[(A)]
\item For every such vertex $v$, the probability that $X_i<X_{i-1}$  for more than $12$ iterations $i \le i^*$  or $X_{i}<X_{i-1}-1$  for more than $6$ iterations $i\le i^*$ is $O(m^{-6})$ (cases (ii) (iii));
\item For every such vertex $v$, the probability that for some $i<i^*$ with $X_i>500$, $X_{i-1}<X_i -12\sqrt{X_{i-1}}$
is $O(m^{-6})$ (case (i)); and
\item The probability that there is some component such that when exploring from  every vertex $v$ in the component, no $X_i$ exceeds $4000$ and there are at most $12$ iterations $i$  for which $X_i<X_{i-1}$ and at most $6$ iterations $i$ for which $X_i<X_{i-1}-1$, is 
\[
 O\left(u_{edge}+u_{\Delta}+u_{\Delta+1}+u_{K_4-e}+u_{K_4}+u_{K_5^+}\right).
\]

\end{enumerate}
In order that $X_i\leq X_{i-1}-1$, we must either have $J_i \ge 1$ or $L_i \ge 1$.
Applying Lemma~\ref{lem:boundary-grows2}, we see that the probability that this occurs during an iteration in the early stage conditioned on the exploration to that point is $O(\frac{\max(n_1,(\log m)^8)}{m})$ which is $O(m^{-1/2})$. In the same vein, in order that $X_i\leq X_{i-1}-2$, we must either
have $J_i \ge 2$ or $L_i \ge 1$. Applying Lemma \ref{lem:boundary-grows2}, we see that the probability that this occurs in the early stage is 
$O(\frac{(\log m)^8}{m})$. So, the probability that in the early stage there are more than twelve  iterations in which $X_i$ 
decreases or more than six  in which $X_i$ decreases by at least $2$ is $o(m^{-6})$ and we are done with (A).

Next, Corollary \ref{iterationanalysiscor} implies that the probability that  for  some iteration $i$ in the early stage,   $X_i>500$ and $X_i<X_{i-1}-12 \sqrt{X_{i-1}}$ is $o(m^{-6})$ and we are done with (B), 

We now turn to proving (C). 
For any $0\leq i^*\leq \lfloor 4(\log m)^4 \rfloor$ and disjoint subsets $S^-,S^+\subseteq S=\{1,2,\dots,\lfloor 4(\log m)^4 \rfloor\}$ with $|S^-|\le 12$ and $|S^+| \le 10000$, we can compute the probability that 
$X_{i^*}=0$ but $0<X_i<4000$ for all $1\leq i<i^*$, $X_i-X_{i-1}>0$ only for $i \in S^+$ and $X_i-X_{i-1}<0$ only for $i \in S^-$.
For each of the remaining $i< i^*$  in which $X_i=X_{i-1}$ either $K_i=d_i(v_i)$ or $J_i+L_i>0$.
So,  applying Lemma \ref{lem:boundary-grows2}, the probability that $X_{i-1}=X_i$  is at most
$\max(m^{-1/3},\frac{20n_2}{m}) \le 10^{-4}$,  while the number of choices for $S^+$ and  $S^-$   and $L_i,J_i,K_i$
for each $i \in S^- \cup\{1,\dots, 300\}$  is  less than $(i^*)^{10012} (4000)^{936} $.
So the probability that such an $S^-$ and $S^+$ exist is of no larger order than the maximum, over all such choices of $\{L_i,J_i,K_i\}_{i\in S^-\cup\{1,\dots,300\}}$, of the product of the conditional 
probabilities given $H_{i-1}$ that  $J_i,K_i,L_i$ take the chosen values for each $i$  in  $\{1,\dots,300\} \cup S^-$.

We  note that if $|S^-|=1$ and the events of the preceding paragraph occur, then $X_{i^*-1} -X_{i^*} \ge X_0-0 \ge 3$ so we must have 
$L_{i^*}=0$ and $J_{i^*}=d_{i^*}(v_{i^*}) \ge 3$.\par
If $n_1\geq 2$ then $u_{edge}=\frac{\max(n_1-1,0)^2}{m} \ge \frac{n_1^2}{2m}$, and Lemma \ref{lem:boundary-grows2} then implies  that the probability that  $|S^-|\geq2$  or $J_{i^*} \ge 3$ is $O(\frac{n^2_1}{m^2})$; so (C) holds in this case.

We therefore only need to prove (C) when  $n_1 \le 1$ (recalling the definition of $\delta^*$ from \eqref{eq:delta-star}, we note that this implies $\delta^*+1 \ge 2$ which will be relevant when we apply Lemma \ref{lem:boundary-grows3}). Applying Lemma \ref{lem:boundary-grows2} we see that this implies that the probability $X_i<X_{i-1}$ is $O(\frac{1}{m} )$ so (C) holds for explorations for which $|S^-|>5$. Therefore we need only consider explorations with $|S^-| \le 5$.  Furthermore,  either  $L_{i^*}=1$ and $J_{i^*}=0$ or $J_{i^*}=d_{i^*}(v_{i^*})=1$  and $L_{i^*}=0$. In either case,  $X_{i^*-1} \le 2$  and  so $|S^-| \ge 2$. Hence, if $n_2=\Omega(n)$ then we are done by considering $u_\Delta$, so we can assume this is not the case and so 
$\delta^*+1 \ge 3$. Furthermore, we are done if $n_2>2$ and $|S^-|>3$ so we can assume this is not the case. 

We now let $i_1>i_2\dots>i_{|S^-|}$ be  the elements of $S^-$, so $i_1=i^*$. If $X_{i_2-1} >6$ then $2L_{i_2}+J_{i_2} \ge 5$.
Thus, we have  $d_{i_2}(v_{i_2}) \ge 3$  and since $J_{i_2} \le n_1 \le 1$, $L_{i_2} \ge 2$, but since $X_{i_2} \le X_{i^*-1}
\le 2$, this contradicts our choice of $v_{i_2}$. So, $X_{i_2-1} \le 6$.  A similar proof then shows $X_{i_3{-1}} \le 12$.
Continuing in the same vein, we obtain that $X_{i_4-1} \le 20$ and $X_{i_5{-1}} \le 30$. Thus we need only consider explorations where every $X_i \le 30$ and hence every vertex explored has degree at most $31$. 

Now, by Lemma \ref{lem:boundary-grows2},  since $ n_1\leq 1$, 
the (conditional) probability that $0<X_i \le  X_{i-1}$  for $i \le 300$ is $O(\frac{\max(n_2,1)}{m})$.  It follows that the probability that there are 5  such $i$ and $0=X_{i^*}<X_{i^*-1}$
is ${O((\frac{\max(n_2,1)}{m})^6)= } O( \frac{\max(u_{\Delta},u_{K_5^+})}{n})$ and (C) holds in this case. But if $i^*>500$ then either there are 5  such $i$ 
or some $X_i$ exceeds $30$, and (C) holds in both cases. So, we need only consider explorations for which $i^* \le 500$.
Hence the  number of choices for the degrees of the vertices added to the graph and the $L_i$ edges added within the graph at  every iteration is a fixed constant and we need only bound the maximum, over  any such choice, of the products of the conditional probability that in each iteration $i$,
the degrees of the vertices added to the tree and back edges are as claimed.We refer to such a choice of the number of vertices of every degree up to $31$ and back edges for each iteration up to $500$ as the \emph{specified choice}. 

We note that since $d_i(v_i)\leq X_{i-1}<30$ in every iteration, we have $ 2|E(H_{i-1})|+X_i <10^6$ 
in every iteration.

\begin{claim}
    The conditional probability, given $H_{i-1}$, of the specified choice for a fixed iteration $1\leq i\leq 500$ occurring, is $O\left(m^{- \left\lceil \frac{\min(X_{i-1},11)-X_i}{2} \right\rceil}\right)$. 
\end{claim}

\begin{proof}
Recall that $J_i+2L_i\geq X_{i-1}-X_i$. So, applying   Lemma \ref{lem:boundary-grows2} proves the claim unless $L_i\geq 2$. Furthermore, since $\delta^*+1\geq 3$, applying Lemma \ref{lem:boundary-grows3} proves the claim unless $d_i(v_i) \geq 4$. If $L_i>2$  this implies $X_i \geq L_i(d_i(v_i)-1)\geq 9$ and again we are done by applying  Lemma \ref{lem:boundary-grows2}. 

So, $L_i=2$, $X_{i-1} \ge (L_i+1)d_i(v_i) \ge 12$ and $11-X_{i} \le 4$. So we are done by Lemma \ref{lem:boundary-grows2} unless $J_1=0$ and $X_i \le 8$. 
This implies $d_i(v_i)=4, L_i=2$, $K_i=2$ and $X_i=8$. So, applying   Lemma \ref{lem:boundary-grows2} proves the claim unless $n_2 \geq 3$. 
This implies $|S^-| \le 3$ which  is impossible as $X_i >6$.
\end{proof}

\begin{cor}
\label{reallygoodcor}
    For any fixed iteration $1\leq i\leq 500$, the conditional probability, given $H_{i-1}$, of the specified choice for every iteration after $i$ occurring, is $O\left(m^{- \left\lceil \frac{\min(X_{i-1},11)}{2} \right\rceil}\right)$. 
\end{cor}
This corollary implies that the probability  that there is  any  exploration dying out
in the first $500$ iterations for 
which some $X_i\geq11$ is $O(\frac{n}{m^6})=O(u_{K_5^+})$, so we need only count choices for which every  $X_i \le 10$. 

\begin{claim}
    Claim (C) holds if we weaken its statement by requiring the component to contain a vertex of degree at least $4$.
\end{claim} 

\begin{proof}

We consider exploring from a vertex $v$  of degree  at least $4$. 

We first consider the case $n_2\geq 3$.
We know that  $X_1 \ge 2d(v)-2J_1-K_1$ or equivalently that $\frac{X_1}{2}-J_1\leq-d(v)+\frac{K_1}{2}$. Applying Lemma \ref{lem:boundary-grows3} we know the probability that the first iteration behaves as claimed is $O(\frac{\max(n_2-2,1)^{K_1}}{m^{J_1+K_1}})$.   So, applying Corollary \ref{reallygoodcor} and 
the fact that $K_1 \le d(v)$ we have that the probability the specified exploration occurs is 
\begin{align*}
    O(\max(n_2-2,1)^{K_1}m^{-\frac{X_1}{2}-J_1-K_1})&=O(\max(n_2-2,1)^{K_1}m^{-d(v)-\frac{K_1}{2}}).
\end{align*}
Since $d(v) \ge max(K_1,4)$, this is $O(\frac{u_\Delta}{m})$. 

It remains to consider the case $n_2 \le 2$, which implies   $\max(n_2-2,1)=O(1)$. Since  $n_1 \le 1$,  if some vertex  of degree at least $4$ has a neighbour of degree at least  $4$, then one of these two vertices has no neighbour of degree $1$. Thus a component containing a vertex of degree at least $4$ either contains a vertex of degree at least  $4$ adjacent to no vertex of degree at least $4$, or a vertex of degree  at least $4$ adjacent to no vertex of degree $1$.

We consider first the probability of a specified  exploration from  $v$ in  
which  $v$ has no neighbour of degree at least $4$, and hence has a neighbour of degree three.  We can assume $X_1 \le 10$ so  applying 
Lemma \ref{lem:boundary-grows3} with $\delta^*+1=3$ to 
the  first iteration  and  Corollary \ref{reallygoodcor} to the remainder of the exploration, we obtain  that  the probability of the specified exploration  is 
\begin{equation*}
    O(n_3^{d(v)-J_1-K_1}m^{-d(v)-\lceil \frac{X_1}{2}\rceil}).
\end{equation*}

 We  have that $X_1\geq 2d(v)-2J_1-K_1\geq 4$ deterministically since $J_1\leq 1$ and $K_1\leq 2$. So, if $n_3 \le 3$ 
the probability of the specified exploration is $O(\frac{u_{K_5^+}}{n})$ and we are done. If $X_1\ge 5$ the probability of the specified exploration is $O(\frac{u_{K_4}}{m})$ 
So $X_1=4$ and we must have $J_1=1$ and $K_2=2$. Thus, the probability of the specified exploration is $O(\frac{u_{\Delta+1}}{m})$.

We consider next the possibility of a specified exploration from $v$ in which $v$ has no neighbour of degree 1.
So,  letting $T_1$ be the number of neighbours of $v$ of degree $3$,we have  have $X_1 \ge 3d(v)-2K_1-T_1$  which implies $\lceil\frac{X_1}{2}\rceil+K_1+T_1\geq\lceil\frac{3}{2}d(v)+\frac{T_1}{2}\rceil$. Hence, applying Lemma \ref{lem:boundary-grows3} to the first iteration followed by Corollary \ref{reallygoodcor} to the rest of the exploration ,
 we have that the probability the specified exploration occurs is
 \begin{equation*}
 O(\max(n_3-1,1)^{T_1}m^{-K_1-T_1-\lceil\frac{X_1}{2}\rceil})=O(\max(n_3-1,1)^{T_1}m^{-\lceil\frac{3}{2}d(v)+\frac{T_1}{2}\rceil}).
 \end{equation*}
If $n_3 < 4$ or $T_1=0$ this is $O(\frac{u_{K_5^+}}{n})$. If $n_3 \ge 4$  and $0<T_1<5$ this is $O(\frac{u_{K_4}}{m})$.

 If $n_3 \ge 4$ and $T_1\geq 5$  setting $a=d(v)-4$  and $a'=T_1-4$, we have $a \geq a'$ which implies that $\frac{a+a'}{2}\geq a'$ and
\begin{equation*}
    -\frac{3}{2}d(v)-\frac{T_1}{2}\leq-4-\frac{a+4+a'+4}{2}=-8-\frac{a+a'}{2},
\end{equation*}
so the probability of the specified exploration is
\begin{equation*}
O(n_3^{4+a'}m^{-8-\frac{a+a'}{2}})=O(\max(n_3-3,0)^4m^{-8})
\end{equation*}
So, the probability there is such an exploration from $v$ is $O(\frac{u_{K_4}}{m})$.

\end{proof}

So, to complete the proof of Claim (C), we need only show: 

\begin{claim}
Claim (C) holds if we weaken its statement by requiring the component to contain no vertex of degree at least $4$.
\end{claim}

\begin{proof}
Let $H$ denote the component subgraph corresponding to the specified choices. Recall $H$ contains at least one vertex of degree $3$, at most one vertex of degree 1 and only vertices of degree at most $3$. Applying Lemma \ref{lem:boundary-grows3} with $\delta^*+1 \ge 3$, we see that the probability the specified exploration occurs somewhere is at most $m^{-|E(H)|}\prod_{v \in V(H)} n_{d(v)}$. We note that this is at most $m^{-\frac{1}{2}\sum_{v\in S}d(v)}\prod_{v\in S}n_{d(v)}$ for any $S\subseteq V(H)$ such that $V(H)\setminus S$ contains no degree $1$ vertices. Recall that $n_1\leq 1$.\par 
So, if $H$ contains:
\begin{enumerate}
    \item At least four vertices of degree $3$, then since the number of vertices of odd degree is even, this probability is $O(u_{K_4})$;
    \item One vertex of degree $1$ and at least two vertices of degree $2$, then this probability is $O(u_{\Delta+1})$;
    \item No vertices of degree $1$ and at least two vertices of degree $2$, then since the number of vertices of odd degree is even, $H$ contains at least two vertices of degree $3$, so this probability is $O(u_{K_4-e})$;
    \item One vertex of degree $1$ and exactly one vertex of degree $2$, then $H$ contains two adjacent vertices of degree $3$, so since the number of vertices of odd degree is even $H$ contains at least three vertices of degree $3$, so this is
\begin{equation*}
    O(m^{-6}n_2\max(n_3-2,0)^3)=O(u_{\Delta+1}+u_{K_4} +u_{K_5^+});
\end{equation*}
\item No vertices of degree $1$ and exactly one vertex of degree $2$, then $H$ contains at least three vertices of degree $3$, so it contains at least four by parity;
\item One vertex of degree $1$ and no vertices of degree $2$, then $H$ contains at least three vertices of degree $3$, but $(1,3,3,3)$ is infeasible, so $H$ contains at least four vertices of degree $3$;
\item No vertices of degree $1$ or $2$, then $H$ contains at least four vertices of degree $3$.
\end{enumerate}
\end{proof}
This completes the proof of Theorem \ref{main}.
\end{proof}

\section{A Remark on Optimality}\label{sec:concl}

We show now that our bounds are tight up to a constant factor if  $D^* \le \frac{m}{3}$
and any of the following hold: $n_1>1$, $n_2>2$, or $n_3=\Omega(n^{1/4})$. 
We need the following which is the part of  Theorem 1 of \cite{gao2025subgraphprobabilityrandomgraphs} restricted to the case $H_2=\emptyset$, which lower-bounds the probability of the existence of 
an edge (they use $J({\cD})$  for what we call $D^*$ and $M=2m$, we use $ij$ in place of their $uv$,  we have dropped any reference to $H_2$ and replaced  $H_1$ by $H$, they use $H^+$ for the event that $H \subseteq G$ and $\Delta$  for $d(n)$). 

\begin{thm}
Let $H$ be a graph on $[n]$  such that $H$ is a possible subgraph on a graph with degree sequence $\cD$, and   $uv\not\in  E(H)$. Then,
\begin{multline*}
    \Cprob{ij \in E(G)}{H \subseteq G}\\
    \ge \left(1-\frac{2D^*+6d(n)}{2m-2|E(H)|}\right)\frac{(d(i)-d_H(i))(d(j)-d_H(j))}{2m-2|E(H)|+(d(i)-d_H(i))(d(j)-d_H(j))}.
\end{multline*}
\end{thm}
Applying this theorem yields the following bound. 

\begin{cor}
  If $D^*\le \frac{m}{3}$, then  conditioned on the existence of  any subgraph $H$ of $G$ with at most 8 vertices, each of which has degree at most three in $G$, we have for any two vertices $i$ and $j$ such that $ ij \not \in H$, $10>d(i)>d_H(i)$, and $10>d(j)>d_H(j)$, the conditional probability that $ij \in E(G)$ given that $H \subseteq G$ is  $\Omega(\frac{1}{m})$.   
\end{cor}

\begin{proof}
 We have $D^* \ge 2d(n)$ so $2D^*+6d(n)  \le \frac{5m}{3}$. Also, $|E(H)| \le 24$ so for large $m$, $\frac{11m}{6}<2m-|E(H)|<2m$.
     So, $\frac{(d(i)-d_H(i))(d(j)-d_H(j))}{2m-|E(H)|}=\Theta(\frac{1}{m})$
     and  $\left(1-\frac{2D^*+6d(n)}{2m-2|E(H)|}\right)=\Omega(1)$.   
\end{proof}

It follows that under the assumption $D^*\leq\frac{m}{3}$, the expected number of edge components  is $\Omega(u_{edge})$,
 the expected number of triangle components is $\Omega(u_{\Delta})$,the expected 
 number of components consisting of a degree 1 vertex attached to a triangle is $\Omega(u_{\Delta+1})$,
 the expected number of $K_4-e$ components is $\Omega(u_{K_4-e})$ and the 
 expected number of $K_4$ components is $\Omega(u_{K_4})$. 

For each $J\in\{edge,\Delta,\Delta+1,K_4-e,K_4\}$, one can then show the expected number of pairs of  components inducing $J$ is  $O(u_J)^2$ 
using Lemma \ref{lem:boundary-grows3}.  It then follows by Chebyshev's inequality that the probability there is a  component  inducing $J$ is $\Omega(u_J)$. 
This shows our bound is tight unless no $u_j$ is $\Omega(n/m^6)$, which is not the case if $n_1>2$, $n_2>3$,
or $n_4=\Omega(n^{1/4})$.

\section{\bf Acknowledgements}
This work began while the first two authors were in residence at the Simons--Laufer Mathematical Sciences Institute (SLMath) during the Spring 2025 semester, supported by NSF grant DMS-1928930, and the third author was visiting SLMath. The first author acknowledges support from an NSERC Discovery Grant and from the Canada Research Chairs program.
%%%%%%%%%%%%%%%%%%%%%%%%%%%%%%%%%%%%%%%%%%%%%%%%%%%%%
%\small
%\addtocontents{toc}{\SkipTocEntry} %TO SUPPRESS LIST OF NOTATION FROM TABLE OF CONTENTS
%\printnomenclature[3.1cm]
%\normalsize
%%%%%%%%%%%%%%%%%%%%%%%%%%%%%%%%%%%%%%%%%%%%%%%%%%%%

\small 
\bibliographystyle{plain}
\bibliography{ref}{}

\end{document}